\documentclass[11pt]{amsart}

\usepackage{amsmath}
\usepackage{amsxtra}
\usepackage{amscd}
\usepackage{amsthm}
\usepackage{amsfonts}
\usepackage{amssymb}
\usepackage{eucal}
\usepackage{xcolor}
\usepackage{url}
\usepackage[all]{xy}

\usepackage{hyperref}

\newtheorem{theorem}{Theorem}[section]
\newtheorem{cor}[theorem]{Corollary}
\newtheorem{lem}[theorem]{Lemma}
\newtheorem{prop}[theorem]{Proposition}

\newtheorem{thm}[theorem]{Theorem}
\newtheorem{conj}[theorem]{Conjecture}
\newtheorem{example}[theorem]{Example}

\theoremstyle{definition}
\newtheorem{defn}[theorem]{Definition}

\theoremstyle{remark}
\newtheorem{rem}{Remark}[section]

\numberwithin{equation}{section}
\begin{document}

\newcommand{\thmref}[1]{Theorem~\ref{#1}}
\newcommand{\secref}[1]{Sect.~\ref{#1}}
\newcommand{\lemref}[1]{Lemma~\ref{#1}}
\newcommand{\propref}[1]{Proposition~\ref{#1}}
\newcommand{\corref}[1]{Corollary~\ref{#1}}
\newcommand{\remref}[1]{Remark~\ref{#1}}
\newcommand{\conjref}[1]{Conjecture~\ref{#1}}
\newcommand{\nc}{\newcommand}
\nc{\on}{\operatorname}
\nc{\ch}{\mbox{ch}}
\nc{\Z}{{\mathbb Z}}
\nc{\C}{{\mathbb C}}
\nc{\pone}{{\mathbb C}{\mathbb P}^1}
\nc{\pa}{\partial}
\nc{\F}{{\mathcal F}}
\nc{\arr}{\rightarrow}
\nc{\larr}{\longrightarrow}
\nc{\al}{\alpha}
\nc{\ri}{\rangle}
\nc{\lef}{\langle}
\nc{\W}{{\mathcal W}}
\nc{\la}{\lambda}
\nc{\ep}{\epsilon}
\nc{\su}{\widehat{{\mathfrak sl}}_2}
\nc{\sw}{{\mathfrak s}{\mathfrak l}}
\nc{\g}{{\mathfrak g}}
\nc{\h}{{\mathfrak h}}
\nc{\n}{{\mathfrak n}}
\nc{\N}{\widehat{\n}}
\nc{\G}{\widehat{\g}}
\nc{\ghat}{\widehat{\g}}
\nc{\De}{\Delta_+}
\nc{\gt}{\widetilde{\g}}
\nc{\Ga}{\Gamma}
\nc{\ga}{\gamma}
\nc{\one}{{\mathbf 1}}
\nc{\z}{{\mathfrak Z}}
\nc{\zz}{{\mathcal Z}}
\nc{\Hh}{{\mathcal H}_\beta}
\nc{\qp}{q^{\frac{k}{2}}}
\nc{\qm}{q^{-\frac{k}{2}}}
\nc{\La}{\Lambda}
\nc{\wt}{\widetilde}
\nc{\qn}{\frac{[m]_q^2}{[2m]_q}}
\nc{\cri}{_{\on{cr}}}
\nc{\kk}{h^\vee}
\nc{\sun}{\widehat{\sw}_N}
\nc{\hh}{{\mathbf H}_{q,t}(\g)}
\nc{\HH}{{\mathcal H}_{q,t}}
\nc{\ca}{\wt{{\mathcal A}}_{h,k}(\sw_2)}
\nc{\si}{\sigma}
\nc{\gl}{\widehat{{\mathfrak g}{\mathfrak l}}_2}
\nc{\el}{\ell}
\nc{\s}{T}
\nc{\bi}{\bibitem}
\nc{\om}{\omega}
\nc{\WW}{\W_\beta}
\nc{\scr}{{\mathbf S}}
\nc{\ab}{{\mathbf a}}
\nc{\rr}{d}
\nc{\ol}{\overline}
\nc{\con}{qt^{-1} + q^{-1}t}
\nc{\den}{q^{\el-1} t^{-\el+1}+ q^{-\el+1} t^{\el-1}}
\nc{\ds}{\displaystyle}
\nc{\B}{B}
\nc{\A}{A^{(2)}_{2\el}}
\nc{\GG}{{\mathcal G}}
\nc{\UU}{{\mathcal U}}
\nc{\MM}{{\mathcal M}}
\nc{\CC}{{\mathcal C}}
\nc{\GL}{^L\G}
\nc{\dzz}{\frac{dz}{z}}
\nc{\Res}{\on{Res}}

\nc{\mc}{\mathcal}
\nc{\wh}{\widehat}
\nc{\bp}{{\mathbf P}}
\nc{\uqg}{U_q \G}

\newcommand{\U}{U}

\newcommand{\ZZ}{\mathbb{Z}}

\newcommand{\Glie}{\mathfrak{g}}
\newcommand{\Yim}{\mathcal{Y}}

\newcommand{\Hlie}{\mathfrak{h}}

\nc{\chal}{\check\al}
\nc{\chla}{\check\la}
\nc{\crho}{\check\rho}
\nc{\cmu}{\check\mu}
\nc{\mb}{\mathbf}

\title{Folded quantum integrable models and deformed ${\mc W}$-algebras}

\author{Edward Frenkel}

\address{Department of Mathematics, University of California,
  Berkeley, CA 94720, USA}

\author{David Hernandez}

\address{Universit\'e de Paris and Sorbonne Universit\'e, CNRS,
  IMJ-PRG, IUF, F-75006, Paris, France}

\author{Nicolai Reshetikhin}

\address{YMSC, Tsinghua University, Beijing, China and Department of
  Mathematics, University of California, Berkeley, CA 94720, USA}

\dedicatory{In memory of Hugues Charnallet}

\begin{abstract}
  We propose a novel quantum integrable model for every non-simply
  laced simple Lie algebra $\g$, which we call the folded integrable
  model. Its spectra correspond to solutions of the Bethe Ansatz
  equations obtained by folding the Bethe Ansatz equations of the
  standard integrable model associated to the quantum affine algebra
  $U_q(\wh{\g'})$ of the simply-laced Lie algebra $\g'$ corresponding
  to $\g$. Our construction is motivated by the analysis of the second
  classical limit of the deformed ${\mc W}$-algebra of $\g$, which we
  interpret as a ``folding'' of the Grothendieck ring of
  finite-dimensional representations of $U_q(\wh{\g'})$. We
  conjecture, and verify in a number of cases, that the spaces of
  states of the folded integrable model can be identified with
  finite-dimensional representations of $U_q({}^L\widehat{\g})$, where
  $^L\widehat{\g}$ is the (twisted) affine Kac--Moody algebra
  Langlands dual to $\ghat$. We discuss the analogous structures in
  the Gaudin model which appears in the limit $q \to 1$. Finally, we
  describe a conjectural construction of the simple $\g$-crystals in
  terms of the folded $q$-characters.
\end{abstract}

\maketitle

\tableofcontents

\section{Introduction}

\subsection{Integrable models} Exactly solvable quantum integrable
models have played a prominent role in mathematical physics ever since
the groundbreaking work of Hans Bethe \cite{Bethe} in which he
described the spectrum of the Hamiltonian of the XXX spin chain in
terms of the solutions of what we call today the {\em Bethe Ansatz
  equations}. The XXX spin chain naturally corresponds to the Yangian
of $\sw_2$, and it can be generalized to quantum spin chain models
corresponding to the Yangian of an arbitrary simple Lie
algebra $\g$ or the corresponding quantum affine algebra
$U_q(\ghat)$. In this paper, we focus on the latter model, which we
call {\em XXZ-type model} associated to $U_q(\ghat)$.

Using the universal $R$-matrix of $U_q(\ghat)$, one assigns a
commuting family of Hamiltonians $t_V(z), z \in \C$, of this model, to
every finite-dimensional representation $V$ of $U_q(\ghat)$, which is
called an {\em auxiliary space}. These Hamiltonians, which are called
the transfer-matrices, act on any finite-dimensional representation
$W$ of $U_q(\ghat)$, which is called a {\em space of states} (or
physical space) of the model. Moreover, the transfer-matrix
construction extends to a compatible family of ring homomorphisms
\begin{equation}    \label{hW}
h_W: \on{Rep} U_q(\ghat) \to \on{End}(W)[[z]]
\end{equation}
sending the class of $V$ in the Grothendieck ring $\on{Rep}
U_q(\ghat)$ of $U_q(\ghat)$ to the corresponding transfer-matrix
$t_V(z)$ acting on $W$ (viewed as a formal power series in $z$ with
coefficients in $\on{End}(W)$). We will assume throughout that
$q\in\mathbb{C}^\times$ is not a root of unity.

In a generic situation, the spectra of these quantum Hamiltonians are
expected to be in one-to-one correspondence with the solutions of the
generalized Bethe Ansatz equations. For the XXZ-type model associated
to $U_q(\ghat)$, where $\ghat$ is an arbitrary affine Kac--Moody
algebra (twisted or untwisted) they were proposed in
\cite{OW,RW,R:1987}. A pair of authors of the present paper
subsequently conjectured in \cite{FR:q} an explicit formula for the
eigenvalues of the transfer-matrices $t_V(z)$ corresponding to a given
solution of these generalized Bethe Ansatz equations. This formula is
written in terms of the $q$-character of $V$.

Another pair of authors of the present paper then extended the above
construction to a larger algebra of quantum Hamiltonians \cite{FH}.
Namely, the homomorphism $h_W$ extends to a homomorphism
\begin{equation}    \label{hW1}
h'_W: \on{Rep}' \to \on{End}(W)(u)[[z]],
\end{equation}
where $\on{Rep}'$ stands for the Grothendieck ring of the category
${\mathcal O}$ introduced in \cite{HJ} or its dual category ${\mathcal
  O}^*$. The corresponding transfer-matrices $t_V(z,u)$ are formal
power series in $z$ with coefficients depending on an element $u$ of
the Cartan subgroup $H$ of the simply-connected Lie group $G$
associated to $\g$. The category ${\mathcal O}^*$ is topologically
generated by the prefundamental representations $R_j^\pm(z), j \in I$,
where $I$ is the set of vertices of the Dynkin diagram of $\g$. Hence
the eigenvalues of the quantum Hamiltonians are determined by the
eigenvalues of the $Q$-{\em operators} $Q_j^\pm(z,u) =
t_{R_j^\pm}(z,u)$, the transfer-matrices associated in \cite{FH} to
the prefundamental representations $R_j^\pm(z)$.

It was proved in \cite{FH} that up to a universal factor depending on
the representation $W$, all eigenvalues of $Q_j^+(z,u)$ on $W$ (which
are {\em a priori} formal power series in $z$) are in fact {\em
  polynomials} in $z$. Moreover, it follows from \cite{FH,FH4} that
the roots of these polynomials are solutions of the Bethe Ansatz
equations (more precisely, the $u$-dependent version of the Bethe
Ansatz equations) under a certain non-degeneracy condition. (This
condition was subsequently dropped in \cite{FJMM}.) Finally, it was
shown in \cite{FH} that the eigenvalues of the transfer-matrix
$t_V(z,u)$ corresponding to a finite-dimensional representation $V$
can be expressed in terms of the eigenvalues of the $Q_j^+(z,u)$ and
the $q$-character of $V$, proving the conjecture of \cite{FR:q}. These
results provide a link between the spectra of the XXZ-type
model associated to $U_q(\ghat)$ and the solutions of the
corresponding Bethe Ansatz equations (BAE).

Explicitly, if the space of states $W$ is the tensor product of
irreducible finite-dimensional representations of $U_q(\ghat)$ with
the Drinfeld polynomials $P_{i,k}, i \in I, k=1,\ldots,N$, then the
BAE have the form:
\begin{equation}    \label{bae gen0}
\prod_{k=1}^N q_i^{\deg P_{i,k}}
  \frac{P_{i,k}(q_i^{-1}/w^{(i)}_r)}{P_{i,k}(q_i/w^{(i)}_r)} =
  - \prod_{s \neq r}
  \frac{w^{(i)}_r-w^{(i)}_sq_i^{-2}}{w^{(i)}_r-w^{(i)}_s q_i^{2}}
  \prod_{j \neq i} \prod_{s=1}^{m_j}
  \frac{w^{(i)}_r-w^{(j)}_sq^{-B_{ij}}}{w^{(i)}_r-w^{(j)}_s
  q^{B_{ij}}}.
\end{equation}
with one equation for each root $w^{(i)}_r, r = 1,\ldots,m_i$, of the
$i$th Baxter polynomial $Q^+_i(z)$, with $i$ ranging over the set $I$.
Here $(B_{ij})$ is the symmetrized Cartan matrix of $\g$: $B_{ij} =
d_i C_{ij}$, where $(C_{ij})$ is the Cartan matrix and the $d_i$ are
relatively prime integers, and we set $q_i = q^{d_i}$.

Note that a typical factor on the right hand side of \eqref{bae gen0}
looks like this:
\begin{equation}    \label{typical}
\frac{w^{(i)}_r-w^{(j)}_sq^{-B_{ij}}}{w^{(i)}_r-w^{(j)}_s
  q^{B_{ij}}}
\end{equation}

\subsection{Miura $q$-opers and folded Bethe Ansatz
  equations}    \label{qopers}

In a recent paper \cite{FKSZ}, certain geometric objects on $\pone$
called {\em Miura $(G,q)$-opers} were introduced. It was shown in
\cite{FKSZ} (see also the earlier work \cite{KSZ} in the case
$G=SL_n$) that there is a one-to-one correspondence between the set of
Miura $(G,q)$-oper satisfying a non-degeneracy condition and the set
of solutions of a system of equations which look very similar to the
BAE associated to $U_q(\ghat)$.\footnote{More precisely, this was
proved in \cite{FKSZ} for Miura-Pl\"ucker $(G,q)$-opers, but in the
subsequent work \cite{KZ} it was shown that this notion is equivalent
to the notion of Miura $(G,q)$-oper.}

In these equations a typical factor with $i \neq j$ reads
\begin{equation}    \label{typical1}
\frac{(w^{(i)}_r-w^{(j)}_sq)^{-C_{ji}}}{(w^{(i)}_r-w^{(j)}_s
  q^{-1})^{-C_{ji}}}, \qquad \on{if} \quad i \neq j,
\end{equation}
where $(C_{ij})$ is the Cartan matrix of $\g$.  If
$\g$ is simply-laced, formulas \eqref{typical} and \eqref{typical1}
coincide and so the equations are just the standard BAE associated to
$U_q(\ghat)$. Therefore, in this case one obtains a ``dual''
description of the spectrum in terms of Miura $(G,q)$-opers, giving
rise to what in \cite{FKSZ} was called the {\em $q$DE/IM
  correspondence}.

However, if $\g$ is not simply-laced, we obtain this way a different
system of equations. A Yangian version of these equations first
appeared in the work of Mukhin and Varchenko
\cite{MV1,MV2}. We note that for $\g=B_\ell$ similar
  equations were also obtained in \cite{DHKM} in the context of 3d
  quiver gauge theories.

In this paper, we will call these equations the {\em folded Bethe
  Ansatz equations} because they can be obtained by ``folding'' the
BAE for the simply-laced simple Lie algebra $\g'$ that gives rise to
$\g$ (i.e. $\g'$ is equipped with an automorphism of order 2 or 3
whose invariant Lie subalgebra is $\g$).

To explain this folding procedure and to illustrate the
  difference between the two types of BAE, consider the case of
  $\g=C_\ell$. Then $\g' = A_{2\ell-1}$. In this case $d_i=1$ for
  $i=1,\ldots,\ell-1$ and $d_\ell = 2$. Therefore, the only factors
  \eqref{typical} for $C_\ell$ with $i \neq j$ and the powers of $q$
  different from $\pm 1$ (which are the only powers of $q$ appearing
  in the factors with $i \neq j$ in the simply-laced cases) occur for
  $i=\ell, j=\ell-1$ or the other way around. The first of them is
\begin{equation}    \label{typical-1}
\frac{w^{(\ell)}_r-w^{(\ell-1)}_sq^2}{w^{(\ell)}_r-w^{(\ell-1)}_s
  q^{-2}} =
\frac{w^{(\ell)}_r-w^{(\ell-1)}_sq^2}{w^{(\ell)}_r-w^{(\ell-1)}_s}
\frac{w^{(\ell)}_r-w^{(\ell-1)}_s}{w^{(\ell)}_r-w^{(\ell-1)}_s
  q^{-2}}.
\end{equation}
On the other hand, since the entry $C_{\ell-1,\ell}$ of the Cartan
matrix of $\g=C_\ell$ is equal to $-2$, the corresponding factor
\eqref{typical1} reads
\begin{equation}    \label{typical-2}
\frac{(w^{(\ell)}_r-w^{(\ell-1)}_sq)^2}{(w^{(\ell)}_r-w^{(\ell-1)}_s
  q^{-1})^2}.
\end{equation}
It can be obtained by folding the expression of the form
\begin{equation}    \label{typical-3}
\frac{w^{(\ell)}_r-w^{(\ell-1)}_sq}{w^{(\ell)}_r-w^{(\ell-1)}_s
  q^{-1}} \frac{w^{(\ell)}_r-w^{(\ell+1)}_sq}{w^{(\ell)}_r-w^{(\ell+1)}_s
  q^{-1}} 
\end{equation}
appearing in the BAE of the simply-laced Lie algebra
$\g'=A_{2\ell-1}$. Indeed, the automorphism of the Dynkin diagram of
$A_{2\ell-1}$ preserves the $\ell$th vertex and exchanges the
$(\ell-1)$st and the $(\ell+1)$st vertices. If we accordingly identify
the variables $w^{(\ell-1)}_s$ and $w^{(\ell+1)}_s$ in the expression
\eqref{typical-3}, we obtain the factor \eqref{typical-2}. This is
what we mean by folding the BAE of a simply-laced Lie algebra $\g'$.

Note that the difference between \eqref{typical-2} and
\eqref{typical-1} (as its RHS shows) is the difference between
$f(w)^2$ and $f(wq)f(wq^{-1})$.

\begin{rem}
After the first version of this paper was posted on arXiv, Heng-Yu
Chen and Taro Kimura informed us about their paper \cite{CK}, in which
they considered two classical limits of the deformed ${\mc W}$-algebra
${\mc W}_{q,t}(\g)$ in the context of the corresponding 5D fractional
quiver gauge theory introduced in \cite{KP2} (where these classical
limits are interpreted as the two Nekrasov--Shatashvili limits). They
obtained a version of the folded Bethe Ansatz equations of the present
paper from the analysis of the partition function of this theory in
one of these limits. They did not consider the folded integrable
model, which is the main focus of the present paper, where these
equations naturally appear from certain subspaces of
finite-dimensional representations of quantum affine algebras (see
Section \ref{folded}).\qed
\end{rem}

\subsection{Folded integrable model}

According to \cite{FKSZ}, non-degenerate Miura $(G,q)$-opers
encode solutions of the folded BAE. But which integrable model do
these equations correspond to?

In this paper (Section \ref{folded}) we propose a conjectural answer
to this question. Namely, we conjecture the existence of what we will
call the {\em folded integrable model} for every non-simply laced
simple Lie algebra $\g$, whose spectra give rise to solutions of the
folded BAE (under a genericity condition). This folded integrable model
combines in a non-trivial way representations of the quantum affine
algebra $U_q(\wh{\g'})$, where $\g'$ is the corresponding simply-laced
Lie algebra (these appear as the {\em auxiliary spaces} of the folded
model) and representations of the quantum affine algebra
$U_q({}^L\ghat)$, where $^L\ghat$ is the twisted affine Kac--Moody
algebra which is {\em Langlands dual} to $\ghat$ (these appear as the
{\em spaces of states} of the folded model).

\begin{rem} Note that to the twisted quantum affine algebra
  $U_q({}^L\ghat)$ one can also associate an XXZ-type quantum
  integrable model. It is constructed in the same way as for the
  untwisted quantum affine algebras, and its spectra correspond to the
  solutions of the BAE that were proposed in \cite{RW,R:1987}. But
  this model is {\em different} from the folded model. Namely, the
  typical factors of the BAE of this model read
\begin{equation}    \label{typical2}
  \frac{(w^{(i)}_r)^{-C_{ji}}-(w^{(j)}_sq)^{-C_{ji}}}{(w^{(i)}_r)^{-C_{ji}}-
    (w^{(j)}_s q^{-1})^{-C_{ji}}}, \qquad \on{if} \quad i \neq j,
\end{equation}
so they differ from the factors \eqref{typical1}.

 For example, in the case of $\g=C_\ell$, instead of the
  factor \eqref{typical-2} we have
\begin{equation}    \label{typical-4}
\frac{(w^{(\ell)}_r)^2-(w^{(\ell-1)}_sq)^2}{(w^{(\ell)}_r)^2-w^{(\ell-1)}_s
  q^{-1})^2} = \frac{w^{(\ell)}_r-w^{(\ell-1)}_sq}{w^{(\ell)}_r-w^{(\ell-1)}_s
  q^{-1}} \frac{w^{(\ell)}_r+w^{(\ell-1)}_sq}{w^{(\ell)}_r+w^{(\ell-1)}_s
  q^{-1}}.
\end{equation}
Thus, it can be obtained by folding the expression \eqref{typical-3}
if we also multiply the spectral parameter by $-1$, i.e.  identify
$w^{(\ell-1)}_s$ and $-w^{(\ell+1)}_s$ in \eqref{typical-3} (rather
than $w^{(\ell-1)}_s$ and $w^{(\ell+1)}_s$). In other words, the
difference between \eqref{typical-2} and \eqref{typical-4} is the
difference between $f(w)^2$ and $f(w) f(-w)$. And similarly for other
Lie algebras corresponding to an automorphism of order $2$. In the
case of $\g=G_2$, when the automorphism has order 3, it's the
difference between $f(w)^3$ and $f(w)f(w\epsilon)f(w\epsilon^{-1})$,
where $\epsilon = e^{2\pi i/3}$. \qed
\end{rem}

\subsection{$QQ$-system}

There is an important intermediate object between the spectra of the
XXZ-type model associated to $U_q(\ghat)$ and the corresponding BAE
called the $QQ$-{\em system}. It was introduced in \cite{MRV,MRV2} in
the context of affine opers. In \cite{FH4}, it was shown that this
$QQ$-system naturally arises in the context of the homomorphisms
$h'_W$ (see the above formula \eqref{hW1}). Namely, in addition to the
set of prefundamental representations $R^+_j(z), j \in I$, discussed
above there is another set of representations, denoted by $X_j(z), j
\in I$, in the category ${\mc O}^*$, such that properly rescaled
classes of these two sets satisfy the $QQ$-system. In other words, if
we assign to $R^+_j(z), j \in I$, the above $Q$-operators $Q_j(z,u)$
and to $X_j(z), j \in I$, the transfer-matrix $\wt{Q}_j(z,u)$ (the
image of $X_j(z)$ under the homomorphism $h'_W$, where $W$ is a
finite-dimensional representation of $U_q(\ghat)$), then these
operators, properly rescaled, will satisfy the
$QQ$-system.\footnote{It was called $Q\wt{Q}$-system in \cite{FH4} but
here, for the sake of brevity, we follow the terminology of
\cite{FKSZ} and call it the $QQ$-system.} Note that for $\g=\sw_2$,
this is the quantum Wronskian relation introduced in \cite{BLZ3},
which provided the initial motivation for this line of research.

Thus, the $QQ$-system encodes universal relations between the classes
of the representations $R^+_j(z)$ and $X_j(z), j \in I$ in $\on{Rep}'$
which translate under the homomorphism \eqref{hW1} into relations
between the corresponding transfer-matrices, and hence their
eigenvalues, on any representation $W$.\footnote{The fact that the
same $QQ$-system arises both from the affine opers and the eigenvalues
of the transfer-matrices is a manifestation of the affine Langlands
duality proposed in \cite{FF:kdv} and further elucidated in
\cite{FH4}. However, we will not discuss this duality in the present
paper.}

As explained in \cite{MRV,MRV2,FH4}, the Bethe Ansatz equations
\eqref{bae gen0} follow directly from the $QQ$-system under a certain
non-degeneracy condition. In fact, from the point of view of the
preceding paragraph, the $QQ$-system is more fundamental to the
question of describing the spectra of the XXZ-type models than the
Bethe Ansatz equations.

Likewise, for the folded BAE: As shown in \cite{FKSZ}, these equations
are equivalent to the $QQ$-system proposed in \cite{FKSZ} (under a
non-degeneracy condition). We will call the latter system the {\em
  folded $QQ$-system} because it can be obtained by folding the
$QQ$-system associated to the simply-laced Lie algebra $\g'$. Thus,
this system appears as an intermediate object between Miura
$(G,q)$-opers and the folded BAE.

In this paper we show that this folded $QQ$-system also appears
naturally as a relation satisfied by the transfer-matrices of our
(conjectural) folded integrable model associated to a non-simply laced
Lie algebra $\g$. The folded BAE equations follow from the folded
$QQ$-system under a non-degeneracy condition.

\subsection{Deformed ${\mc W}$-algebras}

Valuable insights about the folded quantum integrable model can be
learned from the deformed ${\mc W}$-algebras introduced by two of the
authors of the present paper in \cite{FR:w}. Recall that this is a
two-parameter algebra ${\mc W}_{q,t}(\g)$ associated to a simple Lie
algebra $\g$. Recently, the algebra ${\mc W}_{q,t}(\g)$ found
  interesting applications in the study of four-dimensional
  supersymmetric gauge theories, see \cite{Ne,KP1,AFO,KP2,EP}.

The deformed ${\mc W}$-algebra has two classical limits, in which the
algebra becomes commutative, and equipped with a Poisson structure:
the first occurs when $t \to 1$ and the second when $q \to 1$.

The first limit, ${\mc W}_{q,1}(\g)$, is relatively well
understood. It is isomorphic to the center $Z_q(\ghat)$ of
$U_q(\ghat)$ at the critical level. The corresponding commutative
algebra of generating fields can be identified, via a version of
the transfer-matrix construction (see \cite{RST,FR:center}),
with the representation ring $\on{Rep} U_q(\ghat)$.\footnote{More
precisely, there is a homomorphism $\on{Rep} U_q(\ghat) \to
Z_q(\ghat)[[z^{\pm 1}]]$, so that every $V \in \on{Rep} U_q(\ghat)$
gives rise to a formal power series $T_V(z)$, and the Fourier
coefficients of these series topologically generate $Z_q(\ghat)$,
see \cite[Section 8.1]{FR:q}.}  Moreover, under this identification
the free field
realization of ${\mc W}_{q,1}(\g)$ becomes the $q$-character
homomorphism (this was the motivation behind the definition of the
$q$-characters in \cite{FR:q}).

If $\g$ is simply-laced, then the second classical limit, ${\mc
  W}_{1,t}(\g)$, coincides with the first one upon replacing $t$ with
$q$. But if $\g$ is not simply-laced, the second limit is
substantially different from
the first one. In \cite{FR:w}, ${\mc W}_{1,t}(\g)$
was linked to the $t$-deformed Drinfeld--Sokolov reduction of the loop
group associated to $G$ introduced in \cite{FRS,SS} and some
observations were made connecting elements of ${\mc W}_{1,t}(\g)$ to
the $q$-characters of $U_q(\ghat^\vee)$, where $\ghat^\vee$ is the
twisted affine algebra associated to $\g'$ and $\sigma$. But that's
pretty much all that has been known about the limit $q \to 1$ until
now.

In the present paper, we argue that it is this limit that is relevant
to the ``folded structures'' that we discuss here, including the
folded Bethe Ansatz equations and the folded integrable models. Thus,
we can learn a lot about these models by studying this limit. Its
hybrid nature, i.e. the fact that it mixes in a non-trivial way
quantum affine algebras $\wh{\g'}$ and $^L\ghat$, shows that ${\mc
  W}_{1,t}(\g)$ is a fascinating Poisson algebra that deserves further
investigation.

The deformed ${\mc W}$-algebra ${\mc W}_{q,t}(\g)$ creates a bridge
between the two classical limits, and hence between the XXZ-type
quantum integrable model associated to $U_q(\ghat)$ and the
corresponding folded quantum integrable model. However, the
non-commutative nature of ${\mc W}_{q,t}(\g)$ makes deriving practical
consequences of this bridge a daunting task. For this reason, in this
paper we replace ${\mc W}_{q,t}(\g)$ with its simplified commutative
version introduced by two of the authors in \cite{FH1} under the name
{\em interpolating $(q,t)$-characters}. Using a slight refinement of
these objects, we make our conjectures concerning the folded quantum
integrable models more precise. This also enables us to explicitly
verify our conjectures in a number of non-trivial cases (see Section
\ref{ex}).

\subsection{Connection to qKZ equations and quantum $q$-Langlands
  correspondence}

It is known that the critical level limit of the solutions of the qKZ
equations corresponding to $U_q(\ghat)$ give rise to eigenvectors of
the XXZ-type model associated to $U_q(\ghat)$. Thus, the qKZ system
provides a deformation of the latter model.\footnote{The XXZ-type
model is already quantum, but here by a deformation we mean a
non-commutative deformation of the commutative algebra of quantum
Hamiltonians of the XXZ-type model. Therefore, it is a kind of
``second quantization.''} This can also be seen from the fact that for
a large class of representations of $U_q(\ghat)$, the difference
operators of the qKZ system become in the critical level limit the
transfer-matrices of the XXZ-type model (see Proposition
\ref{qKZcrit}).

In \cite{AFO}, a {\em quantum $q$-Langlands correspondence} was
proposed. For a simply-laced simple Lie algebra $\g$, it sets up a
correspondence between solutions of the qKZ system associated to
$U_t(\ghat)$ and the deformed conformal blocks associated to ${\mc
  W}_{q,t}(\ghat)$. Here $q$ depends on the level of $U_t(\ghat)$ in
such a way that the limit $q \to 1$ corresponds to the critical level
limit. In this limit the quantum $q$-Langlands correspondence
essentially becomes the statement that the Hamiltonians of the
XXZ-type model associated to $U_t(\ghat)$ correspond to elements of
${\mc W}_{1,t}(\ghat)$. If $\g$ is simply-laced, then ${\mc
  W}_{1,t}(\ghat)$ indeed coincides with $\on{Rep} U_t(\ghat)$, so
this statement comes down to the existence of the homomorphisms $h_W$
given by formula \eqref{hW}.

However, if $\g$ is not simply-laced, the algebra ${\mc
  W}_{1,t}(\ghat)$ is no longer isomorphic to $\on{Rep} U_t(\ghat)$
and hence does {\em not} give rise to the Hamiltonians of the standard
XXZ-type model associated to $U_t(\ghat)$, which are the
transfer-matrices associated to finite-dimensional representations of
$U_t(\ghat)$. Rather, as we argue in this paper, it gives rise to the
Hamiltonians of the {\em folded} quantum integrable model associated
to $\g$. These Hamiltonians correspond to the transfer-matrices
associated to finite-dimensional representations of $U_t(\wh{\g'})$
rather than $U_t(\ghat)$.

This suggests that for non-simply laced $\g$ the quantum $q$-Langlands
correspondence might be more subtle. Namely, it follows from the
preceding paragraph that the system of $q$-difference equations
appearing on one side of this correspondence is {\em not} the usual
qKZ system associated to $U_t(\ghat)$. If this were the usual qKZ
system associated to $U_t(\ghat)$, then in the critical limit we
would recover the eigenvectors of the XXZ-type model associated to
$U_t(\ghat)$, but this would be inconsistent with the limit on the
other side of the correspondence which yields ${\mc W}_{1,t}(\g)$. As
we discussed above, the latter is not the algebra of Hamiltonians of
the XXZ-type model associated to $U_t(\ghat)$, but rather the algebra
of Hamiltonians of the {\em folded} quantum integrable model
introduced in the present paper.

What should replace the qKZ system associated to $U_t(\ghat)$ in the
quantum $q$-Langlands correspondence for non-simply laced $\g$? The
above discussion shows that this modified qKZ system should have the
property that the leading terms of its solutions in the critical level
limit are eigenvectors of the folded integrable system associated to
$\g$. As far as we know, the existence of such modified qKZ system is
an open question at the moment (naive ways to ``fold'' the qKZ system
associated to $\g'$ don't seem to work, see the Appendix of this
paper). But we expect that this question can be answered using the
geometric and $K$-theoretic methods of \cite{AFO}. Perhaps, these
equations can also be constructed purely algebraically. We hope to
return to this question elsewhere.

\subsection{The Gaudin limit}

To gain further insights, it is instructive to consider the limit in
which the second parameter, denoted by $t$ in the previous subsection,
also goes to $1$. In this limit, the XXZ-type model associated to
$U_t(\ghat)$ becomes the Gaudin model associated to $\g$
\cite{FFR,F:gaudin}; more precisely, its modification with a twist
parameter $\chi$, an element of the Cartan subalgebra of $\g$
\cite{Ryb,FFTL,FFRyb}. It turns out that in the limit $t \to 1$
the folded integrable model associated to $\g$ that we discuss in this
paper becomes the Gaudin model associated to $^L\g$. Thus, in the
Gaudin limit we do not find any new quantum integrable models. In
part, this is because in this limit the irreducible finite-dimensional
representations of $U_q(\ghat)$ decompose into a direct sum of
irreducible representations of the finite-dimensional Lie algebra
$\g$, so the affine Langlands duality $\ghat \to {}^L\ghat$ reduces to
the finite-dimensional Langlands duality $\g \to {}^L\g$, which was
discovered in \cite{FFR}.

However, even in this limit, as we will show in Section \ref{gaudin},
one can observe some intriguing effects related to folding. In
particular, using the results of \cite{FFRyb}, we will construct
embeddings of tensor products of irreducible representations of $^L\g$
into tensor products of the corresponding irreducible representations
of $\g'$ (see Theorem \ref{embqs}). In fact, it's a family of
embeddings depending on $\chi$ (which is assumed to be regular and
generic). It maps eigenvectors of the $^L\g$-Gaudin model with the
twist $\chi$ to eigenvectors of the corresponding $\g'$-Gaudin
model. Under a certain assumption (see Conjecture \ref{completeness}
in the case of a single irreducible representation) this embedding can
be constructed explicitly.

\subsection{Plan of the paper}
In Section \ref{notset} we fix our notation for the Lie algebras and
two-parameter Cartan matrices. In Section \ref{defwa} we recall the
definition of the deformed $\W$--algebra ${\mc W}_{q,t}(\g)$ from
\cite{FR:w}. We then consider its two classical limits. The first
limit, $t \to 1$, is relatively well understood; it can be identified
with the Grothendieck ring of finite-dimensional representations of
$U_q(\ghat)$ as we recall in Section \ref{tto1}. In Section \ref{qto1}
we obtain a description of the second limit, $q \to 1$ (which was much
less understood), analogous to the description of the $t \to 1$ limit
(see Propositions \ref{props-} and \ref{intmoins}). In Section
\ref{clasqcar} we recall the relation between the $t \to 1$ limit of
the deformed ${\mc W}$-algebra and the ring of $q$-characters of
finite-dimensional representations of $U_q(\ghat)$. We then relate the
$q \to 1$ limit to what we call the ring of {\em folded
  $t$-characters} of finite-dimensional representations of
$U_t(\wh{\mathfrak{g}'})$, where $\g'$ is the simply-laced Lie algebra
from which $\g$ is obtained as the Lie subalgebra fixed by an
automorphism (see Theorem \ref{unfold}). We also discuss the link
between this limit and the difference Drinfeld--Sokolov
reduction. Finally, in Section \ref{BAE} we introduce the {\em folded
  Bethe Ansatz equations}.

In Section \ref{folded} we describe the {\em folded quantum integrable
  model} in which the spectra of the Hamiltonians conjecturally
correspond to solutions of the folded Bethe Ansatz equations (see
Conjectures \ref{main conj bis} and \ref{part3}).  In Section
\ref{intchar} we recall the interpolating $(q,t)$-characters from
\cite{FH1}, which may be viewed as commutative algebra analogues of
elements of the non-commutative ${\mc W}$-algebra ${\mc
  W}_{q,t}(\g)$. We then construct a refined version of the
interpolating $(q,t)$-characters. They are elements of a ring
depending on the parameters $q$ and $t$, which is equipped with
$5$ interesting specialization homomorphisms to the rings of $q$- and
$t$-characters of various affine Kac--Moody algebras related to $\g$
(see Theorem \ref{refined}). In Section \ref{fundsec} we partially prove our
conjectures in the important case of $\sigma$-fundamental
representations (these are the irreducible finite-dimensional
representations of $U_q(\wh{\g'})$ with the $\sigma$-invariant highest
monomials of smallest possible degrees). In Section \ref{ex} we present a number
of explicit examples confirming our Conjectures \ref{main conj bis}
and \ref{part3}. In Section \ref{cryssec} we formulate a conjecture
linking the folded $t$-characters to
Kashiwara's extension of Nakajima's monomial model of crystals to
non-simply laced Lie algebras. In Section \ref{gaudin}, we consider
the Gaudin limit of the the folded quantum integrable models. In the
Appendix we discuss a possible construction of a folded version of the
qKZ equations for non-simply laced Lie algebras.

\subsection{Acknowledgments} We thank M. Aganagic, P. Koroteev,
A. Okounkov, and A. Zeitlin for useful discussions. E.F. and D.H. were
partially supported by a grant from the France-Berkeley Fund of UC
Berkeley. N.R. was partially supported by the NSF grant DMS-1902226.

\section{Notation and setup}\label{notset}

\subsection{Lie algebra}

Let $\g$ be a simple Lie algebra of rank $\el$ and $I =
\{ 1,\ldots,\ell \}$ the set of vertices of the Dynkin diagram of
  $\g$. Let $(\cdot,\cdot)$ be the invariant inner product on $\g$,
  normalized so that the square of the maximal root equals $2$. Let
  $\{ \al_1,\ldots,\al_\el \}$ and $\{ \om_1,\ldots,\om_\el \}$ be the
  sets of simple roots and of fundamental weights of $\g$,
  respectively. We have:
$$
(\al_i,\om_j) = \frac{(\al_i,\al_i)}{2} \delta_{i,j}.
$$
Let $d$ be the maximal number of edges connecting two vertices of the
Dynkin diagram of $\g$. Thus, $d=1$ for simply-laced $\g$, $d=2$
for $B_\el, C_\el, F_4$, and $d=3$ for $G_2$. We set $\epsilon =
e^{i\pi/d}$.

Set
$$
D = \on{diag}(\rr_1,\ldots,\rr_\el),
$$
where
\begin{equation}    \label{di}
\rr_i = d \frac{(\al_i,\al_i)}{2}.
\end{equation}
All $\rr_i$'s are integers, which are relatively prime with each
other. For simply-laced $\g$, $D$ is the identity matrix.

Now let $C = (C_{ij})_{1\leq i,j\leq \el}$ be the {\em Cartan matrix} of
$\g$. We have:
$$
C_{ij} = \frac{2(\al_i,\al_j)}{(\al_i,\al_i)}.
$$
Denote by $(I_{ij})_{1\leq i,j\leq \el}$ the {\em incidence matrix},
$$
I_{ij} = 2 \delta_{i,j} - C_{ij}.
$$
Let $\B = (\B_{ij})_{1\leq i,j\leq \el}$ be the following matrix:
$$
B = D C,
$$
i.e.,
$$
\B_{ij} = d (\al_i,\al_j).
$$

The weights $\rho$ and $\rho^\vee$ are defined by $(\rho^\vee,\al_i) =
1$, $d(\rho,\al_i) = \rr_i$ for any $1\leq i\leq \el$.

\subsection{Lie algebras involved}    \label{nomen}

We list here all Lie algebras involved in our study:

\begin{itemize}

\item $\g$ is a simple finite-dimensional Lie algebra.

\item $^L\g$ is its Langlands dual Lie algebra. For example, if
  $\g=B_\el$, then $^L\g=C_\el$.

\item $\ghat$ is the untwisted affine Kac--Moody algebra, which is the
  central extension of $\g[t,t^{-1}]$. For example, if $\g=B_\el$,
  then $\ghat = B_\el^{(1)}$.

\item $\wh{^L\g}$ is the untwisted affine Kac--Moody algebra, which is the
  central extension of $^L\g[t,t^{-1}]$. For example, if $\g=B_\el$,
  then $\wh{^L\g} = C_\el^{(1)}$.

\item $^L\ghat$ is the affine Kac--Moody algebra that is {\em affine}
  Langlands dual to $\ghat$. If $\g$ is simply-laced, then $^L\ghat =
  \ghat$. But if $\g$ is non-simply laced, then $^L\ghat$ is a {\em
    twisted} affine Kac--Moody algebra. Note that $^L\ghat$ contains
  $^L\g$ as the constant Lie subalgebra. For example, if $\g=B_\el$,
  then $^L\ghat = A_{2\el-1}^{(2)}$ (whose constant subalgebra is $^L\g
  = C_\el$); and if $\g=C_\el$, then $^L\ghat = D_{\el+1}^{(2)}$ (whose
  constant subalgebra is $^L\g = B_\ell$).

\item $\g'$ is the unique {\em simply-laced} Lie algebra equipped with
  an automorphism $\sigma$ of order $d$ such that the Lie subalgebra
  of $\sigma$-invariants in $\g'$ is $\g$
  (i.e. $\g=(\g')^\sigma$). For example, if $\g=B_\el$, then
  $\g'=D_{\el+1}$; and if $\g=C_\el$, then $\g'=A_{2\el-1}$.

\item $\ghat^\vee$ is $\ghat$, if $\g$ is simply-laced. If $\g$ is
  non-simply laced, then $\ghat^\vee$ is the twisted affine Kac--Moody
  algebra corresponding to $\g'$ and $\sigma$. Note that its constant
  Lie subalgebra is $\g$ itself. For example, $(B_\el^{(1)})^\vee =
  D_{\el+1}^{(2)}$ (its constant subalgebra is $B_\ell$), and
  $(C_\el^{(1)})^\vee = A_{2\el-1}^{(2)}$ (its constant subalgebra is
  $C_\el$). Note also that we have $\ghat^\vee = {}^L(\widehat{^L\g})$.

\end{itemize}

It might be better to denote $\ghat^\vee$ by $\g'{}^{(d)}$, but we 
will use below the notation $\ghat^\vee$ because it was used in
\cite{FR:w}.

Let us denote by $I'$ the set of vertices of the Dynkin diagram of
$\g'$. Then $\sigma$ acts on $I'$ and the quotient is in bijection
with the set $I$ of vertices of the Dynkin diagram of $I$.
Note that the automorphism $\sigma$ of the Dynkin diagram of $\g'$
acts on the objects labeled by the nodes of this diagram (such as
simple roots, fundamental weights, etc.).

\subsection{Two-parameter Cartan matrices}

We follow the notation of \cite{FR:w}, Sect. 2, except that we replace
$q$ by $q^{-1}$ (however, for $t=1$ this notation is consistent with
the notation of \cite{FR:q}).

Now let $q, t$ be non-zero complex numbers which are not roots of
unity. We will use the standard notation for $n\in\mathbb{Z}$
$$
[n]_q = \frac{q^n - q^{-n}}{q - q^{-1}}.
$$
Let
$$
q_i = q^{\rr_i}.
$$
We define $\el \times \el$ matrices $C(q,t)$, $D(q,t)$, and $\B(q,t)$ by the
formulas
\begin{align}    \label{qtc}
C_{ij}(q,t) &= (q_i t + q_i^{-1} t^{-1}) \delta_{i,j} - [I_{ij}]_q,
\\ D(q,t) &= \on{diag}([\rr_1]_q,\ldots,[\rr_\el]_q), \label{qtd} \\
\B(q,t) &= D(q,t) C(q,t). \notag
\end{align}
Thus,
\begin{equation}    \label{qts}
\B_{ij}(q,t) = [\rr_i]_q \left( (q^{\rr_i} t + q^{-\rr_i} t^{-1})
\delta_{i,j} - [I_{ij}]_q \right).
\end{equation}
It is easy to see that the matrix $\B(q,t)$ is symmetric. For simply-laced
$\g$,
$$
C_{ij}(q,t) = \B_{ij}(q,t) = (q t + q^{-1} t^{-1}) \delta_{i,j} - I_{ij}.
$$
We note that the determinants of these matrices are non-zero
polynomials in $q$ and $t$. Hence they are invertible over the field
of rational functions in $q$ and $t$.

Clearly, the limits of $C(q,t)$, $D(q,t)$, and $\B(q,t)$ as both $q
\arr 1$ and $t \arr 1$ coincide with $C$, $D$, and $\B$,
respectively.  We also have
$$
\B_{ij}(q,1) = [\B_{ij}]_q, \qquad C_{ij}(q,1) = (q_i  + q_i^{-1} )
\delta_{i,j} + [C_{ij}]_q \delta_{i\neq j},
$$
and
$$
\B_{ij}(1,t) = \rr_i ((t + t^{-1}) \delta_{ij} - I_{ij}).
$$

Let $C(q) = C(q,1)$. It is invertible over the field of rational
functions in $q$. We denote its inverse by $\wt{C}(q)$.

\section{Deformed $\W$--algebras and screening operators}\label{defwa}

In this section we recall the definition of the deformed $\W$--algebra
${\mc W}_{q,t}(\g)$ and related objects from \cite{FR:w}. We will then
look at the two classical limits $t \to 1$ and $q \to 1$, which are
defined as the intersections of the kernels of two sets of classical
screening operators. The $t \to 1$ limit was described in
\cite{FR:w,FR:q,fm} and is closely related to the Grothendieck ring
$\on{Rep} U_q(\ghat)$ and the corresponding $q$-characters (as we
recall in the next section). A new result of this section is the
analogous description of the $q \to 1$ limit (see Propositions
\ref{props-} and \ref{intmoins}).

\subsection{Heisenberg algebra $\HH(\g)$}

Let $\HH(\g)$ be the Heisenberg algebra with generators $a_i[n],
i=1,\ldots,\el; n \in \Z$, and relations
\begin{equation}    \label{a}
[a_i[n],a_j[m]] = \frac{1}{n} (q^n - q^{-n}) (t^n - t^{-n}) \B_{ij}(q^n,t^n)
\delta_{n,-m}
\end{equation}
where $1\leq i,j\leq \el;n,m\in\Z\setminus\{0\}$.

Here and in what follows, it is understood that the $0$th generator
commutes with all other generators: $[a_i[0],a_j[m]] = 0$, for all $m
\in \Z$.

The algebra $\HH(\g)$ becomes commutative in the limit $q \arr 1$ and
in the limit $t\arr 1$.

The generators $a_i[n]$ are ``root'' type generators of $\HH(\g)$. There is
a unique set of ``fundamental weight'' type generators, $y_i[n],
i=1,\ldots,\el; n \in \Z$, that satisfy:
\begin{equation}    \label{ay}
[a_i[n],y_j[m]] = \frac{1}{n} (q_i^{n} - q_i^{-n})(t^n - t^{-n})
\delta_{i,j} \delta_{n,-m}.
\end{equation}
They have the following commutation relations:
\begin{equation}    \label{y}
[y_i[n],y_j[m]] = \frac{1}{n} (q^{n} - q^{-n}) (t^n - t^{-n})
M_{ij}(q^n,t^n) \delta_{n,-m},
\end{equation}
where $(M_{ij}(q,t))_{1\leq i,j\leq \el}$ is the following matrix
\begin{align}    \label{tildeM}
M(q,t) &= D(q,t) C(q,t)^{-1} \\
&= D(q,t) B(q,t)^{-1} D(q,t).  \notag
\end{align}

We have
\begin{equation}    \label{expr}
a_i[n] = \sum_{j=1}^\el C_{ji}(q^n,t^n) y_j[n].
\end{equation}

We will use the colon notation for the standard normally ordered
product of elements of this algebra. Introduce the generating series
\begin{equation}    \label{Ai}
A_i(z) = t^{2(\rho^\vee,\al_i)} q^{2d (\rho,\al_i) + 2a_i[0]}
:\exp \left( \sum_{m\neq 0} a_i[m] z^{-m} \right):,
\end{equation}
\begin{equation}    \label{Yi}
Y_i(z) = t^{2(\rho^\vee,\om_i)} q^{2d (\rho,\om_i) + 2 y_i[0]}
:\exp \left( \sum_{m\neq 0} y_i[m] z^{-m} \right):.
\end{equation}
Recall that $(\rho^\vee,\al_i) = 1, d(\rho,\al_i) = \rr_i$.

Formula \eqref{expr} implies that
\begin{multline}    \label{AY}
A_i(z) = : Y_i(zq_it) Y_i(zq_i^{-1}t^{-1}) \\ \times
\prod_{j:I_{ji}=1} Y_j(z)^{-1}
\prod_{j:I_{ji}=2} Y_j(zq)^{-1} Y_j(zq^{-1})^{-1} \prod_{j:I_{ji}=3}
Y_j(zq^2)^{-1} Y_j(z)^{-1} Y_j(zq^{-2})^{-1} :
\end{multline}

Thus, for $\mathfrak{g}$ of non-simply laced type, the two classical
limits of $A_i(z)$ are quite different: when $t \to 1$,
we have
\begin{multline}    \label{AY1}
A_i(z) = Y_i(zq_i) Y_i(zq_i^{-1}) \\ \times \prod_{j:I_{ji}=1} Y_j(z)^{-1}
\prod_{j:I_{ji}=2} Y_j(zq)^{-1} Y_j(zq^{-1})^{-1} \prod_{j:I_{ji}=3}
Y_j(zq^2)^{-1} Y_j(z)^{-1} Y_j(zq^{-2})^{-1}
\end{multline}
but when $q \to 1$, we have a much simpler expression
\begin{equation}    \label{AY2}
A_i(z) = Y_i(zt) Y_i(zt^{-1}) \prod_{j\neq i} Y_j(z)^{-I_{ji}}.
\end{equation}

\begin{rem}\label{crystal1}

(1) The first limit (\ref{AY1}) coincides with the monomial $A_{i,z}$
  which appears in the theory of $q$-characters of finite-dimensional
  representations of quantum affine algebras introduced in
  \cite{FR:q}. This is not surprising because, as explained in
  \cite{FR:q} and in Section \ref{fdrep} below, the $q$-characters may
  be viewed as limits of the fields from ${\mathbf W}_{q,t}$ as $t \to
  1$.

(2) There is a surprising connection between the second limit
  (\ref{AY2}) and Kashiwara's extension to non-simply laced $\g$
  \cite{K} of Nakajima's monomial model for crystals of $U_t(\g)$
  \cite{N}. Let us recall that Nakajima's monomial realization was
  originally motivated by its relation with the $q$-characters in
    the symmetric cases. It turned out that
    this crystal realization was a consequence of the embedding
    theorem \cite{K8}, which makes sense in the symmetrizable case
    \cite{K}. But the relation between the monomial model and the
    $q$-characters was lost for non-simply laced types. Here we
    suggest an analogous relation, in which the role of the $t \to
    1$ classical limit of ${\mathbf W}_{q,t}(\g)$ (whose free field
    realization is essentially the same as the $q$-character
    homomorphism, see Remark \ref{freetoqchar}) is played by the $q
    \to 1$ classical limit.

  Recall that in the monomial model the vertices of the
  crystal are represented by certain monomials in the variables
  $Y_i(t^k)^{\pm 1}$, and the crystal operators are obtained by
  multiplying them with some special monomials corresponding to the
  simple roots. In the simply-laced case, these are the monomials
  $A_{i,a}^{\pm 1}$ occurring in the $q$-character theory (as in
  formula (\ref{AY1})). But in the non-simply laced case the monomials
  $A_{i,a}^{\pm 1}$ in formula (\ref{AY1}) do not work. Instead, as
  explained in \cite{K}, we have to replace these monomials
  $A_{i,a}^{\pm 1}$ with other monomials. A direct comparison shows
  that Kashiwara's monomials coincide with the monomials in the above
  formula (\ref{AY2}). We
  formulate a precise conjecture about this in Section \ref{cryssec}
  below.\qed
\end{rem}

\subsection{Screening operators and definition of ${\mc
    W}_{q,t}(\g)$}        \label{defWqt}

Recall that we have two sets of screening operators introduced in
\cite{FR:w}: $S^+_i(z)$ and $S^-_i(z)$, $i=1,\ldots,\ell$. They
satisfy the difference equations:
\begin{equation}    \label{scr1}
S^+_i(zq_i) = :A_i(z) S^+_i(zq_i^{-1}):,
\end{equation}
and
\begin{equation}    \label{scr2}
S^-_i(zt) = :A_i(z) S^-_i(zt^{-1}):.
\end{equation}
The deformed ${\mc W}$-algebra ${\mc W}_{q,t}(\g)$ was defined in
\cite{FR:w} as the intersection of kernels of the residues $S^+_i$ of
$S^+_i(z), i \in I$, or of the residues $S^-_i$ of $S^-_i(z), i \in
I$.

More precisely, $\mathcal{W}_{q,t}(\mathfrak{g})$ was defined in
\cite{FR:w} as the associative topological algebra depending on two
parameters $q$ and $t$, which is topologically generated by the
Fourier coefficients of certain fields from a deformed chiral algebra
${\mathbf W}_{q,t}(\mathfrak{g})$. The latter was defined in \cite{FR:w}
as the maximal subalgebra commuting with the screening operators
$S_i^\pm, i \in I$, in a deformed chiral algebra ${\mathbf
  H}_{q,t}(\mathfrak{g})$ constructed from the Heisenberg algebra
$\mathcal{H}_{q,t}(\mathfrak{g})$.

In the classical limits $t \to 1$ (resp. $q \to 1$), ${\mathbf
  H}_{q,t}(\mathfrak{g})$ becomes commutative:
\begin{equation}    \label{boldH}
{\mathbf H}_{q,1}(\g) = \C[Y_j(zq^{n_j})^{\pm 1}]_{j \in I, n_j \in
  \Z}, \qquad {\mathbf H}_{1,t}(\g) = \C[Y_j(zt^{n_j})^{\pm 1}]_{j
  \in I, n_j \in \Z}.
\end{equation}
The corresponding classical limit of ${\mathbf W}_{q,t}(\mathfrak{g})$
is a commutative subalgebra of this polynomial algebra, which is equal
to the intersection of the kernels of the classical screening
operators $S^+_i$ (resp. $S^-_i$), $i \in I$.

We will call these limits as the {\em classical ${\mc W}$-algebras}
and denote them by ${\mc K}^+_q(\g)$ and ${\mc K}^-_t(\g)$,
respectively. Below we describe both of these algebras. We will see
that they are quite different if $\g$ is non-simply laced.

\subsection{The $t \to 1$ limit}    \label{tto1}

In the limit $t \to 1$, the family $S^+_i, i \in I$, survives and
gives rise to the following derivations
$$
S^+_i: \C[Y_j(zq^{n_j})^{\pm 1}]_{n_j \in \Z} \to \left( \bigoplus_{m \in
  \Z} \C[Y_j(zq^{n_j})^{\pm
  1}]_{n_j \in \Z} \otimes S^+_i(zq^m) \right)/(S^+_i(zq_i^2) - A_i(zq_i)
S^+_i(z))
$$
acting by the formula
\begin{equation}    \label{S+}
S^+_i \cdot Y_j(zq^{n_j})^{\pm 1} = \pm \delta_{ij}
Y_j(zq^{n_j})^{\pm 1} \otimes S^+_i(zq^{n_j}).
\end{equation}

The following isomorphism was proved in \cite[Proposition 6]{FR:q} and
\cite[Proposition 5.2]{fm}:
\begin{equation}    \label{KerS+}
  \on{Ker} S^+_i =
\C[Y_j(zq^{n_j})^{\pm 1}]_{j\neq i;
  n_j \in \Z} \otimes \C[Y_i(zq^{n_i})(1 +
A_i(zq^{n_i}q_i)^{-1})]_{n_i \in \Z},
\end{equation}
where $A_i(z)$ is given by formula \eqref{AY1}.

Now set
\begin{equation}    \label{K+}
\mathcal{K}^+_q(\mathfrak{g}) :=
\bigcap_{i \in I} \on{Ker} S^+_i.
\end{equation}
Thus, $\mathcal{K}^+_q(\mathfrak{g})$ is the $t \to 1$ limit of
${\mathbf W}_{q,t}(\g)$. The following theorem was proved in
\cite[Theorem 5.1]{fm}.

\begin{thm}\label{props+}
(1) The commutative algebra $\mathcal{K}^+_q(\mathfrak{g})$ is
  isomorphic to $\on{Rep}_z U_q(\ghat)$, the Gro\-then\-dieck ring of
  the tensor subcategory $\mathcal{C}_{\mathbb{Z}}$ of the category of
  finite-dimensional representations of $U_q(\ghat)$ whose objects are
  representations with the Jordan-H\"older constituents having
  Drinfeld polynomials with roots in $q^{\mathbb{Z}}$.
  
(2) Under this isomorphism, the embedding
  $\mathcal{K}^+_q(\mathfrak{g}) \to {\mathbf H}_{q,1}(\g)$ becomes
  the $q$-character homomorphism.
\end{thm}

From Theorem \ref{props+} we obtain a natural basis of
$\mathcal{K}^+_q(\mathfrak{g})$ consisting of the $q$-characters of
simple modules from ${\mc C}_{\Z}$. It is known that these are
parametrized by dominant monomials, i.e. monomials in the variables
$Y_j(zq^{n_j})$, $j \in I$, $n_j\in\mathbb{Z}$, with only non-negative
powers. In particular, the elements of this basis corresponding to the
degree one dominant monomials $Y_j(zn^j)$ coincide with the
$q$-characters of the corresponding fundamental representations (those
are known to contain a unique dominant monomial, see \cite[Corollary
  4,(1)]{FR:q}).

\subsection{The $q \to 1$ limit}   \label{qto1}

Now consider the limit $q \to 1$. Then it is the family $S^-_i, i \in
I$, that survives and gives rise to the derivations
$$
S^-_i: \C[Y_j(zt^{n_j})^{\pm 1}]_{n_j \in \Z} \to \left( \bigoplus_{m \in
  \Z} \C[Y_j(zt^{n_j})^{\pm
  1}]_{n_j \in \Z} \otimes S^+_i(zt^m) \right)/(S^-_i(zt^2) - A_i(zt)
S^-_i(z))
$$
acting by the formula
\begin{equation}    \label{S-}
S^-_i \cdot Y_j(zt^{n_j})^{\pm 1} = \pm \delta_{ij}
Y_j(zt^{n_j})^{\pm 1} \otimes S^-_i(zt^{n_j}),
\end{equation}
where $A_i(z) = Y_i(zt) Y_i(zt^{-1}) Y_j(z)^{-I_{ji}}$ (formula
\eqref{AY2}).

We have the following analogue of the isomorphism \eqref{KerS+} (it is
equivalent to \eqref{KerS+} if $\g$ is simply-laced, but for a
non-simply laced Lie algebra $\g$ this statement is new, as far as we
know).

\begin{prop}    \label{props-}
We have
\begin{equation}    \label{KerS-}
\on{Ker} S^-_i = \bigcap_{i \in I}
\C[Y_j(zt^{n_j})^{\pm 1}]_{j\neq i;
  n_j \in \Z} \otimes \C[Y_i(zt^{n_i})(1 + A_i(zt^{n_i+1})^{-1})]_{n_i
  \in \Z}
\end{equation}
where $A_i(z) = Y_i(zt) Y_i(zt^{-1}) \prod_{j \neq i}
Y_j(z)^{-I_{ji}}$ (formula \eqref{AY2}).
\end{prop}

The proof is obtained by applying the argument used in the proof of
\eqref{KerS+} in \cite[Proposition 5.2]{fm}.

Now set
\begin{equation}    \label{K-}
\mathcal{K}^-_t(\mathfrak{g}) =
\bigcap_{i \in I} \on{Ker} S^-_i.
\end{equation}
Thus, $\mathcal{K}^-_q(\mathfrak{g})$ is the $q \to 1$ limit of
${\mathbf W}_{q,t}(\g)$.

Unlike the limit $t \to 1$ (see Theorem \ref{props+}), for non-simply
laced $\g$ we do not have an identification of
$\mathcal{K}^-_t(\mathfrak{g})$ with the Grothendieck ring of a
category of representation. Nonetheless, we do have a basis in
$\mathcal{K}^-_t(\mathfrak{g})$ analogous to the basis of the
$q$-characters of simple modules in $\mathcal{K}^+_q(\mathfrak{g})$.

\begin{prop}\label{intmoins}
(1) Every element of $\mathcal{K}^-_t(\mathfrak{g})$ is characterized
  by the multiplicities of the dominant monomials contained in it
  (i.e. monomials in the $Y_{j,t^{n_j}}$, $j \in I$,
  $n_j\in\mathbb{Z}$, with only non-negative powers).

(2) For every dominant monomial $m$, there is a unique element
  $F(m)$ of $\mathcal{K}^-_t(\mathfrak{g})$ such that $m$ is the
  unique dominant monomial of $F(m)$. Therefore we obtain a basis $\{
  F(m) \}$ of $\mathcal{K}^-(\mathfrak{g})$ parametrized by dominant
  monomials $m$.
\end{prop}

\begin{proof}
The proof is the same as the proof in \cite[Theorem 5.13]{H2}. All that
remains is to check the existence of the $F(Y_i(zq^n))$ for rank $2$
Lie algebras. For simply-laced types, this is true because we can use
the ordinary $q$-characters of the fundamental representations. For
type $B_2$, we find the following elements:
$$F(Y_1(z)) = Y_1(z) + Y_1(zq^2)^{-1}Y_2(zq)^2 + 2
Y_2(zq)Y_2(zq^3)^{-1} + Y_2(zq^3)^{-2}Y_1(zq^2) +
Y_1(zq^4)^{-1},$$
$$F(Y_2(z)) = Y_2(z) + Y_2(zq^2)^{-1}Y_1(zq) +
Y_1(zq^3)^{-1}Y_2(zq^2) + Y_2(zq^4)^{-1},$$
and for type $G_2$ we find
$$F(Y_1(z)) = Y_1(z) + Y_1(zq^2)^{-1}Y_2(zq)^3 
+ 3 Y_2(zq)^2Y_2(zq^3)^{-1} + 3 Y_2(zq)Y_2(zq^3)^{-2}Y_1(zq^2)
$$ $$+ Y_2(zq^3)^{-3} Y_1(zq^2)^2 + 3 Y_2(zq)Y_2(zq^3)Y_1(zq^4)^{-1}
+ 2 Y_1(zq^2)Y_1(zq^4)^{-1}
+ 3 Y_2(zq)Y_2(zq^5)^{-1} $$
$$+ Y_1(zq^4)^{-2}Y_2(zq^3)^3 + 3 Y_2(zq^3)^{-1}Y_2(zq^5)^{-1} Y_1(zq^2)
+ 3 Y_1(zq^4)^{-1}Y_2(zq^3)^2 Y_2(zq^5)^{-1} $$
$$+3 Y_2(zq^3)Y_2(zq^5)^{-2}
+ Y_2(zq^5)^{-3}Y_1(zq^4)
+ Y_1(zq^6)^{-1},
$$
$$
F(Y_2(z)) = Y_2(z) + Y_2(zq^2)^{-1}Y_1(zq) +
Y_1(zq^3)^{-1}Y_2(zq^2)^2  + 2 Y_2(zq^2)Y_2(zq^4)^{-1} $$ $$+
Y_2(zq^4)^{-2}Y_1(zq^3) + Y_2(zq^4)Y_1(zq^5)^{-1} +
Y_2(zq^6)^{-1}.$$
This completes the proof, up to the fact that the algorithm may
produce elements 
$F(Y_i(zq^n))$ with an infinite number of terms. But it follows from
Theorem \ref{unfold} below
that the elements $F(Y_i(zq^n))$ can also be obtained as folded
$q$-characters of fundamental representations of
$U_q(\wh{\mathfrak{g}'})$, which do have finite numbers of terms.
\end{proof}

It follows from Proposition \ref{intmoins} that we have natural
analogues $F(Y_i(zt^n))$ of the $q$-characters of the fundamental
representations. We also have natural analogues of $q$-characters of
the Kirillov--Reshetikhin modules (see Section \ref{fdrep} below).

\begin{rem}    \label{mcA}
This discussion motivates the following natural question: Is there a
Hopf algebra ${\mc A}_t(\g)$ (an analogue of $U_t(\ghat)$) together
with an injective ($t$-character) homomorphism
$$
\on{Rep}_z
  {\mc A}_t(\g) \to \C[Y_j(zt^{n_j})^{\pm
    1}]_{j \in I; n_j \in \Z},
$$
where $\on{Rep}_z {\mc A}_t(\g)$ is a subring of the Grothendieck ring
of the category of finite-dimensional representations of 
${\mc A}_t(\g)$, whose image is $\mathcal{K}_t^-(\mathfrak{g})$?

It is tempting to try to answer this question using an automorphism
$\sigma$ of the quantum affine algebra $U_t(\wh{\g'})$ defined by
formula (\ref{dact1}) below. The subalgebra $(U_t(\wh{\g'}))^\sigma$
of $\sigma$-invariants acts on every finite-dimensional representation
$V$ of $U_t(\wh{\g'})$.  However, it is not clear how to define a
comultiplication on the algebra $(U_t(\wh{\g'}))^\sigma$. Hence it is
not clear how one could define a quantum integrable model this
way.\qed
\end{rem}

\subsection{The deformed ${\mc W}$-algebra for general $q$ and
  $t$}

The structure of the deformed chiral algebra ${\mathbf W}_{q,t}(\g)$
for general values of $q$ and $t$ is much more complicated than that of
its classical limits discussed above. Conjecture 1 of \cite{FR:w}
implies that every basis element of ${\mathbf W}_{q,1}(\g) = {\mc
  K}^+_q(\g)$ given by the $q$-character of a simple module over
$U_q(\ghat)$ can be deformed to a basis element of ${\mathbf
  W}_{q,t}(\g)$ (i.e. an element of ${\mathbf H}_{q,t}(\g)$ which lies
in the kernel of the screening operators).

However, apart from a few explicit examples presented in \cite{FR:w},
there is no proof of existence of these elements in general. On the
other hand, in \cite{FH1} a simplified, commutative version of ${\mc
  W}_{q,t}(\g)$ was introduced, called the space of {\em interpolating
  $(q,t)$-characters}.  These are defined from certain subrings which
are modeled on what we expect the kernels of the screening operators
to be (based on the description of the kernels of screening operators
associated to the ordinary $q$-characters). We will recall this
construction, and add further details to it, in Section \ref{intchar}.

At the moment, the relation between the interpolating
$(q,t)$-characters and the deformed ${\mc W}$-algebra ${\mathbf
  W}_{q,t}(\g)$ is conjectural. However, as we will see below, for our
purposes the $(q,t)$-characters provide a good substitute for elements
of ${\mathbf W}_{q,t}(\g)$.

\section{Classical limits of the deformed ${\mc W}$-algebra and
  $q$-characters}\label{clasqcar}

In this section we first recall some details on the relation between
the $t \to 1$ limit of the deformed ${\mc W}$-algebra and the
$q$-characters of finite-dimensional representations of
$U_q(\ghat)$. We then relate the $q \to 1$ limit to what we call {\em
  folded} $t$-characters of finite-dimensional representations of
$U_t(\wh{\mathfrak{g}'})$ (see Theorem \ref{unfold}).

\subsection{Reminder on the $q$-characters of representations of quantum
  affine algebras}\label{fdrep}

First, consider the untwisted quantum affine algebra
$U_q(\widehat{\mathfrak{g}})$. Let $\on{Rep} \U_q(\wh{\Glie})$ be
the Grothen\-dieck ring of finite-dimensional representations of
$\U_q(\wh{\Glie})$. The $q$-character homomorphism \cite{FR:q} is an
injective ring homomorphism
$$\chi_q : \on{Rep} \U_q(\wh{\Glie}) \rightarrow \Yim_q =
\ZZ[Y_{i,a}^{\pm 1}]_{i\in I, a\in \mathbb{C}^\times}.$$

If we replace each $Y_{i,a}$ by $y_i$, we recover the usual character
homomorphism for the $\U_q(\Glie)$-module obtained by restriction of
$\U_q(\wh{\Glie})$-module, which encodes its grading by the lattice of
integral weights of a Cartan subalgebra of the Lie algebra $\Glie$. In
what follows, we will refer to these integral weights as $\Glie$-{\em
  weights}.  In particular, each monomial in $\Yim_q$ has a
$\Glie$-weight.

It is proved in \cite{FR:q, fm} (see also Theorem \ref{props+} above)
that
\begin{equation}    \label{Imchiq}
  \on{Im}(\chi_q) = \bigcap_{i\in I}\mathfrak{K}_{i,q},
\end{equation}
where
\begin{equation}    \label{keri}
\mathfrak{K}_{i,q} = \ZZ[Y_{j,a}^{\pm 1}, Y_{i,a}(1 +
  A_{i,aq_i}^{-1})]_{j\neq i,a\in \mathbb{C}^\times}
\end{equation}
and $A_{i,a}$ is defined by formula (\ref{AY1}), where we replace
$Y_j(za)$ with $Y_{j,a}$.

\begin{rem}    \label{freetoqchar}
Note that in the context of
deformed ${\mc W}$-algebras, it is convenient to restrict ourselves to
the variables $Y_j(zq^{n_j}), n_j \in \Z$ (i.e. restrict ourselves to
the multiplicative lattice of spectral parameters $a = zq^n, n \in
\Z$). But in the context of $q$-characters, we usually consider all
spectral parameters $a \in \C^\times$ and denote the corresponding
variables by $Y_{j,a}$ (see \cite[Sect. 7]{FR:q} for more detail).

In particular, comparing formulas \eqref{K+} and \eqref{KerS+} with
formulas \eqref{Imchiq} and \eqref{keri}, respectively, we find that
if we replace the variables $Y_j(zq^{n_j}), n_j \in \Z$, by the
variables $Y_{j,a}, a \in \C^\times$, then ${\mc K}^+_q$ becomes
$\on{Im}(\chi_q)$. That's what we mean by the statement that the $t
\to 1$ limit of the free field realization of ${\mathbf W}_{q,t}(\g)$
corresponds to the $q$-character homomorphism of $U_q(\ghat)$.\qed
\end{rem}

A monomial in $\Yim_q$ is called dominant if it is a product of
positive powers of the $Y_{i,a}, i\in I, a\in \mathbb{C}^\times$. A simple
$\U_q(\wh{\Glie})$-module is uniquely characterized by the highest
monomial (in the sense of its $\Glie$-weight) in its $q$-character
(this monomial encodes the data of the Drinfeld
polynomials; for the definition of the latter, see Theorem 12.2.6 of
\cite{cp}). This monomial is dominant. An element of
$\text{Im}(\chi_q)$ is characterized by the multiplicities of its
dominant monomials. A $\U_q(\wh{\Glie})$-module is said to be
affine-minuscule if its $q$-character has a unique dominant monomial.

If a dominant monomial is in $\ZZ[Y_{i,a}^{\pm 1}]_{i\in I, a\in
  q^{\mathbb{Z}}}$, then the $q$-character of the corresponding simple
module also belongs to this subring.

A Kirillov--Reshetikhin (KR) module of $\U_q(\wh{\Glie})$ is a simple
module with the highest monomial of the form
$Y_{i,a}Y_{i,aq_i^2}\cdots Y_{i,aq_i^{2(k-1)}}$ with
$a\in\mathbb{C}^\times$, $i\in I$ and $k\geq 0$.

It was proved in \cite{Nad, hcr} that the KR modules of
$\U_q(\wh{\Glie})$ are affine-minuscule.  For $k=1$, that is for
fundamental representations, this was proved in \cite{fm}.

\subsection{Twisted affine algebras}

Next, consider the twisted quantum affine algebra
$U_t({}^L\wh{\Glie})$, and let $\on{Rep} \U_t({}^L\wh{\Glie})$ be the
Grothen\-dieck ring of the category of its finite-dimensional
representations.

The twisted $t$-character homomorphism \cite{H} is an injective ring
homomorphism
$$\chi_t: \on{Rep} U_t({}^L\wh{\Glie}) \rightarrow
\ZZ[Z_{i,a^{\rr_i^\vee}}^{\pm 1}]_{a\in \mathbb{C}^\times, i\in I},$$ where
we have set
\begin{equation}    \label{dvee}
  \rr_i^\vee = \rr + 1 - \rr_i.
\end{equation}
These are the analogues of
the $\rr_i$ for the Langlands dual Lie algebra ${}^L\Glie$.

As in the
untwisted case, we have the notions of dominant monomials,
affine-minuscule modules and KR modules. An
element of $\text{Im}(\chi_t)$ is again characterized by its
dominant monomial and the KR modules of $\U_t({}^L\wh{\Glie})$ are
affine-minuscule, as proved in \cite{H}.

If a dominant monomial is in $\ZZ[Z_{i,a^{\rr_i^\vee}}^{\pm 1}]_{a\in
  \epsilon^\ZZ t^\ZZ, i\in I}$, then the twisted $t$-character of the
corresponding simple module also belongs to this subring.

The image of $\chi_t$ is equal to
$$\bigcap_{i\in I}\ZZ[Z_{j,a^{\rr_j^\vee}}^{\pm 1}, Z_{i,a^{\rr_i^\vee}}(1 +
  B_{i,(at)^{\rr_i^\vee}}^{-1})]_{j\neq i,a\in \C^\times},$$
where 
$$B_{i,a} =
Z_{i,at^{\rr_i^\vee}}Z_{i,at^{-\rr_i^\vee}}\times\prod_{j\sim
  i|\rr_j^\vee = d}Z_{j,a^{\rr_i}}^{-1}\times \prod_{j\sim i, a'|
  \rr_j^\vee = 1,(a')^{\rr_i^\vee} = a}Z_{j,a'}^{-1},$$
where we write $i \sim j$ if $I_{ij} \neq 0$ (recall that $(I_{ij})$
denotes the incidence matrix).

Note that a special definition should be used for the monomials
$B_{i,a}$ in the case of type $A_{2n}^{(2)}$, but we are not
considering this case here because this affine Kac--Moody algebra is
not dual to an untwisted affine algebra (note that $A_{2n}^{(2)}$ does
not appear in Section \ref{nomen}).
 
According to Theorem \ref{props+}, the $t \to 1$ limit ${\mc
  K}^+_q(\g)$ of ${\mathbf W}_{q,t}(\g)$ is isomorphic to $\on{Rep}
U_q(\ghat)$ so that the embedding of ${\mathbf W}_{q,1}(\g)$ into
${\mathbf H}_{q,1}(\g)$ becomes the $q$-character homomorphism.

\medskip

Our task is to relate the $q \to 1$ limit ${\mc K}^-_t(\g)$ of
${\mathbf W}_{q,t}(\g)$ to $t$-characters of representations of
quantum affine algebras. We start with two examples and then derive a
general result. The upshot is that ${\mc K}^-_t(\g)$ is spanned by
what we will call {\em folded $t$-characters of $U_t(\wh{\g'})$},
where $\g'$ is the simply-laced Lie algebra equipped with an
automorphism whose invariant Lie subalgebra is $\g$ (the Dynkin
diagram of $\g$ can be obtained by folding the Dynkin diagram of
$\g'$). These are the $t$-characters of the finite-representations of
$U_t(\wh{\g'})$ in which we identify the variables $Y_i(z)$ and
$Y_{\sigma(i)}(z)$ for all $i \in I'$.

\subsection{Examples}    \label{gprime}

Consider the case $\g=B_\el$. Then we have the following formula for
the element $T_1(z)$ of ${\mathbf W}_{q,t}(B_\el)$ corresponding to the
first fundamental representation of $U_q(B^{(1)}_\el)$ (see
\cite{FR:w}, Sect. 5.1.2). Set
$$
J = \{ 1,\ldots,\el,0,\ol{\el},\ldots,\ol{1} \}
$$
\begin{align*}
\La_i(z) &= :Y_i(zq^{2i-2} t^{i-1}) Y_{i-1}(zq^{2i} t^i)^{-1}:, \quad
\quad i=1,\ldots,\el-1, \\
\La_\el(z) &= :Y_\el(zq^{2\el-3} t^{\el-1}) Y_\el(zq^{2\el-1} t^{\el-1})
Y_{\el-1}(zq^{2\el} t^\el)^{-1}:, \\
\La_0(z) &= \frac{(q+q^{-1})(qt - q^{-1} t^{-1})}{q^2 t - q^{-2} t^{-1}}
:Y_\el(zq^{2\el-3} t^{\el-1}) Y_\el(zq^{2\el+1} t^{\el+1})^{-1}:, \\
\La_{\ol{\el}}(z) &= :Y_{\el-1}(zq^{2\el-2} t^\el) Y_\el(zq^{2\el-1}
t^{\el+1})^{-1} Y_\el(zq^{2\el+1} t^{\el+1})^{-1}:, \\
\La_{\ol{i}}(z) &= :Y_{i-1}(zq^{4\el-2i-2} t^{2\el-i}) Y_i(zq^{4\el-2i}
t^{2\el-i+1})^{-1}:, \quad \quad i=1,\ldots,\el-1.
\end{align*}
Here and below we set $Y_0(z)=1$.
According to \cite{FR:w},
\begin{equation}    \label{t1}
T_1(z) = \sum_{i \in J} \La_i(z)
\end{equation}
commutes with the screening operators $S^\pm_i, i=1,\ldots,\el$ and hence
belongs to ${\mathbf W}_{q,t}(B_\el)$.

The $t \to 1$ limit of the rational function
\begin{equation}    \label{rat fn}
f_\ell(q,t) = \frac{(q+q^{-1})(qt^{-1} - q^{-1} t)}{q^2 t^{-1} - q^{-2} t}
\end{equation}
is equal to 1. Hence the limit of the above formula for $T_1(z)$ as $t
\to 1$ has $2\el+1$ terms, and one can check that it coincides with
the $q$-character of the first fundamental representation of
$U_q(B_\el^{(1)})$.

Now consider the $q \to 1$ limit of \eqref{t1}. Note that $f_\el(1,t)
= 2$, so in this limit the term $\La_0(z)$ appears with coefficient 2.
As observed in \cite[Sect. 6.3]{FR:w}, the $q \to 1$ limit of $T_1(z)$
looks like the $t$-character of $U_t(D_{\el+1}^{(2)})$ in which we
remove all $\epsilon$ factors. For this reason, it was conjectured in
\cite[Sect. 6.3]{FR:w} that the $q \to 1$ limit of the elements of
${\mathbf W}_{q,t}(\g)$ in general (which are, by definition, elements of
$\mathcal{K}_t^-(\mathfrak{g})$) should be given by the $t$-characters
of finite-dimensional representations of $U_t(\G^\vee)$ in which we
remove the $\epsilon$ factors (note that $D_{\el+1}^{(2)} =
(B_\el^{(1)})^\vee$).

This conjecture is likely to be true, but the problem is that we don't
know how to interpret this removal of $\epsilon$ factors in terms of
representation theory, and therefore this does not help us with
constructing the corresponding integrable models. Hence in this paper
we give a different interpretation of this limit.

To explain it, let's look more closely at the $q \to 1$ limit of the
formula \eqref{t1} for $T_1(z)$. This is an element of ${\mc
  K}^-_t(B_\el)$ given by the same formula \eqref{t1}, where now
\begin{align*}
\La_i(z) &= Y_i(zt^{i-1}) Y_{i-1}(zt^i)^{-1}, \quad
\quad i=1,\ldots,\el-1, \\
\La_\el(z) &= Y_\el(zt^{\el-1})^2 Y_{\el-1}(zt^\el)^{-1}, \\
\La_0(z) &= 2 Y_\el(zt^{\el-1}) Y_\el(zt^{\el+1})^{-1}, \\
\La_{\ol{\el}}(z) &= Y_{\el-1}(zt^\el) Y_\el(zt^{\el+1})^{-2}, \\
\La_{\ol{i}}(z) &= Y_{i-1}(zt^{2\el-i}) Y_i(zt^{2\el-i+1})^{-1},
\quad \quad i=1,\ldots,\el-1.
\end{align*}

Let's compare this formula with the $q \to 1$ limit of the formula for
$T_1(z)$ from ${\mathbf W}_{q,t}(D_{\el+1})$ (see \cite{FR:w},
Sect. 5.1.4). It is an element ${\mc K}^-_t(D_{\el+1})$ given by
formula \eqref{t1} but now with
$$
J = \{ 1,\ldots,\el+1,\ol{\el+1},\ldots,\ol{1} \}
$$
\begin{align*}
\La_i(z) &= Y_i(zt^{i-1}) Y_{i-1}(t^i)^{-1}, \quad
\quad i=1,\ldots,\el-1, \\
\La_{\el}(z) &= Y_{\el+1}(zt^{\el-1}) Y_{\el}(zt^{\el-1}) Y_{\el-1}(zt^{\el})^{-1}, \\
\La_{\el+1}(z) &= Y_{\el+1}(zt^{\el-1}) Y_{\el}(zt^{\el+1})^{-1}, \\
\La_{\ol{\el+1}}(z) &= Y_{\el}(zt^{\el-1}) Y_{\el+1}(zt^{\el+1})^{-1}, \\
\La_{\ol{\el}}(z) &= Y_{\el-1}(zt^{\el})
Y_{\el}(zt^{\el+1})^{-1} Y_{\el+1}(zt^{\el+1})^{-1}, \\
\La_{\ol{i}}(z) &= Y_{i-1}(zt^{2\el-i}) Y_i(zt^{2\el-i+1})^{-1},
  \quad \quad i=1,\ldots,\el-1.
\end{align*}

By inspecting these formulas, we obtain the following result (recall
that ${\mc K}^-_t(\g)$ is defined in formula \eqref{K-}).

\begin{lem}    \label{identi}
  If we identify the generators $Y_{\el+1}(z)$ with $Y_\ell(z)$ for
  $D_{\el+1}$, then the formula for $T_1(z)$ in ${\mc
    K}^-_t(D_{\el+1})$ becomes the formula for $T_1(z)$ in ${\mc
    K}^-_t(B_\el)$.
\end{lem}

Now observe that in this case $\g=B_\ell$ and $\g'=D_{\el+1}$, with
the corresponding automorphism $\sigma$ exchanging the $\el$th and the
$(\el+1)$st nodes of the Dynkin diagram. Hence the identification of
Lemma \ref{identi} corresponds precisely to the {\em folding} of the
Dynkin diagram of $D_\ell$, which gives the Dynkin diagram of $B_\el$.

\bigskip

Let us apply the same procedure in the case $\g=C_\el$. Then
$\g'=A_{2\el-1}$ and $\sigma$ exchanges the $i$th and the $2\el-i$th
nodes of the Dynkin diagram. Comparing formulas in Sects. 5.1.1 and
5.1.3 of \cite{FR:w}, we obtain
\begin{lem}
  Let us identify the generators $Y_{2\el-i}(z), i=1,\ldots,\el-1$,
  with $Y_i(z)$ for $A_{2\el-1}$. Then the formula for $T_1(z)$
  in ${\mc K}^-_t(A_{2\el-1})$ becomes the formula for $T_1(z)$ in
  ${\mc K}^-_t(C_\el)$.
\end{lem}

\subsection{General case}    \label{Kminus}

Formula \eqref{AY2} shows that if we impose the relations $Y_i(z) =
Y_{\sigma(i)}(z)$ for all $i \in I'$, then the generators $A_i(z), i
\in I'$ of ${\mc K}^-_t(\g')$ go to the corresponding generators
$A_i(z), i \in I$, of ${\mc K}^-_t(\g)$. Here and below, abusing
notation, we identify $i \in I'$ with its image in $I = I'/\langle
\sigma \rangle$.

Thus, we have natural commutative diagram
$$
\begin{CD}
\bigcap_{i \in I'} \C[Y_j(zt^{n_j})^{\pm 1}]_{j\neq i;
  n_j \in \Z} \otimes \C[Y_i(zt^{n_i})(1 + A_i(zt^{n_i+1})^{-1})]_{n_i
  \in \Z} @>>> \C[Y_j(zt^{n_j})^{\pm 1}]_{j\in I'; n_j \in \Z}\\
@VVV     @VVV\\
\bigcap_{i \in I} \C[Y_j(zt^{n_j})^{\pm 1}]_{j\neq i;
  n_j \in \Z} \otimes \C[Y_i(zt^{n_i})(1 + A_i(zt^{n_i+1})^{-1})]_{n_i
  \in \Z} @>>> \C[Y_j(zt^{n_j})^{\pm 1}]_{j\in I; n_j \in \Z}
\end{CD}
$$
with the vertical maps being surjective and the horizontal maps being
injective. This proves the following theorem stating that
$\mathcal{K}_t^-(\mathfrak{g}) = {\mathbf W}_{1,t}(\g)$ is spanned by
the $t$-characters of $U_t(\wh{\g'})$ in which we identify $Y_i(z)$
with $Y_{\sigma(i)}(z)$.

\begin{theorem}\label{unfold}
There is a surjective
  homomorphism $\on{Rep}_z U_t(\wh{\g'}) \to \mathcal{K}_t^-(\mathfrak{g})$
  that fits in the commutative diagram
$$
\begin{CD}
\on{Rep}_z U_t(\wh{\g'}) @>{\sim}>> \bigcap_{i \in I'}
\C[Y_j(zt^{n_j})^{\pm 1}]_{j\neq i;
  n_j \in \Z} \otimes \C[Y_i(zt^{n_i})(1 + A_i(zt^{n_i+1})^{-1})]_{n_i
  \in \Z} \\
@VVV     @VVV\\
\mathcal{K}_t^-(\mathfrak{g}) @>{\sim}>> \bigcap_{i \in I}
\C[Y_j(zt^{n_j})^{\pm 1}]_{j\neq i;
  n_j \in \Z} \otimes \C[Y_i(zt^{n_i})(1 + A_i(zt^{n_i+1})^{-1})]_{n_i
  \in \Z}
\end{CD}
$$
\end{theorem}

We will call the composition of the left vertical and the lower
horizontal maps the {\em folded $t$-character homomorphism} and denote
it by $^{\on{f}}\chi_t$:
\begin{equation}    \label{foldedtchar}
^{\on{f}}\chi_t: \on{Rep}_z U_t(\wh{\g'}) \to \C[Y_i(zt^{n_i})]_{i \in I}.
\end{equation}
It extends naturally to the entire $\on{Rep} U_t(\wh{\g'})$. We will
also call $\mathcal{K}_t^-(\mathfrak{g})$ the {\em folded
  $t$-character ring}.

Thus, we obtain an interpretation of the $q \to 1$ limit ${\mc
  K}^-_t(\g)$ of ${\mathbf W}_{q,t}(\g)$ as a ``folding'' of the
Grothendieck ring of finite-dimensional representations of
$U_t(\wh{\g'})$.

\subsection{Connection to the Drinfeld--Sokolov reduction and the
  center of quantum affine algebra} \label{center}

In this subsection we briefly discuss links between the classical
limits of ${\mc W}_{q,t}(\g)$ and other (Poisson) algebras.

First, Conjecture 3 of \cite{FR:w} states that the limit $q \arr 1$ of
$\W_{q,t}(\g)$ is isomorphic, as a Poisson algebra, to the Poisson
algebra obtained by the deformed Drinfeld--Sokolov reduction of
$G((z))$ with parameter $p=t^d$. This was confirmed by an explicit
computations in the case of $\g=C_2$ presented in Appendix B of
\cite{FR:w}.

Second, recall that the $t \to 1$ limit of $\W_{q,t}(\g)$ is
isomorphic, as a commutative algebra, to the center $Z_q(\ghat)$ of
$U_q(\ghat)$ at the critical level. Indeed, we can associate elements
of $Z_q(\ghat)$ to finite-dimensional representations of $U_q(\ghat)$
using a ``double'' of the transfer-matrix construction
\cite{RST,FR:center}, and hence connect $Z_q(\ghat)$ to $\on{Rep}
U_q(\ghat)$, which is ${\mc K}^+_q(\g) = {\mathbf W}_{q,1}(\g)$. In
\cite{FR:center}, it was shown that in the case of $\g=A_\el$ this is
in fact an isomorphism of Poisson algebra. We expect this to be true
for a general $\g$ as well.

If there exists an algebra ${\mc A}_t(\g)$ satisfying the properties
of Remark \ref{mcA}, then it is reasonable to expect that the center
of ${\mc A}_t(\g)$ at its ``critical level'' is isomorphic to the $q
\to 1$ limit of ${\mc W}_{q,t}(\g)$.

\begin{rem}
Note that the limit of $Z_q(\ghat)$ as $q \to 1$ is the center of
$U(\ghat)$ at the critical level. According to the Feigin--Frenkel
isomorphism, the latter is isomorphic to the classical ${\mc
  W}$-algebra of the Langlands dual Lie algebra $^L\g$ (and not of
$\g$). This is consistent with the fact that in the $q \to 1$ limit
the screening operators $S^+_i$, whose joint kernel is the Poisson
algebra ${\mc W}_{q,1}(\g)$, become the screening operators of the
classical ${\mc W}$-algebra of $^L\g$.\qed
\end{rem}

\subsection{Bethe Ansatz equations from deformed ${\mc
    W}$-algebras}    \label{BAE}

First, consider the limit $t \to 1$, in which we obtain the algebra of
$q$-characters of finite-dimensional representations of
$U_q(\ghat)$. In \cite[Sect. 6.3]{FR:q}, it was shown how to derive
the corresponding Bethe Ansatz equation (BAE) from these
$q$-characters. Namely, according to the conjecture in
\cite[Sect. 6.1]{FR:q}, which was proved in \cite{FH}, the eigenvalues
of the transfer-matrices on the representation $W = \bigotimes_{j=1}^N
V(\bp_j)$ can be expressed (up to an overall factor) as the
$q$-characters in which we replace the variables $Y_{i,a}$ by the
ratios $Q_i(aq_i^{-1})/Q_i(aq_i)$ of the corresponding Baxter
polynomials. Then the seeming poles corresponding to the monomials of
the form $M$ and $M A_{i,aq_i}^{-1}$ in the $q$-characters should
cancel each other. Using formula \eqref{AY1}, we obtain the following
BAE:
\begin{equation}    \label{bae gen}
\prod_{k=1}^N q_i^{\deg P_{i,k}}
  \frac{P_{i,k}(q_i^{-1}/w^{(i)}_r)}{P_{i,k}(q_i/w^{(i)}_r)} =
  - \prod_{s \neq r}
  \frac{w^{(i)}_r-w^{(i)}_sq_i^{-2}}{w^{(i)}_r-w^{(i)}_s q_i^{2}}
  \prod_{j \neq i} \prod_{s=1}^{m_j}
  \frac{w^{(i)}_r-w^{(j)}_sq^{-B_{ij}}}{w^{(i)}_r-w^{(j)}_s
  q^{B_{ij}}}.
\end{equation}
where $\{ w^{(i)}_k \}$ is the set of roots of the $i$th Baxter
polynomial $Q_i(z)$, and $P_{j,i}(z), j=1,\ldots,N; i=1,\ldots,\ell$,
are the Drinfeld polynomials of $W$.

\begin{rem}
This formula corrects a typo in formula (6.6) of \cite{FR:q}; namely,
the entries $C_{li}$ of the Cartan matrix there should be replaced by
the entries $B_{li}$ of the {\em symmetrized} Cartan matrix.\qed
\end{rem}

Thus, the key point in deriving these BAE is formula \eqref{AY1}
expressing variables $A_i(z)$ in terms of $Y_j(z)$.

\medskip

Now consider the limit $q \to 1$ and apply the argument of
\cite[Sect. 6.3]{FR:q} to formula \eqref{AY2} instead of
\eqref{AY1}. In the same way as in \cite{FR:q} we then obtain the
following equations:
\begin{equation}    \label{bae gen1}
\prod_{k=1}^N q_i^{\deg P_{i,k}}
  \frac{P_{i,k}(t^{-1}/w^{(i)}_r)}{P_{i,k}(t/w^{(i)}_r)} =
  - \prod_{s \neq r}
  \frac{w^{(i)}_r-w^{(i)}_st^{-2}}{w^{(i)}_r-w^{(i)}_s t^2}
  \prod_{j \neq i} \prod_{s=1}^{m_j}
  \frac{(w^{(i)}_r-w^{(j)}_st)^{-C_{ji}}}{(w^{(i)}_r-w^{(j)}_st^{-1})^{-C_{ji}}}.
\end{equation}

This system is equivalent to the system of Bethe Ansatz equations
obtained in \cite{FKSZ} (formula (6.16)). It can also be obtained by
folding the BAE corresponding to $\g'$. See Section \ref{qopers} for
the explanation of the procedure of folding. We summarize this in the
following statement.

\begin{prop}   \label{foldedBAE}
The system \eqref{bae gen1} corresponding to the Lie algebra $\g$ is
equivalent to the system obtained by folding the BAE \eqref{bae gen}
corresponding to the Lie algebra $\g'$, i.e. assuming that
$m_i=m_{\sigma(i)}$ for all $i \in I'$, identifying $w^{(i)}_r \equiv
w^{\sigma(i)}_r$, and writing the equations in terms of the
variables $w^{(i)}_r, i \in I$.
\end{prop}

\section{Folded integrable models}    \label{folded}

In this section we describe a novel quantum integrable model in which
the spectra of the Hamiltonians correspond (conjecturally) to
solutions of the Bethe Ansatz equations (\ref{bae gen1}).

\subsection{Action of an automorphism $\sigma$}    \label{conj sect} 

Let $W$ be a simple finite-dimensional representation of
$U_q(\widehat{\g'})$ whose highest monomial is $\sigma$-invariant,
i.e. it is a monomial in the elements $\wt{Y}_{i,a}$ defined in
Section \ref{wtY} below.\footnote{More generally, we could consider
the case when $W$ is a tensor product of simple representations, with
the same invariance property for its highest monomial (but not
necessarily for each simple factor). However, we will not do so in
this paper.} Denote the corresponding highest $\g'$-weight by
$\Lambda(W)$. Clearly, $\sigma(\Lambda(W)) = \Lambda(W)$.

Denote by $P$ the set of all $\g'$-weights and by $P^\sigma$ its subset
of $\sigma$-invariant $\g'$-weights. For $\ga \in P$, denote by $W_\ga
\subset W$ the weight $\ga$ subspace of $W$. Let $\wh{W}$ be the
direct sum of the subspaces of $W$ corresponding to $\sigma$-invariant
$\g'$-weights $\ga$:
$$
\wh{W} = \bigoplus_{\ga \in P^\sigma} W_\ga,
$$
In particular, $\wh{W}$ contains the one-dimensional highest weight
subspace $W_{\Lambda(W)}$ of $W$.

Next, let $\wt{W}\subset W$ be the direct sum of $\ell$-weight
subspaces of $W$ corresponding to the $\sigma$-invariant
$\ell$-weights (equivalently, $\sigma$-invariant monomials in
$\chi_q(W)$). Associating to each $\ell$-weight the corresponding
$\g'$-weight, we obtain a grading on $\wt{W}$ by $\sigma$-invariant
$\g'$-weights:
$$
\wt{W} = \bigoplus_{\ga \in P^\sigma} \wt{W}_\ga.
$$
We have an inclusion $\wt{W} \subset \wh{W}$
respecting the grading by ($\sigma$-invariant) $\g'$-weights:
\begin{equation}    \label{inclga}
  \wt{W}_\ga \subset W_\ga, \qquad \ga \in P^\sigma.
\end{equation}
The restriction of this inclusion to $\wt{W}_{\Lambda(W)}$ is an
isomorphism. Next, we define an automorphism of $U_q(\wh{\g'})$
corresponding to $\sigma$ (abusing notation, we denote it in the same
way).

\begin{lem}
There is a unique algebra automorphism $\sigma$ of $U_q(\wh{\g'})$
defined on the Drinfeld generators by the formulas
\begin{equation}\label{dact1}
  \sigma(h_{i,r}) = h_{\sigma(i),r}, \qquad \sigma (x_{i,m}^\pm) 
  =  x_{\sigma(i),m}^\pm, \qquad \sigma(k_i^{\pm 1})
  =k_{\sigma(i)}^{\pm 1}.
\end{equation}
\end{lem}

\begin{proof}
One checks directly that the map $\sigma$ preserves the relations
between the Drinfeld generators.
\end{proof}

We associate to $W$ its twist by the automorphism $\sigma$. Namely, if
$\rho: U_q(\widehat{\g'}) \to \on{End}(W)$ is the action of
$U_q(\widehat{\g'})$ on $W$, we define the $\sigma$-twisted action of
$U_q(\widehat{\g'})$ on the same vector space by the formula
$$
\rho_\sigma(g) = \rho(\sigma(g)), \qquad g \in
U_q(\widehat{\g'}).
$$

\begin{lem}    \label{whsigma}
(1) The representations $(W,\rho)$ and $(W,\rho_\sigma)$ are
isomorphic, and there is a unique linear automorphism
$\wh\sigma: W\rightarrow W$, such that 
\begin{equation}    \label{rhosigma}
\rho(\sigma(g)) = \wh\sigma \rho(g) \wh\sigma^{-1}, \qquad \forall
g\in U_q(\widehat{\g'})
\end{equation}
and the restriction of $\wh\sigma$ to $W_{\Lambda(W)}$ is the
identity.

(2) The operator $\wh\sigma$ maps every $\g'$-weight (resp.
$\ell$-weight) subspace of $W$ corresponding to the $\g'$-weight
$\gamma$ (resp. monomial $M$ in $Y_{i,a}$) to a $\g'$-weight
(resp. $\ell$-weight) subspace corresponding to $\sigma(\gamma)$
(resp. $\sigma(M)$). In particular, it preserves the subspaces
$\wh{W}$ and $\wt{W}$ and their graded components corresponding to
$\sigma$-invariant $\g'$-weights and $\ell$-weights, respectively.
\end{lem}

\begin{proof}
(1) Drinfeld's classification of irreducible finite-dimensional
  representations of $U_q(\widehat{\g'})$ (see \cite{cp}) shows that
  such a representation $W$ is generated by its one-dimensional
  highest weight subspace $W_\Lambda$ and is determined by the
  eigenvalues of the Cartan--Drinfeld generators $h_{i,r}$ on it,
  which are recorded by the highest monomial of its $q$-character (see
  Section \ref{fdrep}).  Formulas \eqref{dact1} imply that $W_\Lambda$
  is still the highest weight subspace of $W$ under the twisted
  representation $\rho_\sigma$. The $\sigma$-invariance of the highest
  monomial of the representation $(W,\rho)$ implies that the highest
  monomial of $(W,\rho_\sigma)$ is the same. Therefore there is an
  isomorphism $\wh\sigma: W \to W$ intertwining the representations
  $(W,\rho)$ and $(W,\rho_\sigma)$ of $U_q(\widehat{\g'})$ and
  preserving the one-dimensional subspace $W_\Lambda$. It follows that
  $\wh\sigma$ satisfies formula \eqref{rhosigma}. By Schur's lemma,
  there is a unique such $\wh\sigma$ that is equal to the identity on
  $W_\Lambda$.

(2) Let $\{ X_i, i \in I' \}$ be either $\{ k_i, i \in I' \}$ or $\{
  h^{\on{ss}}_{i,r}, i \in I' \}$ with a fixed $r \neq 0$, where
  $h^{\on{ss}}_{i,r}$ is the semi-simplification of the action of
  $h_{i,r}$ on $W$. Suppose that $v$ is a joint eigenvector of $\{
  X_i, i \in I' \}$, in $W$ with the eigenvalues $\la_i, i \in
  I'$. Then
\begin{equation}    \label{Xi}
X_i \cdot \wh\sigma(v) = \wh\sigma (\wh\sigma^{-1} X_i \wh\sigma)\cdot
v = \wh\sigma X_{\sigma^{-1}(i)} \cdot v = \la_{\sigma^{-1}(i)}
\wh\sigma(v).
\end{equation}
This completes the proof.
\end{proof}

\begin{rem}
Let $W^{\wh\sigma}\subset W$ be the subspace of vectors fixed by
$\wh\sigma$. It contains $W_{\Lambda(W)}$ and is stable under the
action of the subalgebra $U_q(\widehat{\g'})^\sigma$ of
$\sigma$-invariant elements of $U_q(\widehat{\g'})$. However, in
general $W^{\wh\sigma}$ is not stable under the action of the Cartan
generators $k_i^{\pm 1}, i \in I'$ or the Cartan--Drinfeld generators,
and hence does not have a well-defined character or
$q$-character.  Besides, the invariant subalgebra
  $U_q(\widehat{\g'})^\sigma$ does not have a natural structure of
  Hopf algebra and has other deficiencies, as can be shown by an
  argument similar to that of \cite[Section 2.7]{H}.

This is why we consider instead the subspaces $\wh{W}$ and
$\wt{W}$. Their graded subspaces are preserved by $\wh\sigma$, and
hence $\wh{W}$ and $\wt{W}$ have well-defined character and
$q$-character, respectively.\qed
\end{rem}

\subsection{XXZ-type model associated to
  $U_q(\wh{\mathfrak{g}}')$}    \label{sigmala}

Consider the Borel subalgebra $U_q(\widehat{{\mathfrak b}}')$ of
$U_q(\widehat{\g'})$ and its category $\mathcal{O}^*$ which contains
the prefundamental representations $R_j^+(a)$, $R_j^-(a)$ ($j\in I'$,
$a\in\mathbb{C}^\times$) constructed in \cite{HJ}.

For every $j \in I'$, we have the $Q$-{\em operator}
$Q_j^\pm(z,u) = t_{R_j^\pm}(z,u)$ associated in \cite{FH} to the
prefundamental representation $R_j^\pm(z)$. This operator is a
formal power series in $z$, and it also depends on an element $u$ of
the Cartan subgroup $H'$ of the simply-connected Lie group $G'$
associated to $\g'$.

More generally, for any $V$ in the Grothendieck ring
$K_0(\mathcal{O}^*)$, we have the corresponding transfer-matrix
$t_V(z,u)$.

The transfer-matrices $t_V(z,u), V \in K_0(\mathcal{O}^*)$ (in
particular, operators $Q_j^\pm(z,u), j \in I'$) commute with each
other. These are the Hamiltonians of the XXZ-type model associated to
$U_q(\wh{\mathfrak{g}}')$. Every finite-dimensional representation $W$
of $U_q(\wh{\g'})$ decomposes into a direct sum of the generalized
eigenspaces of these Hamiltonians. In this paper we will focus on the
eigenvalues of these Hamiltonians and ignore the structure of the
corresponding Jordan blocks. Hence we introduce the {\em
  semi-simplification} $Q^{\pm, \on{ss}}_i(z,u)$ of $Q_i^\pm(z,u)$,
i.e. the unique diagonalizable operator on $W$ whose eigenspaces and
eigenvalues are the generalized eigenspaces and eigenvalues of
$Q_i^\pm(z,u)$. It follows that each $Q^{\on{ss},\pm}_i(z,u)$ can be
expressed as a polynomial in the original operator $Q_i^\pm(z,u)$, and
therefore $Q^{\pm, \on{ss}}_i(z,u)$ commutes with all $Q_j^\pm(z,u)$
and with all $Q^{\pm, \on{ss}}_j(z,u), j \in I'$.

Thus, for every $u \in H'$, we have a direct sum decomposition of $W$
into joint eigenspaces of the operators $Q^{\pm,\on{ss}}_j(z,u), j \in
I'$.  The following is proved in \cite[Theorem 5.9]{FH} and
\cite{FH4}.

\begin{thm}    \label{eigg}
  Let $W$ be a simple module over $U_q(\widehat{\g'})$. Then

(1) The eigenvalues of $Q^{+,\on{ss}}_j(z,u)$ on $W$ are
  polynomials in $z$, up to an overall function in $z$ depending only
  on $W$. Generically, the roots of these polynomials yield
  a solution of the BAE \eqref{bae gen} corresponding to $\g'$.

(2) Every joint eigenspace of $Q^{\pm, \on{ss}}_j(z,u), j \in I'$, on
  $W$ is a subspace of a $\g'$-weight subspace of $W$ corresponding to
  the $\g'$-weight $\Lambda(W) - \sum_{j \in I'} n_j \al_j$, where
  $n_j$ is the degree of the Baxter polynomial encoding the eigenvalue
  of $Q^{+,\on{ss}}_j(z,u)$.
\end{thm}

We call the polynomials in part (1) of this theorem the {\em Baxter
    polynomials}.

Part (1) of the theorem implies that the decomposition of $W$ into a
direct sum of joint eigenspaces of $Q^{\pm, \on{ss}}_j(z,u), j \in
I'$, is a refinement of the decomposition of $W$ into a direct sum of
its $\g'$-weight subspaces. In other words, for every $u \in H'$ and
$\ga \in P$ (the set of $\g'$-weights), the corresponding component
$W_\ga$ of $W$ is a direct sum
\begin{equation}    \label{decWga}
W_\ga = \bigoplus_{\la \in {\mc E}_\ga(u)} W_{\ga,\la}(u)
\end{equation}
where ${\mc E}_\ga(u)$ is the set of distinct joint eigenvalues of
$Q^{\pm, \on{ss}}_j(z,u), j \in I'$, on $W_\ga$, and $W_{\ga,\la}(u)$
is the eigenspace corresponding to $\la \in {\mc E}_\ga(u)$. Here
$\la$ denotes a collection $\{ \la_j^{\pm}(z,u), j \in I' \}$ of joint
eigenvalues of $Q^{\pm, \on{ss}}_j(z,u), j \in I'$

As shown in \cite[Proposition 5.5]{FH}, $Q_j^+(z,u)$ has a well-defined
$u \to 0$ limit, which is equal to
$$T_j(z) = \text{exp}\left( \sum_{m > 0} z^m
\frac{\wt{h}_{j,-m}}{[d_j]_q[m]_{q_j}} \right).$$ This is a generating
series of the Cartan--Drinfeld generators
\begin{equation}    \label{CD}
  \wt{h}_{j,-m} = \sum_{k\in I'} \wt{C}_{k,j}(q^m)h_{k,-m}.
\end{equation}
According to this formula, we can recover all Cartan--Drinfeld
generators $h_{k,-m}, k\in I', m > 0$ from the $T_j(z), j \in I'$.

It follows that the set ${\mc E}_\ga(u)$ has a well-defined $u \to 0$
limit, which we denote by ${\mc E}_\ga(0)$. Moreover, ${\mc E}_\ga(0)$
can be identified with the set of $\ell$-weights with the underlying
$\g'$-weight $\ga$. For each $\ell$-weight $\la \in {\mc E}_\ga(0)$,
the eigenspace $W_{\ga,\la}(0)$ is the corresponding $\ell$-weight
subspace of $W_\ga$.

\begin{rem} The $u \to 0$ limit of the $Q$-operator $Q_j^-(z,u)$
is equal to $(T_j(z))^{-1}$. This follows from a direct computation or
from the decomposition of $R_{i,a}^+\otimes R_{i,a}^-$ in
$K_0(\mathcal{O}^*)$. Indeed, it is equal to the class $[1]$ of the
trivial one-dimensional representation plus the classes of simple
representations whose highest weights are equal to positive linear
combinations of the simple roots, which implies that they do not
contribute to the $u\to 0$ limit.\qed
\end{rem}

\subsection{The invariant subspace}    \label{invsub}

Let us apply the automorphism $\sigma$ of the algebra $U_q(\wh{\g'})$
given by formula \eqref{dact1} to the transfer-matrices
$t_V(z,u)$. Recall that they are constructed by taking the trace of
$u$ times the universal $R$-matrix ${\mc R}$ of $U_q(\wh{\g'})$ over
$V$ (see, e.g. \cite{FH}).

\begin{lem}
    The automorphism $\sigma$ is an automorphism of the Hopf algebra
    structure on $U_q(\wh{\g'})$. Moreover, it preserves the universal
    $R$-matrix ${\mc R}$ of $U_q(\wh{\g'})$: $(\sigma \otimes
    \sigma)({\mc R}) = {\mc R}$.
\end{lem}

\begin{proof}
One checks directly that the comultiplication defined on
the Drinfeld--Jimbo generators is invariant under $\sigma$. For the
$i$th Drinfeld--Jimbo generators, where $i \in I'$, this is
obvious. For the the $0$th Drinfeld--Jimbo generators, we use a formula
expressing them in terms of the Drinfeld generators from
\cite[p. 393]{cp}.

The universal $R$-matrix of $U_q(\wh{\g'})$ corresponds to the
canonical element in the double of $U_q({\mathfrak b}_+)$ under the
identification of this double (modulo the Cartan subalgebra with
$U_q(\wh{\g'})$), see \cite{Dr}. It follows from its definition that
the automorphism $\sigma$ is compatible with the double structure and
therefore it sends the canonical element to itself. Hence it preserves
the universal $R$-matrix of $U_q(\wh{\g'})$.
\end{proof}

For any representation $V$ in the category ${\mc O}^*$, denote by
$\sigma^*(V)$ the twist of $V$ by the automorphism $\sigma$ defined as
in Section \ref{conj sect}. It is clear that $\sigma^*$ induces an
automorphism of the Grothendieck ring $K_0(\mathcal{O}^*)$. By
construction, $\sigma^*(R^\pm_{j,a}) \simeq R^\pm_{\sigma(j),a}$. Note
also that $\sigma$ defines an automorphism of the Cartan subgroup $H'$
of $G'$.

\begin{lem}    \label{autsigma}
    We have
\begin{equation}    \label{sigmau}
     \sigma^{-1}(t_V(z,u)) = t_{\sigma^*(V)}(z,\sigma(u)), \qquad
    \sigma^{-1}(Q^\pm_j(z,u)) = Q^\pm_{\sigma(j)}(z,\sigma(u)).
\end{equation}
\end{lem}

From now on, {\em we will assume that $u$ is $\sigma$-invariant},
and hence defines an element of $H = (H')^\sigma$, which is a Cartan
subgroup of the group $G = (G')^\sigma$. With this assumption,
formulas \eqref{sigmau} become
\begin{equation}    \label{sigmau1}
    \sigma^{-1}(t_V(z,u)) = t_{\sigma^*(V)}(z,u), \qquad
    \sigma^{-1}(Q^\pm_j(z,u)) = Q^\pm_{\sigma(j)}(z,u).
\end{equation}

Now we come to a key definition.

\begin{defn}
We define the {\em invariant subspace} of $W$ as
$$
W(u) := \bigoplus_{\ga \in P^\sigma} W_\ga(u),
$$
where $P^\sigma$ is the set of $\sigma$-invariant $\g'$-weights and
\begin{equation}    \label{Vu}
W_\ga(u) := \{w\in W_\ga \; | \; Q_j^{\pm,\on{ss}}(z,u).w =
Q_{\sigma(j)}^{\pm, \on{ss}}(z,u).w, \forall j\in I'\} \subset
\wh{W}_\ga.
\end{equation}
\end{defn}

Since the prefundamental representations (topologically) generate the
entire Grothendieck ring $K_0(\mathcal{O}^*)$, we have
$$W_\ga(u) = \{w\in W_\ga \; | \; t_V^{\on{ss}}(z,u).w =
t_{\sigma^*(V)}^{\on{ss}}(z,u).w, \forall [V]\in K_0(\mathcal{O}^*)\}.$$

Using formula \eqref{sigmau1} in the same way as in the proof of Lemma
\ref{whsigma}, we obtain the following result.

\begin{lem}    \label{inv}
  $W_\ga(u)$ is preserved by $\wh\sigma$.
\end{lem}

 We can also describe $W_\ga(u)$ as a span of joint
  eigenvectors of $Q_j^{\pm,\on{ss}}(z,u), j \in I'$.
  
  \begin{lem}    \label{spaneig}
    $W_\ga(u)$ is equal to the span of the joint eigenvectors of
    $Q_j^{\pm,\on{ss}}(z,u), j \in I'$, in $W_\ga$ with eigenvalues
    $\la_j^{\pm}(z,u), j \in I'$, such that
    $\la_j^{\pm}(z,u) = \la_{\sigma(j)}^{\pm}(z,u)$ for all $j \in
    I'$. In other words, in the notation of Section \ref{sigmala},
\begin{equation}    \label{decWu}
  W_\ga(u) = \bigoplus_{\la: \sigma(\la)=\la} W_{\ga,\la}(u).
\end{equation}
  \end{lem}

  \begin{proof}
    Every vector $v \in W_\ga$ can be written as a linear combination
    of joint eigenvectors of $Q_j^{\pm,\on{ss}}(z,u), j \in I'$:
    \begin{equation}    \label{decompv}
    v = \sum a_\la v_\la,
    \end{equation}
where $v_\la$ denotes an eigenvector with the collection of joint
eigenvalues $\la = \{ \la_j^{\pm}(z,u), j \in I' \}$. Now, if $v \in
W_\ga(u)$, then
    $$
    \sum a_\la \la_j^\pm(z,u) v_\la = \sum a_\la
    \la_{\sigma(j)}^\pm(z,u) v_\la,
    $$
    which implies that every collection $\la =
    \{ \la_j^{\pm}(z,u), j \in I' \}$ appearing in the decomposition
    \eqref{decompv} with a non-zero coefficient $a_\la$ must be
    $\sigma$-invariant, i.e. $\la_j^{\pm}(z,u) =
    \la_{\sigma(j)}^{\pm}(z,u)$ for all $j \in I'$.
  \end{proof}
    
Now, the discussion after Theorem \ref{eigg} implies

\begin{lem}    \label{Vga0}
We have $W_\ga(0) = \wt{W}_\ga$, i.e. $W_\ga(0)$ is spanned by the
$\ell$-weight vectors corresponding to the $\sigma$-invariant
$\ell$-weights whose underlying $\g'$-weight is $\gamma \in P^\sigma$.
\end{lem}

Remark \ref{not spanned} below shows, however, that we do not expect
the statement analogous to Lemma \ref{Vga0} to hold for $W_\ga(u)$
with $u \neq 0$.

Lemma \ref{inv} implies that $W(u)$ is preserved by $\wh\sigma$. Lemma
\ref{Vga0} implies that $W(0) = \wt{W}$. We also record the following
useful result.

  \begin{lem}    \label{invell}
Suppose that $v \in W$ is an eigenvector of $Q_j^{\pm,\on{ss}}(z,u), j
\in I'$ (resp. an $\ell$-weight vector) such that $\wh\sigma(v) = \mu
v$, where $\mu$ is a scalar (thus, $\mu$ is a $d$-th root of
unity). Then the set of eigenvalues of $Q_j^{\pm,\on{ss}}(z,u), j
\in I'$ on $v$ (resp. the $\ell$-weight of $v$) is $\sigma$-invariant.
  \end{lem}

  \begin{proof}
Let $\la^\pm_j(z,u), j \in I'$, be the eigenvalues of
$Q_j^{\pm,\on{ss}}(z,u), j \in I'$ on $v$. Using formulas
\eqref{sigmau1} and \eqref{rhosigma}, we obtain that
$$
\la^\pm_{\sigma(j)}(z,u) v = Q_{\sigma(j)}^{\pm,\on{ss}}(z,u) \cdot v = \wh\sigma
Q_j^{\pm,\on{ss}}(z,u) \wh\sigma^{-1}(v) $$ $$= \mu^{-1} \wh\sigma
Q_j^{\pm,\on{ss}}(z,u) \cdot v =
\mu^{-1} \la^\pm_j(z,u) \wh\sigma(v) = \la^\pm_j(z,u) v,
$$
for all $j \in I'$. The proof for $\ell$-weight vectors is similar.
\end{proof}

\begin{rem}    \label{not spanned}
In general, joint eigenspaces of the operators $Q^{\pm, \on{ss}}_j(z,u),
j \in I'$, in $W_\ga$ are {\em not} spanned by $\ell$-weight
vectors. In other words, these eigenspaces differ from their limits as
$u \to 0$. For example, let us take as $W$ the tensor product
$V_1\otimes V_a$ of two fundamental representations of
$U_q(\wh{sl}_2)$ with the highest weight monomial $Y_1 Y_a$. The
$\ell$-weight vectors in it are computed for example in \cite[Example
  3.3]{hinv}. Let $\{ w_+,w_- \}$ and $\{ v_+,v_- \}$ be bases of
weight vectors in $V_1$ and $V_a$, respectively, with $w_+$ and $v_+$
being the highest weight vectors. Then $w_-\otimes v_+$ is an
$\ell$-weight vector whose $\ell$-weight is $Y_{q^2}^{-1}Y_{a}$.

On the other hand, the action of $Q^+(z,u)$ on the $0$-weight subspace of
$V_1\otimes V_a$, which is spanned by $w_- \otimes v_+$ and $w_+
\otimes v_-$, can be computed following \cite[Example 7.8]{FH}. The
result is
$$\frac{T(z)}{1 - u^2} + \frac{zu^2 (q - q^{-1})}{(1
  - u^2)(1 - u^2 q^{-2})}f_1 T(z) f_0,$$
where $T(z)$ is the $u \to 0$ limit of $Q^+(z,u)$, which can be
expressed (see formula \eqref{CD} above) as a generating function of
the Cartan--Drinfeld elements, and $f_1$, $f_0$ are the
Drinfeld--Jimbo generators.

In particular, if we denote by $g(z)$ the eigenvalue of $T(z)$ on
$w_+\otimes v_+$, then we obtain
$$Q^+(z,u). (1 - u^2)(1 - u^2 q^{-2})(g(z))^{-1}(w_-\otimes v_+)$$ 
$$= (1 - zq^{-1} - u^2q^{-2} + q^{-1}zu^2)w_-\otimes
v_+ + zu^2 (1 - q^{-2}) w_+\otimes v_-.$$
Hence $w_-\otimes v_+$ is not an eigenvector of $Q^+(z,u)$ if $u\neq
0$.\qed
\end{rem}

\subsection{Eigenvalues of transfer-matrices on the invariant
  subspace}\label{eigenfd}

Since the operators $Q^{\pm, \on{ss}}_j(z,u), j \in I'$, commute with
each other, they have a well-defined joint spectrum on the invariant
subspace $W(u)$. Using Proposition \ref{foldedBAE} (the fact that the
system of folded BAE (\ref{bae gen1}) for $\g$ is obtained by folding
of the system of BAE (\ref{bae gen}) for $\g'$) and Theorem
\ref{eigg}, we can now link the eigenvalues of $Q^{\pm,
  \on{ss}}_j(z,u), j \in I'$, on $W(u)$ and solutions of the folded
BAE (\ref{bae gen1}) for $\g$.

\begin{thm}\label{BAElem}
Suppose that the Baxter polynomials encoding the joint eigenvalues of
$Q^{+, \on{ss}}_j(z,u)$, $j \in I'$, on a vector in $W(u)$ are
generic, so that their roots satisfy the BAE (\ref{bae gen}) for
$\g'$. Then, after the identification of the roots of these
polynomials corresponding to $Q^{+, \on{ss}}_j(z,u)$ and $Q^{+,
  \on{ss}}_{\sigma(j)}(z,u)$, we obtain a solution of the folded BAE
(\ref{bae gen1}).
\end{thm}

Next, the following result is proved in \cite[Theorem 5.11]{FH}:

\begin{thm}
The eigenvalues of the transfer-matrix $t_V(z,u)$, where $V$ is in
$\on{Rep} U_q(\wh{\g'})$, on a simple module $W$ can be expressed (up
to an overall factor) as the $q$-character of $V$, in which we replace
each $Y_{i,a}, i \in I$, by a ratio of the corresponding Baxter
polynomials.
\end{thm}

We now use this result to describe the eigenvalues of $t_V(z,u)$ on
$W(u) \subset W$.

\begin{thm}    \label{spectrafolded}
Let $V$ be a finite-dimensional representation of
$U_q(\wh{\g'})$. Then the transfer-matrix $t_V(z,u)$ preserves the
subspace $W(u)$ and its generalized eigenvalues are given (up to an
overall factor) by the folded $t$-character $^{\on{f}}\chi_t(V)$
(obtained by identifying $Y_{i,a}$ with $Y_{\sigma(i),a}$ in
$\chi_t(V)$) in which we further replace each $Y_{i,a}$ by a ratio of
the corresponding Baxter polynomials.
\end{thm}

This theorem gives us the sought-after link between the spectra of
commuting quantum Hamiltonians and solutions of the folded Bethe
Ansatz equations. Namely, the quantum Hamiltonians are the
transfer-matrices $t_V(z,u)$ of the XXZ-type model associated to
$U_q(\wh{\g'})$ (so the auxiliary spaces are representations of
$U_q(\wh{\g'})$, or more general objects of the corresponding category
${\mc O}^*$) but we restrict them to the invariant subspaces $W(u)$ of
the irreducible finite-dimensional representations $W$ of
$U_q(\wh{\g'})$ and consider the corresponding spectra (so the spaces
of states are the subspaces $W(u) \subset W$).

However, for non-simply laced $\g$ this answer does not quite define a
quantum integrable model since we have not yet described the invariant
subspaces $W(u)$ as representations of a Hopf algebra. In the next
subsection, we conjecture such a description. More precisely, we
conjecture that there is a certain distinguished subspace
$\ol{W}(u) \subset W(u)$ which is isomorphic to a representation of
the twisted quantum affine algebra $U_q({}^L\widehat{\g})$. If this is
true, then we can indeed describe the spaces of states as
representations of a Hopf algebra, and then we indeed obtain a quantum
integrable model.

\subsection{Defining the folded integrable model}

As before, let $W$ be a simple representation of $U_q(\widehat{\g'})$
whose highest monomial is $\sigma$-invariant, i.e. it is a monomial in
the elements $\wt{Y}_{i,a}$ defined in Section \ref{wtY}.

Recall that for a Lie algebra $\g$, we refer to the integral weights
of the Cartan subalgebra of $\g$ as $\g$-{\em weights}. We will use
the following simple fact demonstrated below in Section \ref{appear}:

\begin{lem}    \label{bij}
There is a natural isomorphism between the lattice $P^\sigma$
of $\sigma$-invariant $\g'$-weights and the lattice $^LP$ of
$^L\g$-weights.
\end{lem}

In particular, since $W(u)$ and $\wt{W} = W(0)$ are, by definition,
vector spaces graded by $P^\sigma$, we can view them as vector
spaces graded by the lattice $^LP$ of $^L\g$-weights.

\begin{conj}\hfill    \label{main conj bis}

\begin{itemize}

\item[(i)] For generic $u$, there is an isomorphism $W(u) \simeq
  \wt{W}$ of vector spaces graded by $^L\g$-weights.
 
 \item[(ii)] For generic $u$, there is a subspace $\overline{W}(u)$ of
   $W(u)$ which is stable under the operators
   $Q^{\pm,\on{ss}}_j(z,u)$ and
   isomorphic, as a vector space graded by ${}^L\g$-weights, to the
   vector space underlying a $U_q({}^L\widehat{\g})$-module $M(W)$
   (which does not depend on $u$).
\end{itemize}
\end{conj}

In Section \ref{ex} we explicitly verify this conjecture, as well as
Conjecture \ref{part3} below, in a number of examples. In Section
\ref{fundsec} we identify the representation $M(W)$, in terms of the
closely related Conjecture \ref{part3} below, for the simplest
$U_q(\widehat{\g'})$-modules with $\sigma$-invariant highest monomials
(modulo Conjecture \ref{positivity}).

\begin{rem}    \label{bfW}
A possible candidate for $\ol{W}(u)$ is the subspace
${\mb W}(u) := \bigoplus_{\ga \in P^\sigma} {\mb W}_\ga(u)$, where
\begin{equation}    \label{Vu1}
{\mb W}_\ga(u) := \{w\in W_\ga \; | \; Q_j^\pm(z,u).w =
Q_{\sigma(j)}^\pm(z,u).w, \forall j\in I'\} \subset
\wh{W}_\ga.
\end{equation}
The difference with formula \eqref{Vu} is that we now consider the
operators $Q_j^\pm(z,u)$, rather than their semi-simplifications. We
expect that for generic $u$, ${\mb W}(u)$ is isomorphic to ${\mb W}(0)
\subset W(0) = \wt{W}$, which is defined similarly in terms of
$T_j(z)$ (not their semi-simplifications).

Another candidate for $\ol{W}(u)$ is the subspace ${\mb W}^\sigma(u)
\subset {\mb W}(u)$ spanned by the $\sigma$-invariant generalized
eigenvectors.\qed
\end{rem}

Conjecture \ref{main conj bis} means that for any non-simply laced
simple Lie algebra $\mathfrak{g}$, there exists a {\em folded quantum
  integrable system} with the quantum Hamiltonians being the
transfer-matrices $t_V(z,u)$, corresponding to finite-dimensional
representations $V$ of $\on{Rep} U_t(\wh{\g'})$, or more general
objects of the corresponding category ${\mc O}^*$ (these are the
auxiliary spaces of this integrable model). These Hamiltonians act on
vector spaces that underlie finite-dimensional representations of
$U_t({}^L\widehat{\g})$ (these are the spaces of states of this
integrable model) and, according to Theorem \ref{spectrafolded}, the
spectra of $t_V(z,u)$ can be expressed in terms of the folded
$t$-character of $V$ and the corresponding Baxter
polynomials. Moreover, by Theorem \ref{BAElem}, the roots of these
Baxter polynomials correspond to solutions of the folded BAE (\ref{bae
  gen1}) associated to $\g$.

\begin{rem}    \label{remtriv}

   (1) In all examples we have studied so far, with $W$ a simple
  module, we found a simple $U_q({}^L\widehat{\g})$-module $M(W)$ that
  satisfies the statements of this conjecture.

  (2) Suppose that a weight subspace $\wh{W}_\ga, \ga \in P^\sigma$,
  of $\wh{W}$ is one-dimensional. Then the automorphism $\wh\sigma$
  from Lemma \ref{whsigma} must act on $\wh{W}_\ga$ as a non-zero
  scalar. By Lemma \ref{invell}, this implies that
    $\wh{W}_\ga = \wt{W}_\ga$ and $\wh{W}_\ga = W_\ga(u)$ for all
  $u$. Hence the statement of part (i) of the conjecture is clear
  for such weight subspaces.

  (3) In the statement of part (ii) of the conjecture, one might be
  tempted to replace $U_t({}^L\widehat{\g})$ with
  $U_t(\widehat{^L\g})$. Indeed, the statement involves a vector
  space $\ol{W}(u)$ graded by $^L\g$-weights, and therefore it could
  come from a representation of $U_t({}^L\g)$ (the quantum group of
  the simple finite-dimensional Lie algebra $^L\g$), which is a
  subalgebra of both $U_t({}^L\widehat{\g})$ and
  $U_t(\widehat{^L\g})$.

In fact, in the examples we have considered so far (see Section
\ref{ex}), the character of $\ol{W}(u)$ is not only the character of a
representation of $U_t({}^L\widehat{\g})$ but also the character of a
representation of $U_t(\widehat{^L\g}$). But the representations of
$U_t(\widehat{^L\g})$ that appear here are significantly less natural
than the representations of $U_t({}^L\widehat{\g})$. For instance, as
we illustrate below (see Sections \ref{ex2}--\ref{fifth}), in most
cases they are not simple, and often contain direct sums of copies of
the trivial representation as direct summands (even for
representations of small dimension). Even more importantly, if we
consider representations of $U_t(\widehat{^L\g})$, we can only
reproduce the ordinary character of its restriction to $U_t({}^L\g)$,
and {\em not} its $t$-character. On the other hand, in all examples we
consider below, we can reproduce the $t$-character of a
$U_t({}^L\widehat{\g})$-module $M(W)$ using interpolating
$(q,t)$-characters, see Conjecture \ref{part3} below.
     
(4) It would be interesting to write explicit formulas for the action
     of $U_t({}^L\widehat{\g})$ on $M(W)$, at least in the cases when
     $M(W) \simeq \wt{W}$.\qed
\end{rem}

We will give a more detailed description of the module $M(W)$, and
hence a more precise formulation of Conjecture \ref{main conj bis},
using the theory of interpolating $(q,t)$-characters which we will
recall (and refine) in Section \ref{intchar} below. Namely, we have
the following conjecture.

\begin{conj}    \label{part3}
The $t$-character $\chi_t(M(W))$ of the
$U_t\left({}^L\widehat{\g}\right)$-module $M(W)$ can be obtained via
the specialization map $\Pi_t$ from
  an element $X_{q,t} \in \ol{\mc K}_{q,t}(\g)$. In fact, $X_{q,t}$
  has the following three specializations:
$$\begin{xymatrix}{ & X_{q,t}
    \ar[dl]_{\overline{\Pi}_t}\ar[rd]^{\Pi_t'}\ar[d]^{\Pi_t} &
    \\ {}^{\on{f}}\chi_t(W) & \chi_t(M(W)) & {}^{\on{f}}\chi_t(W')&
   }
\end{xymatrix}$$
where 
\begin{itemize}

\item[] $^{\on{f}}\chi_t(W)\in \mathcal{K}_t^-(\g)$ is the folded $t$-character of 
the $U_t\left(\wh{\mathfrak{g}'}\right)$-module $W$, 
 
\item[] ${}^{\on{f}}\chi_t(W')\in \mathcal{K}_t^-({}^{L}\g)$ is the 
folded $t$-character of a $U_t\left(\wh{\left({}^L\mathfrak{g}\right)'}\right)$-module $W'$.
\end{itemize}

\end{conj}

\subsection{The subspace $\ol{W}(u)$}    \label{olW}

Here we discuss the question of how to describe the subspace
$\overline{W}(u)\subset W(u)$ in general.

Note that the fixed subspace $W^{\wh{\sigma}}\subset W$ is not
necessarily stable under the action of the operators 
  $Q_j^{\pm,\on{ss}}(z,u)$ in general. But let us introduce the
subspace $W^\sigma(u) \subset W^{\wh{\sigma}}$ spanned by all joint
eigenvectors of  $Q_j^{\pm,\on{ss}}(z,u)$, $j\in I'$, which
belong to $W^{\wh{\sigma}}$.

Recall the decomposition \eqref{decWu} of $W_\ga(u)$:
\begin{equation}    \label{decWu1}
  W_\ga(u) = \bigoplus_{\la: \sigma(\la)=\la} W_{\ga,\la}(u).
\end{equation}

\begin{lem}    \label{summ}
  We have
\begin{equation}    \label{Wsigmau}
W^\sigma(u) = \bigoplus_{\gamma \in P^\sigma} \bigoplus_{\la:
  \sigma(\la)=\la} W^\sigma_{\ga,\la}(u) \subset W(u)
\end{equation}
where $W^\sigma_{\ga,\la}(u)$ is the $\wh\sigma$-invariant subspace of
$W_{\ga,\la}(u)$.
\end{lem}

\begin{proof}
According to Theorem \ref{eigg},(1), every joint eigenvector $v$ of
$Q_j^{\pm,\on{ss}}(z,u)$, $j\in I'$, belongs to a weight subspace $W_\ga
\subset W, \gamma \in P$. By Lemma \ref{whsigma},(2) $\wh\sigma$ maps
$W_\gamma, \gamma \in P$, to $W_{\sigma(\gamma)}$. Therefore, if
$\wh\sigma(v) = v$, then $v \in W_\gamma, \gamma \in
P^\sigma$. Next, we have the decomposition \eqref{decWga}
  of $W_\ga, \ga \in P^\sigma$, into eigenspaces of
  $Q_j^{\pm,\on{ss}}(z,u)$. Applying formula \eqref{Xi} with $X_i =
  Q_i^{\pm,\on{ss}}(z,u)$, $j\in I'$, we obtain that if $v \in
  W_{\ga,\la}$, then $\wh\sigma(v) \in W_{\ga,\sigma(\la)}$. Hence, if
  $\wh\sigma(v) = v$, then $\sigma(\la)=\la$. It follows that the
  subspace $W_{\ga,\la}$ with $\ga
  \in P^\sigma$ and $\sigma(\la)=\la$ is preserved by
  $\wh\sigma$. Thus, $W_{\ga,\la}$ decomposes into a direct sum of
  eigenspaces of $\wh\sigma$, which are labeled by the $d$-th roots of
  unity. Denoting the invariant part (on which $\wh\sigma$ acts the
  identity) by $W^\sigma_{\ga,\la}(u)$, we obtain the
  decomposition \eqref{Wsigmau}.
\end{proof}

There is an analogue of the subspace $W^\sigma(u) \subset
  W(u)$ for $u=0$; namely, the subspace $\wt{W}^\sigma$ spanned by all
  $\wh\sigma$-invariant $\ell$-weight vectors in $W$. In the same way
  as in the proof of Lemma \ref{summ} one shows that
\begin{equation}    \label{wtWsigma}
\wt{W}^\sigma = \bigoplus_{\gamma \in P^\sigma}
\bigoplus_{M:\sigma(M)=M} \wt{W}^\sigma_{\gamma,M}
\subset \wt{W},
\end{equation}
where $\wt{W}^\sigma_{\gamma,M}$ is $\wh\sigma$-invariant part of the
$\ell$-weight subspace of $\wt{W}_\gamma, \gamma \in P^\sigma,$ whose
$\ell$-weight corresponds to a monomial $M$ (note that
$\wt{W}^\sigma_{\gamma,M}$ is preserved by $\wh\sigma$ if and only if
$M$ is a $\sigma$-invariant monomial, see Lemma
\ref{whsigma},(2)). Thus, $\wt{W}^\sigma$ is a subspace of
$\wt{W} = W(0)$.

We expect that $\overline{W}(u)$ is a subspace of $\wt{W}^\sigma$
(although they are not equal in general, as we can see from the
example in Section \ref{fifth} below). In addition to the examples
that will be presented in Section \ref{ex}, some supporting evidence
comes from the following result in the finite-type case.

Let $W$ be a simple finite-dimensional representation of
  $U_q(\mathfrak{g}')$ with a $\sigma$-invariant highest weight. Then,
  as in the affine case, $\sigma$ gives rise to an automorphism
  $\wh\sigma$ of $W$. Let $W^\sigma$ be the span of all weight vectors
  in $W$ which belong to $W^{\wh{\sigma}}$. In the same way as in
  the proof of Lemma \ref{summ} one shows that
$$
W^\sigma = \bigoplus_{\gamma \in P^\sigma} W^\sigma_\gamma,
$$
where $W^\sigma_\gamma$ is the subspace of
  $\wh\sigma$-invariant vectors in $W_\ga, \ga \in P^\sigma$.
Recall from Lemma \ref{bij} that we can view elements of the set
$P^\sigma$ of $\sigma$-invariant $\g'$-weight spaces as
$^{L}\mathfrak{g}$-weights. Thus, the character $\chi(W^\sigma)$ is a
linear combination of $^{L}\mathfrak{g}$-weights.

\begin{prop}    \label{virt}
$\chi(W^\sigma)$ is invariant under the action of the Weyl group of
  $^L\g$, and hence it is the character of a virtual representation of
  $U_q({}^{L}\mathfrak{g})$.
\end{prop}

\begin{proof}
 For $i \in I'$, denote by $U_i$ the $U_q(\sw_2)$
  subalgebra of $U_q(\g')$ generated by $e_i$, $f_i$, $k_i^{\pm 1}$.

Let us show that $\chi(W^\sigma)$ is invariant under the simple
  reflections of $^{L}\mathfrak{g}$ associated to the $\sigma$-orbits
  of $i\in I'$. This is clear if $\sigma(i) = i$ as $W^\sigma$ is
  stable under the action of $U_i$ and hence its character is
  invariant under the $i$th reflection of $^{L}\mathfrak{g}$.

Suppose $\sigma(i)\neq i$.  Since $\g'$ is not of type $A_{2n}$, the
subalgebra $U_i^\sigma$ of $U_q(\mathfrak{g}')$ generated by the
$U_q(\sw_2)$ subalgebras $U_{\sigma^k(i)}$, $1\leq k\leq d$, is
isomorphic to $U_q(\sw_2)^{\otimes d}$. To simplify notation, let us
assume that $d = 2$ (the proof for $d = 3$ is quite similar).

Viewed as a representation of $U_i^\sigma$, $W$ is semi-simple:
\begin{equation}\label{ssd}
  W = \bigoplus_j L_j, \qquad L_j = L_j^{(1)}\otimes L_j^{(2)},
\end{equation}
where $L_j^{(1)}$ and $L_j^{(2)}$ are simple representations of the
subalgebras $U_i$ and $U_{\sigma(i)}$, respectively.

If $x$ is a weight vector which belongs to $W^\sigma$ (i.e. the
weight of $x$ is $\sigma$-invariant), then so are the vectors
$\overline{e}.x$ and $\overline{f}.x$, where
$$
\overline{e} = e_ie_{\sigma(i)}, \qquad \overline{f} =
f_if_{\sigma(i)}.
$$
Hence $W^\sigma$ is stable under $\overline{e}$ and $\overline{f}$.

Let $\omega$ be a weight such that that the weight subspace
$(W^\sigma)_\omega$ is non-zero. Then $\omega(\alpha_i^\vee) =
\omega(\alpha_{\sigma(i)}^\vee) = m\in\mathbb{Z}$. Suppose that $m\geq
0$.  Using the decomposition (\ref{ssd}), we can write any weight
vector $v$ in $(W^\sigma)_\omega$ as $v = \sum_j v_j$, where
$$
v_j \in (L^{(1)}_j)_m \otimes (L^{(2)}_j)_m,
$$
$(L^{(1)}_j)_m$ and
$(L^{(2)}_j)_m$ being the weight subspaces corresponding to weight $m$
in $L^{(1)}_j$ and $L^{(2)}_j$, respectively. From representation
theory of $U_q(\sw_2)$ we know that these weight subspaces are
one-dimensional, so $v_j = v_j^{(1)} \otimes v_j^{(2)}$ is a pure
tensor. Moreover, we have
$$\ol{f}^m.(v_j^{(1)} \otimes v_j^{(2)}) = (f_i^m.v_j^{(1)})\otimes
(f_{\sigma(i)}^m.v^{(2)}_j)\neq 0.$$
Hence we obtain an injective linear map :
$$\ol{f}^m : (W^\sigma)_{\omega} \rightarrow (W^\sigma)_{\omega -
  m(\alpha_i + \alpha_{\sigma(i)})}$$
and so
$$\text{dim} (W^\sigma)_{\omega} \leq \text{dim} (W^\sigma)_{\omega -
  m(\alpha_i + \alpha_{\sigma(i)})}$$
if $m\geq 0$.

The opposite inequality is obtained by considering the action of
$\ol{e}^m$ on $(W^\sigma)_{\omega - m(\alpha_i +
  \alpha_{\sigma(i)})}$. This completes the proof.
\end{proof}

We expect that this virtual representation is an actual representation
of $U_q({}^{L}\mathfrak{g})$ (although it is not clear to us how to
construct the corresponding action of $U_q({}^{L}\mathfrak{g})$). In
Section \ref{N=1} we will show (assuming Conjecture
\ref{completeness}) that in the limit $q \to 1$ it is possible to construct
an explicit embedding of the irreducible representation of $^L\g$ with
the highest weight corresponding to that of $W$ (which is
$\sigma$-invariant by our assumption) into $W^\sigma$.

In the rest of this section, we discuss the corresponding
$QQ$-system.

\subsection{The $QQ$-system}

\newcommand{\Psib}{\mbox{\boldmath$\Psi$}}

The $QQ$-system (or $Q\wt{Q}$-system in the terminology of \cite{FH4})
of type $\wh{\mathfrak{g}'}$ reads
$$ \left[-\frac{\alpha_i'}{2}\right]Q_{i,aq^{-1}}\wt{Q}_{i,aq} -
\left[\frac{\alpha_i'}{2}\right] Q_{i,aq}\wt{Q}_{i,aq^{-1}} = \prod_{
j\neq i}Q_{j,a}^{-C_{j,i}'},$$ where $C'$ (resp. $\alpha_i'$) is the
Cartan matrix (resp. a simple root) of $\mathfrak{g}'$. It was written
in \cite{MRV} in the context of affine opers, and established
\cite{FH4} as a system of relations in $K_0(\mathcal{O}^*)$, with the
$Q_{i,a}$ and $\wt{Q}_{i,a}$ being the properly normalized classes of
certain simple modules; namely, the module $R_{i,a}^+$ for $Q_{i,a}$
and another module, which we denote by $X_{i,aq_i^{-2}}'$, divided by
an invertible element which does not depend on $a$ (recall that here we
consider the category ${\mathcal O}^*$, see \cite[Remark 3.2
(iii)]{FH4}). The $\left[-\frac{\alpha_i'}{2}\right]$ are
classes of certain one-dimensional representations in $\mathcal{O}^*$.
Then on $W(u)$, the eigenvalues of $Q_{j,a}$ are
identified with those of $Q_{\sigma(j),a}$ and the eigenvalues of
$\wt{Q}_{i,a}$ are identified with the eigenvalues of
$\wt{Q}_{\sigma(j),a}$. Hence we obtain the following result.

\begin{thm}    \label{QQfolded}
The following $QQ$-system holds on the
invariant subspace $W(u)$:
\begin{equation} \label{QtiQ}
\left[-\frac{\alpha_i}{2}\right]Q_{i,aq^{-1}}\wt{Q}_{i,aq}
- \left[\frac{\alpha_i}{2}\right] Q_{i,aq}\wt{Q}_{i,aq^{-1}} =
\prod_{ j\neq i}Q_{j,a}^{-C_{j,i}},
\end{equation}
where $C$ is the Cartan matrix of $\mathfrak{g}$, $Q_{i,a}$
(resp. $\wt{Q}_{i,a}$) is, up
to an invertible constant, the $Q$-operator $Q_{i,a}^+$ (resp. the
transfer-matrix associated to $X_{i,aq_i^{-2}}$).
\end{thm}

Note that for a non-simply laced $\g$, this is {\em neither} the
$QQ$-system of type $\wh{\mathfrak{g}}$ (as defined in \cite{MRV2,
  FH4}) {\em nor} the $QQ$-system of the twisted type
$\wh{\mathfrak{g}}^\vee$ conjectured in \cite[Section 3.3]{FH4}. But
this system is equivalent to the $QQ$-system obtained in \cite{FKSZ}
in the context of Miura $(G,q)$-opers (we note that a Yangian
version of this system first appeared in the work of Mukhin and
Varchenko \cite{MV1,MV2}). We will call it the {\em folded
  $QQ$-system} associated to $\g$.

According to Theorem \ref{QQfolded}, the spectra of the Hamiltonians
of the folded integrable model introduced in this section give rise to
solutions of the folded $QQ$-system. This is in agreement with Theorem
\ref{BAElem} because, as shown in \cite{FKSZ}, under a genericity
condition there is a bijection between solutions of the folded
$QQ$-system and solutions of the folded BAE \eqref{bae gen1}.

\begin{rem}

(1) Recall that each character $\omega\in H^*$ of the commutative
  group $H$ gives a one-dimensional representation
  $U_q(\wh{\mathfrak{b}})$-module $[\omega]$ which is in
  $\mathcal{O}$ and in $\mathcal{O}^*$.  We obtain a subring
  $K_0(H)\subset K_0(\mathcal{O})$ of representations whose simple
  constituents are of this form. These are called constant elements as
  the associated transfer-matrices are constant (independent of the
  spectral parameter). The ordinary
  character of a representation in $\mathcal{O}$ can be viewed as an
  element of $K_0(H)$.

(2) According to \cite{FH4}, the precise relation between the solutions of
$Q\wt{Q}$-system \eqref{QtiQ} and elements of $K_0(\mathcal{O})$ is as
follows. The variables $Q_{i,z}$ and $\wt{Q}_{i,z}$ correspond to the
classes $[L_{i,z}^+]$ and $[X_{i,zq_i^{-2}}]$, respectively,
renormalized by constant invertible elements of $K_0(\mathcal{O})$.

However, Keyu Wang has pointed out that the proof of \cite[Lemma
  4.11]{FH4} is incomplete. In that lemma a formula $\chi(X_{i,a}) =
(1 - [-\alpha_i])^{-1}\prod_{j\neq i}\chi_j^{-C_{j,i}}$ for the
character of $X_{i,a}$ was given in terms of the characters $\chi_i =
\chi(L_{i,a}^+)$ of the prefundamental representations, and the
normalization used in \cite{FH4} (see the preceding paragraph) was
based on this formula.  Although we believe that this character
formula is correct (and so the normalization used in \cite{FH4} is
correct), at the moment we do not have a proof of this character
formula. Hence we propose to slightly modify formulas (3.1) and (3.2)
in \cite{FH4} for the normalizing factors, so we can avoid relying on
this character formula. We explain this normalization in the next
subsection.

(3) According to Theorem \ref{QQfolded} and to Section \ref{eigenfd},
eigenvectors in the subspace $W(u)$ give solutions of the folded
$QQ$-system (or, under some genericity condition, of the folded BAE
\eqref{bae gen1}). However, we then restrict to a smaller subspace
$\overline{W}(u)$ of $W(u)$, which means that we only take a subset of
these solutions (we will see in the example in Section \ref{fifth}
below that $\overline{W}(u)$ can in fact be strictly smaller than
$W(u)$). This is likely related to the fact that taking various limits
of elements of the deformed $\mathcal{W}$-algebras, or various
specializations of the corresponding interpolating $(q,t)$-characters
(as discussed below in Section \ref{intchar}), may very well have
different numbers of monomials.  Hence it is natural to expect that
there is a characterization of the solutions corresponding to the
eigenvectors that belong to the subspace $\overline{W}(u)$ through a
similar kind of interpolation, which would enable us to tell which
solutions of the folded $QQ$-system (and the folded BAE) correspond to
the representation $M(W)$ for the Langlands dual quantum affine and
which do not. At the moment, this remains as an open question.\qed
\end{rem}

\subsection{Normalization of $Q_{i,z}$ and $\wt{Q}_{i,z}$}

\begin{prop}
The functions $$Q_{i,z} = [L_{i,z}^+]/\chi_i, \qquad \wt{Q}_{i,z} =
[X_{i,zq_i^{-2}}]/\wt{\chi}_i,$$
where
$$\chi_i = \chi(L_{i,z}^+)\text{ and }\wt{\chi}_i =
\chi(X_{i,z})\left(\left[\frac{\alpha_i}{2}\right] -
\left[-\frac{\alpha_i}{2}\right]\right),$$
solve the $Q\wt{Q}$-system (in the notation of \cite{FH4}).
\end{prop}

\begin{rem}
Here, in contrast to \cite{FH4}, we do not specify a precise relation
between the $\wt{\chi}_i$ and the $\chi_i$.  In addition, for
symmetry, we have chosen to renormalize $[L_{i,z}^+]$ by dividing by
$\chi_i$, although in \cite{FH4} $Q_{i,z}$ is $[L_{i,z}^+]$ without
renormalization. Both normalizations are compatible, as the
$Q\wt{Q}$-system written with the normalization in \cite{FH4} implies
that the variables in the Proposition above satisfy also the
$Q\wt{Q}$-system.\qed
\end{rem}

\begin{proof} We give the proof of the Proposition for the category
  $\mathcal{O}$. The analogous result in $K_0(\mathcal{O}^*)$ is
  obtained in the same way.

We establish the following $q$-character formula
\begin{equation}\label{qcharform}\chi_q(X_{i,1}) = [\wt{\Psib}_{i,1}]\chi_{i,1} \chi(X_{i,1})(1 - [\alpha_i])\end{equation}
where $\wt{\Psib}_{i,1}$ is the highest $\ell$-weight of $X_{i,1}$ and $\chi_{i,1} = \sum_{m\geq
0}(A_{i,1}A_{i,q_i^{-2}}\cdots A_{i,q_i^{-2m}})^{-1}$. This formula implies the $Q\wt{Q}$-system in $K_0(\mathcal{O})$ by the arguments given in \cite{FH4}.

The proof of (\ref{qcharform}) is based on the results in \cite{FH4}, except that we do
not use the complete proof of \cite[Lemma 4.11]{FH4}, but only \cite[Lemma 4.10]{FH4}. Indeed, it implies only
$$\chi_q(X_{i,1}) = [\wt{\Psib}_{i,1}]\sum_{m\geq
0}(A_{i,1}A_{i,q_i^{-2}}\cdots A_{i,q_i^{-2m}})^{-1}\chi(m),$$ where for $m\geq 0$,
$\chi(m)\in K_0(H)$ is a constant element. It suffices to prove that
$\chi(m)$ do not depend on $m$.

To do this, consider for $r\geq 0$ the representation $X_{i,1}^{(r)} =
L(\wt{\Psib}_{i,1} \Psib_{i,q_i^{-2r}})$. Then we obtain from \cite[Theorem
8.1]{hlast} that
$$
\chi_q(X_{i,1}^{(r)}) = \wt{\Psib}_{i,1}
\Psib_{i,q_i^{-2r}} \prod_{j\neq i}\chi_j^{-C_{i,j}}\sum_{0\leq m\leq
r}(A_{i,1}A_{i,q_i^{-2}}\cdots A_{i,q_i^{-2m}})^{-1}.
$$
Hence the class
of $X_{i,1}\otimes L_{i,q_i^{-2r}}^+$ can be decomposed as a sum
$$[X_{i,1}\otimes L_{i,q_i^{-2r}}^+] = [X_{i,1}^{(r)}]\chi_r' + \sum_{\Psib'} n_{\Psib'}[L(\Psib')]$$
where $\chi_r'\in K_0(H)$ is an invertible constant, the $n_{\Psib'}$ are positive integers and the
$\ell$-weight $\Psib'$ which occur in the sum are, up to a constant, of the
form
$$\wt{\Psib}_{i,1} \Psib_{i,q_i^{-2r}}A_{i,1}^{-1}\cdots
A_{i,q_i^{-2m}}^{-1}$$
for some $m \geq 0$. We claim that only $\ell$-weights $\Psib'$ with $m\geq r$ can occur, that is the class of
$L(\wt{\Psib}_{i,1} \Psib_{i,q_i^{-2r}}A_{i,1}^{-1}\cdots
A_{i,q_i^{-2m}}^{-1})$ cannot occur in the decomposition for $m < r$.

Indeed, for $m < r$, $A_{i,q_i^{-2m}}^{-1}$ is a factor of one of the $\ell$-weight of the simple module $L(\wt{\Psib}_{i,1} \Psib_{i,q_i^{-2r}}A_{i,1}^{-1}\cdots
A_{i,q_i^{-2m}}^{-1})$.
This follows from an $\sw_2$-reduction, as an elementary analysis shows that $A_{q^{-2m}}^{-1}$ occurs in
an $\ell$-weight of the representation in the $\sw_2$-case
$$L(\Psib_{q^2}\Psib_{q^{-2m}}^{-1}\Psib_{q^{-2m-2}}^{-1}\Psib_{q^{-2r}})
\simeq L(\Psib_{q^2}\Psib_{q^{-2m}}^{-1})\otimes L(\Psib_{q^{-2m-2}}^{-1}\Psib_{q^{-2r}}),$$
which is the tensor product of $L(\Psib_{q^2}\Psib_{q^{-2m}}^{-1})$ evaluation representation of a Verma module and
$L(\Psib_{q^{-2m-2}}^{-1}\Psib_{q^{-2r}})$ finite-dimensional representation. So, if this representation appeared in the decomposition, then
all its $\ell$-weights would be occur in the $q$-character of $X_{i,1}\otimes L_{i,q_i^{-2r}}^+$.
In particular, $A_{i,q_i^{-2m}}^{-2}$ would be the factor of one of
the $\ell$-weights of $X_{i,1}\otimes L_{i,q_i^{-2r}}^+$, which is a
contradiction.

Now, as $\chi_q(X_{i,1}\otimes L_{i,q_i^{-2r}}^+) = \chi_q(X_{i,1})[\Psib_{i,q_i^{-2r}}^+]\chi_i$, by identifying for $m\leq r$ the coefficients
of $(A_{i,1}A_{i,q_i^{-2}}\cdots A_{i,q_i^{-2m}})^{-1}$, we obtain :
$$\chi_r'\prod_{j\neq i}\chi_j^{-C_{i,j}} =\chi(m) \chi_i.$$
This implies that all the $\chi(m)$ are equal and
we obtain the $q$-character formula (\ref{qcharform}).
\end{proof}

\section{Interpolating $(q,t)$-characters}    \label{intchar}

Our approach to the folded quantum integrable systems associated to
quantum affine algebras, as formulated in Conjectures \ref{main conj
  bis} and \ref{part3}, involves the interpolating $(q,t)$-characters
introduced in \cite{FH1} as a tool for the study of a certain
Langlands duality between finite-dimensional representations of
quantum affine algebras. In this section we extend and refine the
definition of the interpolating $(q,t)$-characters from
\cite{FH1}. They are defined as elements of the refined Grothendieck
ring $\overline{\mathcal{K}}_{q,t}(\g)$, which is defined in this
section (Theorem \ref{refined}). It turns out that they have $5$
interesting specializations corresponding to various $q$- and
$t$-characters. Moreover, the interpolating $(q,t)$-characters may be
viewed as commutative analogues of elements of the deformed ${\mc
  W}$-algebra ${\mathbf W}_{q,t}(\g)$.

\subsection{Notation for monomials}\label{wtY}

For $j \in I'$, a node of the Dynkin diagram of $\g'$, let us set

\begin{equation}\label{wtY1}
 \notag \wt{Y}_{j,a} = \begin{cases} Y_{j,a}&\text{ if
      $\sigma(j) = j$,}\\  Y_{j,a}Y_{\sigma(j),a}&\text{ if
      $\sigma^2(j)= j$ and $\sigma(j) \neq j$,}
    \\ \notag Y_{j,a}Y_{\sigma(j),a}Y_{\sigma^2(j),a}&\text{
      if $\sigma^3(j) = j$ and $\sigma(j)\neq j$.} \end{cases}\end{equation}

For $i\in I$, a node of the Dynkin diagram of $\g$, let us set
\begin{equation}    \label{barY}
  \overline{Y}_{i,a} = Y_{i,a}^{ 1 + d - \rr_i},
\end{equation}
\begin{align} W_{i,a} = \begin{cases} Y_{i,a}&\text{ if $\rr_i = d$,}
\\ Y_{i,aq^{-1}}Y_{i,aq}&\text{ if $\rr_i =d - 1$,} \\Y_{i,aq^{-2}}Y_{i,a}Y_{i,aq^2}&\text{ if $\rr_i = d - 2$,} \end{cases}\end{align}
\begin{align} Z_{i,a^{1 + d - \rr_i}} = \begin{cases} Y_{i,a}&\text{ if $\rr_i = d$,}
\\ Y_{i,a}Y_{i,\epsilon^2 a}&\text{ if $\rr_i = d - 1$,}
\\Y_{i,a}Y_{i,\epsilon^2 a}Y_{i,\epsilon^4 a}&\text{ if $\rr_i = d -
  2$.} \end{cases}\end{align}
Note that $d_i$ can equal $d-1$ (resp. $d-2$) only if $d=2$
(resp. $d=3$), and recall that $\epsilon = e^{i\pi/d}$.

Note that $L(W_{i,a})$ is a KR module over $U_q(\wh{\mathfrak{g}})$.
It is a fundamental representation if $\rr_i = d$ (in particular, for
simply-laced types, these representations are always fundamental).

\subsection{Polynomial rings and specialization homomorphisms}

Now we recall the definition of the ring of interpolating
$(q,t)$-characters from \cite{FH1} and then define a refined version
of this ring. We start with some preliminary definitions.

Consider the ring 
$$\mathcal{Y}_{q,t} = \mathbb{Z}[W_{i,a}^{\pm 1}, \alpha Y_{i,a}^{\pm 1}, \alpha]_{i\in I, a\in q^{\mathbb{Z}}t^{\mathbb{Z}}}\subset \mathbb{Z}[\alpha, Y_{i,a}^{\pm 1}]_{i\in I, a\in q^{\mathbb{Z}}t^{\mathbb{Z}}}.$$ 
where $\alpha$ is an indeterminate.

\begin{rem} For simply-laced types, $\mathcal{Y}_{q,t}$ is just
  $\mathbb{Z}[Y_{i,a}^{\pm 1}, \alpha]_{i\in I, a\in
    q^{\mathbb{Z}}t^{\mathbb{Z}}}$.\qed
\end{rem}

The ring of interpolating $(q,t)$-characters
$\mathcal{K}_{q,t}(\mathfrak{g})$ was defined in \cite{FH1} as a
subring of a quotient $\wt{\mathcal{Y}}_{q,t}$ of $\mathcal{Y}_{q,t}$. This quotient is defined from the specialization maps 
$$\Pi_q = \Pi_{t = 1, \alpha = 1} : \mathcal{Y}_{q,t} \rightarrow \mathbb{Z}[Y_{i,a}^{\pm 1}]_{i\in I, a\in q^{\mathbb{Z}}},$$  
$$\Pi_t = \Pi_{q = \epsilon, \alpha = 0} : \mathcal{Y}_{q,t} \rightarrow \mathbb{Z}[Z_{i,a}^{\pm 1}]_{i\in I, a\in t^{\mathbb{Z}}\epsilon^{\mathbb{Z}}},$$ 
by the formula
$$\wt{\mathcal{Y}}_{q,t} = \mathcal{Y}_{q,t}/(\text{Ker}(\Pi_q)\cap \text{Ker}(\Pi_t)).$$ 

For our purposes, we also want to use the additional specialization homomorphisms 
$$\overline{\Pi}_t = \Pi_{q = 1, \alpha = d} : \mathcal{Y}_{q,t} \rightarrow \mathbb{Z}[ Y_{i,a}^{\pm 1}]_{i\in I, a\in t^{\mathbb{Z}}},$$ 
$$\overline{\Pi}_q = \Pi_{t = 1, \alpha = 0} : \mathcal{Y}_{q,t}
\rightarrow \mathbb{Z}[ W_{i,a}^{\pm 1}]_{i\in I, a\in
  q^{\mathbb{Z}}}.$$ These specialization homomorphisms are well-defined on
$\mathcal{Y}_{q,t}$, but $\overline{\Pi}_t$ does not descend to
$\wt{\mathcal{Y}}_{q,t}$ if $d>1$ (indeed, $\alpha^2 - \alpha$
projects onto $0$ in $\wt{\mathcal{Y}}_{q,t}$, but
$\overline{\Pi}_t(\alpha^2 - \alpha) = d^2 - d$). For this reason, we will
work with an intermediate quotient
$$\overline{\mathcal{Y}}_{q,t} =
\mathcal{Y}_{q,t}/(\text{Ker}(\Pi_q)\cap \text{Ker}(\Pi_t)\cap
\text{Ker}(\overline{\Pi}_q) \cap \text{Ker}(\overline{\Pi}_t))$$
for which we have surjective ring homomorphisms
$$\mathcal{Y}_{q,t} \twoheadrightarrow \overline{\mathcal{Y}}_{q,t}
\twoheadrightarrow \wt{\mathcal{Y}}_{q,t}.$$

 \begin{rem}\label{remqt} (1) The interpolating $(q,t)$-characters are
    defined below as elements of a subalgebra of the commutative
    algebra $\overline{\mathcal{Y}}_{q,t}$. Their purpose is to
    imitate the properties of elements of the non-commutative deformed
    ${\mc W}$-algebra ${\mathbf W}_{q,t}(\g)$. In particular, the
    variable $\alpha$ is introduced in order to imitate the behavior
    of the rational functions like \eqref{rat fn} arising in the
    formulas for elements of the deformed ${\mc W}$-algebra ${\mathbf
      W}_{q,t}(\g)$ such as \eqref{t1}. This is why in the above
    specializations we set $\alpha$ equal to $1, 0$, or $d$ depending
    on the situation.
		
		(2) There is a polynomial $P(\alpha)$ in $\alpha$ so that $\alpha^d P(\alpha)$ is equal to $\alpha$ in the quotient 
		$\overline{\mathcal{Y}}_{q,t}$ (for example, using a Lagrange interpolating polynomial). Hence, for an arbitrary monomial
		$m$ in the variables $Y_{i,a}^{\pm 1}$, as $\alpha^N m$ is in $\mathcal{Y}_{q,t}$ for $N$ large enough, $\alpha m$
		makes sense in $\overline{\mathcal{Y}}_{q,t}$.
		\qed
  \end{rem}

		Let us recall some terminology from \cite{FH1}. By a
  monomial in $\wt{\mathcal{Y}}_{q,t}$ we will understand an element
  $m$ of the form $P(\alpha)M$, where $P(\alpha)$ is a polynomial in $\alpha$
  and $M$ is a monomial in the $Y_{j,a}^{\pm 1}$.  Note that a
  monomial in $\wt{\mathcal{Y}}_{q,t}$ may be written in various ways as for example $\alpha
  Y_{i,a} = \alpha Y_{i,at}$ and $(1 - \alpha)Y_{i,aq^{2d}} = (1 -
  \alpha)Y_{i,a}$.  A monomial is said to be $i$-{\em dominant} if it can be
  written by using only the variables $\alpha$, $Y_{i,a}$ and
  $Y_{j,a}^{\pm 1}$, where $j\neq i$. Let $B_i$ be the set of
  $i$-dominant monomials and for $J\subset I$, let $B_J = \cap_{j\in
    J} B_j$. Finally, let $B = B_I$ be the set of {\em dominant
    monomials}.

  We will use the analogous definition of $i$-dominant
  (resp. dominant) monomials in $\overline{\mathcal{Y}}_{q,t}$.
		
		\begin{lem}    \label{idom}
		A monomial in $\overline{\mathcal{Y}}_{q,t}$ is $i$-dominant if and only if 
		its specializations 
		under $\Pi_q$, $\overline{\Pi}_t$ and $\Pi_t$ are all
                $i$-dominant.
		\end{lem}
		
		\begin{proof}
		The direct implication is clear. For the converse, let 
		us first consider a monomial $m$ in $\overline{\mathcal{Y}}_{q,t}$ which is a product of various $Y_{i,a}^{\pm 1}$, $a\in q^{\mathbb{Z}}t^{\mathbb{Z}}$ (for this question, we may discard the other variables $Y_{j,a}^{\pm 1}$ with $j\neq i$). Suppose $m$ specializes to $i$-dominant monomials 
$$m_1 = \prod_{s\in\mathbb{Z}}Y_{i,q^s}^{u_{i,s}(m_1)}, \qquad
m_2 =
\prod_{s\in\mathbb{Z},\epsilon'\in\epsilon^{\mathbb{Z}}}Y_{i,\epsilon'
  t^s}^{u_{i,\epsilon',s}(m_2)}, \qquad m_3 = \prod_{s\in\mathbb{Z}}Y_{i,t^s}^{u_{i,s}(m_3)}$$ 
for the respective specializations $t = 1$, $q = \epsilon$, $q=1$. 
Then the sum 
$$u_i = \sum_{s\in\mathbb{Z}} u_{i,s}(m_1)  = \sum_{s\in\mathbb{Z},\epsilon'\in\epsilon^{\mathbb{Z}}}u_{i,\epsilon',s}(m_2 ) = 
\sum_{s\in\mathbb{Z}} u_{i,s}(m_3) \geq 0$$ 
of the various powers of the various variables is the same for $m_1$,
$m_2$ and $m_3$. We also have for any $\epsilon' = \epsilon^R$, $R\in\mathbb{Z}$, the relation
$$\sum_{s\in\mathbb{Z}} u_{i,\epsilon',s}(m_2) = \sum_{s\in R + 2 d \mathbb{Z}}u_{i,s}(m_1).$$
Hence we can construct an $i$-dominant monomial 
$M$ in $\overline{\mathcal{Y}}_{q,t}$ so that its specializations at $t = 1$ and $q = \epsilon$ are $m_1$ and $m_2$ respectively. 
Such a monomial is not unique, but, for example, the powers in the variables in $M$ can be defined inductively by 
a standard combinatorial algorithm, starting from the $r_0$, $s_0$ where $r_0$ (resp. $s_0$) is minimal so that $u_{i,r_0}(m_1)\neq 0$ (resp. $u_{i,s_0}(m_3)\neq 0$). Then $M$ specializes necessarily to $m_3$ at $q = 1$ as for any $s\in\mathbb{Z}$ 
$$u_{i,s}(m_3) = \sum_{\epsilon'\in\epsilon^{\mathbb{Z}}} u_{i,s,\epsilon'}(m_2).$$
In particular $m$ gets identified with $M$ in $\wt{\mathcal{Y}}_{q,t}$ and so $m$ is $i$-dominant. 
The case when the monomial has a factor depending of $\alpha$ is
treated in a similar way.
		\end{proof}

		\begin{rem} The analogous statement is not true in
                  $\mathcal{Y}_{q,t}$ for non-simply laced types (with $r > 1$). 
		For example, $Y_{i,1}Y_{i,t}^{-1}Y_{i,tq^d}$ is not $i$-dominant, but specializes at $t = 1$, $q = \epsilon$, $q = 1$ 
		respectively to $Y_{i,q^d}$, $Y_{i,1}$, $Y_{i,1}$ which are all $i$-dominant. 
		But in $\overline{\mathcal{Y}}_{q,t}$ this monomial gets
		identified with $Y_{i,q^d}$, and so it is
                $i$-dominant, in accordance with
                Lemma \ref{idom}.\qed
                \end{rem}

\subsection{Definition of interpolating $(q,t)$-characters}

Recall the definition of the ring $\mathcal{K}_{q,t}$ of {\em
  interpolating $(q,t)$-characters} from \cite{FH1}. It is defined as
the intersection of subrings $\mathcal{K}_{i,q,t}\subset
\wt{\mathcal{Y}}_{q,t}$, $i\in I$, by analogy with the
characterization of $\on{Im} \chi_q$ as the intersection of the
subrings $\mathfrak{K}_{i,q}$ (see formulas \eqref{Imchiq} and
\eqref{keri}) as well as the definition of ${\mathbf
  W}_{q,t}(\g)$ as the intersection of the kernels of the screening
operators in \cite{FR:w} and Section \ref{defWqt} above.

We will need the following analogues of the generating series $A_i(z)$
in the deformed ${\mc W}$-algebra ${\mathbf W}_{q,t}(\g)$ given by
formula \eqref{Ai}:
\begin{equation}\label{wta}
  \wt{A}_{i,a} = Y_{i,a(q_it)^{-1}}Y_{i,aq_it}\times\prod_{j\in I,
  C_{j,i} = -1}Y_{j,a}^{-1}\times \prod_{j\in I, C_{j,i} =
  -2}Y_{j,aq^{-1}}^{-1}Y_{j,aq}^{-1} \times\prod_{j\in I, C_{j,i} =
  -3}Y_{j,aq^{-2}}^{-1}Y_{j,a}^{-1}Y_{j,aq^2}^{-1}.\end{equation}

\begin{defn} Let $\mathcal{K}_{i,q,t}$ be the subring of
$\wt{\mathcal{Y}}_{q,t}$ generated by the variables $\alpha$,
$W_{j,a}^{\pm 1}$, $\alpha Y_{j,a}^{\pm 1}$ ($j\neq i$, $a\in q^{\mathbb{Z}}t^{\mathbb{Z}}$), the $\alpha Y_{i,a}(1 +
\wt{A}_{i,aq_it}^{-1})$ ($a\in q^{\mathbb{Z}}t^{\mathbb{Z}}$) and
\begin{align} W_{i,a} \times \begin{cases} (1 + \wt{A}_{i,aq_it}^{-1})&\text{ if $\rr_i = d$,}
\\ (1 + \alpha \wt{A}_{i,aq^2t}^{-1} + \wt{A}_{i,aq^2t}^{-1} \wt{A}_{i,a t}^{-1})&\text{ if $\rr_i =d - 1$,} 
\\ (1 + \alpha \wt{A}_{i,aq^3 t}^{-1} + \alpha \wt{A}_{i,aq^3 t}^{-1}\wt{A}_{i,aq t}^{-1} + \wt{A}_{i,aq^3 t}^{-1}\wt{A}_{i,aq t}^{-1}\wt{A}_{i,aq^{-1} t}^{-1} )&\text{ if $\rr_i = d - 2$,} \end{cases}\end{align}
where $a\in q^{\mathbb{Z}}t^{\mathbb{Z}}$.

Following \cite{FH1}, we define the {\em ring $\mathcal{K}_{q,t}(\g)$ of
  interpolating $(q,t)$-characters associated to} $\g$ as the
following ring intersection:
$$\mathcal{K}_{q,t}(\g) := \bigcap_{i\in I}\mathcal{K}_{i,q,t}\subset
\wt{\mathcal{Y}}_{q,t}.$$
\end{defn}

\begin{rem}\label{simplac}
If $\g$ is simply-laced,
$$\wt{A}_{i,a} = Y_{i,a(qt)^{-1}}Y_{i,aqt}\times\prod_{j\in I, I_{ji}
  = -1}Y_{j,a}^{-1}$$ coincides with $A_{i,qt}$. Hence
$\mathcal{K}_{i,q,t}$ is the image in $\wt{\mathcal{Y}}_{q,t}$ of:
$$\mathbb{Z}[\alpha, Y_{j,a}^{\pm 1}, Y_{i,a}(1 +
  A_{i,aqt}^{-1})]_{a\in q^{\mathbb{Z}}t^{\mathbb{Z}}}\subset
\mathcal{Y}_{q,t}.$$\qed
\end{rem}

The following result was proved in \cite{FH1}.

\begin{thm}\hfill    \label{recall}

\begin{itemize}

\item[(i)] $\mathcal{K}_{q,t}(\g)$ is non-zero.

\item[(ii)] Every element $F \in \mathcal{K}_{q,t}(\g)$
  (resp. $\in\mathcal{K}_{i,q,t}$, $i\in I$) is uniquely determined by
  the multiplicities of the dominant monomials (resp.
    $i$-dominant monomials) occurring in $F$.

\item[(iii)] For each dominant monomial $m$, there is a unique
  $F_{q,t}(m)\in\mathcal{K}_{q,t}(\g)$ such that $m$ is the unique
  dominant monomial occurring in $F_{q,t}(m)$ (moreover, there is an
  algorithm to construct $F_{q,t}(m)$ explicitly).

\end{itemize}

\end{thm}

In particular, the interpolating $(q,t)$-character $F_{q,t}(Y_{i,a})$
corresponding to the $i$th fundamental representation is well-defined,
as are the interpolating $(q,t)$-characters
$$
F_{q,t}(Y_{i,a}Y_{i,aq_i^2}\cdots Y_{i,q_i^{2( k - 1)}})
$$
corresponding to the KR modules.

\begin{rem}\label{remqu}
  (1) In general, the ring $\mathcal{K}_{q,t}(\g)$ cannot be lifted
  to a subring of $\mathcal{Y}_{q,t}$ satisfying the properties listed
  in Theorem \ref{recall}. For example, suppose that $\rr_i = d - 1$
  and consider the following element:
\begin{multline}    \label{counter}
 \alpha^2 Y_{i,q^{-1}}(1 +
\wt{A}_{i,t}^{-1})Y_{i,q}(1 + \wt{A}_{i,q^2t}^{-1}) - \alpha^2 W_{i,1}
(1 + \alpha \wt{A}_{i,q^2t}^{-1} + \wt{A}_{i,q^2t}^{-1} \wt{A}_{i,
  t}^{-1}) \\ = (\alpha^2 - \alpha^3)W_{i,1} \wt{A}_{i,q^2t}^{-1} +
\alpha^2 W_{i,1} \wt{A}_{i,t}^{-1}
\end{multline}
Viewed as an element of $\wt{\mathcal{Y}}_{q,t}$, it can
  be identified with $\alpha W_{i,1} \wt{A}_{i,1}^{-1}$ which belongs to
$\mathcal{K}_{i,q,t}$ and is an $i$-dominant monomial.

However, viewed as an element of
$\mathcal{Y}_{q,t}$, it does not contain any $i$-dominant
monomials. This contradicts property (ii) of Theorem \ref{recall}.

(2) If $\g$ is simply-laced, there is no such obstruction and we can
consider interpolating $(q,t)$-characters as elements of
$\mathcal{Y}_{q,t}$. These are just the ordinary $q$-characters, but
with the quantum parameter $qt$ instead of $q$.  More precisely, for
$r\in\mathbb{Z}$, let $\mathcal{K}_{q,t}^{(r)}(\g)$ be the subring of
elements of $\mathcal{K}_{q,t}(\g)$ involving only variables $Y_{i,q^d
  (tq)^s}^{\pm 1}$, $i\in I$, $s\in\mathbb{Z}$.  Then
$\mathcal{K}_{q,t}^{(r)}(\g)$ is isomorphic to
$\mathcal{K}_{q,t}^{(0)}(\g)$ as it is obtained from
$\mathcal{K}_{q,t}^{(0)}(\g)$ by the automorphism shifting spectral
parameters by $q^d$. Moreover,
$$\mathcal{K}_{q,t}(\g)\simeq \mathbb{Z}[\alpha]\otimes
\underset{r\in\mathbb{Z}}{\bigotimes} \mathcal{K}_{q,t}^{(r)}(\g).$$
By Remark \ref{simplac}, $\mathcal{K}_{q,t}^{(0)}(\g)$ is the image in
$\wt{\mathcal{Y}}_{q,t}$ of
$$\text{Im}(\chi_{qt})\subset \mathcal{Y}_{q,t}$$
where 
$$\chi_{qt} : K_0(\mathcal{C}_{\mathbb{Z}}) \rightarrow
\mathbb{Z}[Y_{i,(qt)^r}^{\pm 1}]_{i\in I, r\in\mathbb{Z}}$$ is the
ordinary $q$-character homomorphism, but with quantum parameter $q$
replaced by $qt$  and $\mathcal{C}_{\mathbb{Z}}$ is the subcategory of
finite-dimensional representations of $U_{qt}(\wh{\g})$ whose simple
constituents have Drinfeld polynomials have roots that are powers of
$qt$. \qed
\end{rem}

\begin{thm}
  $\mathcal{K}_{q,t}(\g)$ can be lifted to a subring
$\overline{\mathcal{K}}_{q,t}(\g)\subset\overline{\mathcal{Y}}_{q,t}$ so
that we have a commutative diagram:
$$\begin{xymatrix}{ & \mathcal{Y}_{q,t}
    \ar@{>>}[d]\ar@{>>}[rd]^{\overline{\Pi}_t, \overline{\Pi}_q} &
    \\ \overline{\mathcal{K}}_{q,t}(\g)\ar@{^{(}->}[r] \ar@{>>}[d]&
    \overline{\mathcal{Y}}_{q,t}\ar@{>>}[d]\ar@{>>}[r]_{\overline{\Pi}_t,\overline{\Pi}_q}
    & \mathbb{Z}[ Y_{j,b}^{\pm 1}]_{j\in I, b\in
      t^{\mathbb{Z}}q^{\mathbb{Z}}}
    \\ \mathcal{K}_{q,t}(\g)\ar@{^{(}->}[r] & \wt{\mathcal{Y}}_{q,t}&}
\end{xymatrix}.$$
\end{thm}

This theorem follows from Theorem \ref{refined} below.

\subsection{The refined ring of interpolating $(q,t)$-characters}

Let us consider the completion $\overline{\mathcal{Y}}_{q,t}^\infty$ of the ring $\overline{\mathcal{Y}}_{q,t}$ which
includes infinite linear combinations of monomials whose
$\g$-weights belong to a finite union of cones $\{\mu\preceq
\omega\}$, with $\preceq$ being the standard ordering (as for the
$q$-characters in the category $\mathcal{O}$). 
Next, we define the subrings $\overline{\mathcal{K}}_{i,q,t}\subset
\overline{\mathcal{Y}}_{q,t}^\infty$ in the same way as the subrings
$\mathcal{K}_{i,q,t}$ of $\wt{\mathcal{Y}}_{q,t}$, except that we
include infinite which make sense in $\overline{\mathcal{Y}}_{q,t}^\infty$.
Note that the elements generating $\mathcal{K}_{i,q,t}$ are well-defined in $\overline{\mathcal{Y}}_{q,t}$ and in $\overline{\mathcal{Y}}_{q,t}^\infty$ 
by (2) in Remark \ref{remqt}.

\begin{defn} We define the {\em refined ring of
interpolating $(q,t)$-characters} $\overline{\mathcal{K}}_{q,t}(\g)$
  as the intersection
$$\bigcap_{i\in I}
\overline{\mathcal{K}}_{i,q,t} \subset
\overline{\mathcal{Y}}_{q,t}^\infty.$$
\end{defn}

\begin{thm}    \label{refined}
The refined ring of interpolating $(q,t)$-characters
$\overline{\mathcal{K}}_{q,t}(\g)$ satisfies the same properties as
the properties of $\mathcal{K}_{q,t}(\g)$ listed in Theorem
\ref{recall}. In particular, it contains a unique element
  $$\overline{F}_{q,t}(m)\in\overline{\mathcal{K}}_{q,t}(\g)$$ 
	for each dominant monomial $m$.
\end{thm}

\begin{proof} All proofs in \cite{FH1} remain valid if we replace
$\wt{\mathcal{Y}}_{q,t}$ by $\overline{\mathcal{Y}}_{q,t}^\infty$. More
  precisely, 

(i) follows from (iii).

(ii) The property for the $i$-dominant monomials is proved as for \cite[Lemma 4.1]{FH1}. 
In the simply-laced cases ($d = 1$), the proof is the same as for ordinary $q$-character as discussed above (see Remark \ref{remqu}).
If $d = 2$, the crucial point is indeed that the formula (\ref{counter}) has an $i$-dominant monomial: 
its respective specializations under $\overline{\Pi}_t$, $\Pi_q$, $\Pi_t$ are $0$, $Y_{i,q^{-1}}Y_{i,q}A_{i,1}^{-1}$, $0$, which are $i$-dominant.
Hence expression (\ref{counter}) identifies with $\alpha (2 - \alpha) Y_{i,q^{-1}}Y_{i,q}A_{i,1}^{-1}$ in $\overline{\mathcal{Y}}_{q,t}^\infty$ 
which is $i$-dominant. An analogous reasoning gives the result for $d = 3$.

The property for dominant monomials is proved exactly as in
\cite[Lemma 4.2]{FH1} from the property we just obtained for the
$i$-dominant monomials.

(iii) This is proved as in \cite[Section 4.2]{FH1}: an algorithm is proposed which produces 
the $F(W_{i,a})$ for any $i,a$, and from which the $F(m)$ are obtained as algebraic combinations of
these $F(W_{i,a})$. 
For the $F(W_{i,a})$, the algorithm and the proof that it does not fail are the same. 
This proof is obtained by induction on the rank of the Lie algebra. That is why we have
to check for $n = 1$ (type $A_1$) and $n = 2$ (types $A_1\times A_1$, $A_2$, $C_2$, $G_2$). The simply-laced cases are clear 
by Remark \ref{simplac}. For type $C_2$, the formulas obtained in
\cite{FH1} for $\mathcal{K}_{q,t}(\mathfrak{g})$ work as well for $\overline{\mathcal{K}}_{q,t}(\mathfrak{g})$:
\begin{equation}\label{fun}\overline{F}_{q,t}(Y_{2,1}) = Y_{2,1} + Y_{2,q^4t^2}^{-1} Y_{1,qt}Y_{1,q^3t} + \alpha Y_{1,qt}Y_{1,q^5t^3}^{-1} 
+ Y_{1,q^3t^3}^{-1}Y_{1,q^5t^3}^{-1}Y_{2,q^2t^2} + Y_{2,q^6t^4}^{-1},\end{equation}
\begin{equation}\label{fdeux}\overline{F}_{q,t}(W_{1,1}) = Y_{1,q^{-1}}Y_{1,q} 
+ \alpha Y_{1,q^{-1}}Y_{1,q^3t^2}^{-1}Y_{2,q^2t} 
+ Y_{1,qt^2}^{-1}Y_{1,q^3t^2}^{-1}Y_{2,t}Y_{2,q^2t}
\end{equation}
$$+\alpha Y_{1,q^{-1}}Y_{1,q^5t^2}Y_{2,q^6t^3}^{-1}+Y_{2,q^2 t}Y_{2,q^4 t^3}^{-1}
+Y_{1,qt^2}^{-1}Y_{1,q^5 t^2}Y_{2,q^6 t^3}^{-1}Y_{2,t}
+\alpha Y_{1,q^{-1}}Y_{1,q^7t^4}^{-1}
$$
$$+Y_{2,q^4t^3}^{-1}Y_{2,q^6t^3}^{-1}Y_{1,q^3t^2}Y_{1,q^5t^2}+ \alpha
Y_{1,qt^2}^{-1}Y_{1,q^7t^4}^{-1}Y_{2,t} +
\alpha Y_{2,q^4t^3}^{-1}Y_{1,q^3t^2}Y_{1,q^7t^4}^{-1}
+Y_{1,q^5 t^4}^{-1}Y_{1,q^7 t^4}^{-1},$$
\begin{equation}\label{ftrois}\overline{F}_{q,t}(\alpha Y_{1,1}) = \alpha( Y_{1,1} 
+ Y_{1,q^2t^2}^{-1}Y_{2,qt} + Y_{2,q^5t^3}^{-1} Y_{1,q^4t^2} + Y_{1,q^6t^4}^{-1})
.\end{equation}
For the type $G_2$, the formulas obtained in \cite[Section 5.2]{FH1} work as well.
\end{proof}

Now let us consider the various specializations of $\mathcal{K}_{q,t}(\g)$.  We
have first the three specializations corresponding to the respective
specializations of the deformed ${\mc W}$-algebras at $t = 1$, $q =
\epsilon$ and $q = 1$.

\begin{itemize}

\item The specialization $\Pi_q = \Pi_{t = 1, \alpha = 1}$: we get the
  $q$-character of a $U_q(\wh{\mathfrak{g}})$-module, that is an
  element of $\mathcal{K}_q^+(\mathfrak{g})$ (this is proved in
  \cite{FH1}).

\item The specialization $\Pi_t = \Pi_{q = \epsilon, \alpha = 0}$: we
  get the $t$-character of a $U_t({}^L\wh{\mathfrak{g}})$-module (this
  is proved in \cite{FH1}).

\item The specialization $\overline{\Pi}_t = \Pi_{q = 1, \alpha = d}$:
  we get the $t$-character of an element of
  $\mathcal{K}_t^-(\mathfrak{g})$ (in light of Proposition
  \ref{props-}, the proof is parallel to the proof for the other
    specializations since
    $\overline{\Pi}_t(\overline{\mathcal{K}}_{i,q,t}) =
    \text{Ker}(S_i^-)$). Hence, it can also be viewed as the
    folded $t$-character of a $U_t(\wh{\mathfrak{g}'})$-module 
		(see Section \ref{gprime}).
\end{itemize}

There are also two additional interesting specializations:

\begin{itemize}

\item The specialization $\Pi_t' = \Pi_{q = 1, \alpha = 0}$: it is
  well-defined as the composition of $\Pi_t$ and the homomorphism 
$$\mathbb{Z}[W_{i,a}^{\pm 1}]_{i\in I, a\in t^{\mathbb{Z}}\epsilon^{\mathbb{Z}}}\rightarrow \mathbb{Z}[\overline{Y}_{i,a}^{\pm 1}]_{i\in I, a\in t^{\mathbb{Z}}}$$ 
(see formula \eqref{barY} for the definition of $\overline{Y}_{i,a}$) identifying $\epsilon$ with $1$. We get the $t$-character of an element of $\mathcal{K}_t^-({}^L\mathfrak{g})$ 
which is defined in terms of the variables $\overline{Y}_{i,a}$ instead of $Y_{i,a}$ (this is parallel to the proof for the other specializations as $\Pi_t'(\overline{\mathcal{K}}_{i,q,t}) = \text{Ker}(S_i^-)$). Hence, it can also be viewed as the folded $t$-character of a $U_t\left(\wh{\left({}^L\mathfrak{g}\right)'}\right)$-module 
(see Section \ref{gprime}).

\item The specialization $\overline{\Pi}_q = \Pi_{t = 1, \alpha = 0}$:
  we obtain elements of $\text{Im}(\chi_q^L)$, i.e. $q$-characters of
  modules over the Langlands dual quantum affine algebra
  $U_q({}^L\wh{\mathfrak{g}})$, as defined in \cite[Section
    12]{hlast} (in the context of the parametrization of simple
  representations of shifted quantum affine algebras).
\end{itemize}

So, in the total, we have $5$ interesting specializations of the refined ring
of interpolating $(q,t)$-characters:

$$\begin{xymatrix}{
&& \on{Rep} U_t({}^L\wh{\mathfrak{g}})\ar[r]_{\cong} & \on{Rep} U_t(\widehat{({}^L\g)'})&
\\ 
\on{Rep}U_q(\wh{\g})\ar[r]_{\cong}&\mathcal{K}_q^+(\mathfrak{g})&
\overline{\mathcal{K}}_{q,t}(\g)
\ar[r]^{\overline{\Pi}_t}\ar[l]_{\Pi_q}\ar[u]^{\Pi_t}\ar[dr]_{\Pi_t'}\ar[dl]^{\overline{\Pi}_q}
& \mathcal{K}_t^-(\mathfrak{g})&\ar@{>>}[l] \on{Rep} U_t(\widehat{\g'})
\\ &\text{Im}(\chi_q^L)&
&\mathcal{K}_t^-({}^L\mathfrak{g})&\ar@{>>}[l] \on{Rep} U_t(\widehat{({}^L\g)'}))}
\end{xymatrix}.$$

\begin{rem} To get the $t$-characters of
  $U_t({}^L\wh{\mathfrak{g}})$-modules in the sense of \cite{H} (as
  recalled above), one has to change the sign in the definition of the
  variables $Z_{i,a}$, as described in \cite{FH1} (formula (2) for $r
  = 2$ and Section 3.2 for $r = 3$). In this paper, to simplify our
    notation, we do not make this sign change.\qed
\end{rem}

\begin{rem} In \cite[Lemma 4.16]{FH1} we proved that the
    generators of $\mathcal{K}_{q,t}(\g)$, i.e. the fundamental
    elements $F_{q,t}(W_{i,a})$, are finite linear combinations of
    monomials.  In the case of the refined ring
    $\ol{\mathcal{K}}_{q,t}(\g)$ the algorithm in the proof of Theorem
    \ref{refined} produces the fundamental elements
    $\overline{F}_{q,t}(W_{i,a})$ but so far we have only been able to
    prove that they are (possibly infinite) linear combinations of
    monomials with weights in the union of finitely many cones. That's
    why we have allowed such linear combinations in the definition of
    the $\overline{\mathcal{K}}_{i,q,t}$ above.  However, we also know that
    the specializations under $\Pi_q$, $\Pi_t$, $\overline{\Pi}_t$ and
    $\Pi_t'$ of an interpolating $(q,t)$-character with a finite
    number of dominant monomials are finite linear combinations. This
    follows from Proposition \ref{fundqtr} below.
    Therefore finiteness of elements of $\ol{\mathcal{K}}_{q,t}(\g)$
    follows from Conjecture \ref{positivity} saying that the
    coefficients of all monomials in $\overline{F}_{q,t}(W_{i,a})$ are
    positive (see Corollary \ref{finiteness}).\qed
\end{rem}

\begin{example}\label{funex} Let us consider the interpolating
  $(q,t)$-character (\ref{fun}) corresponding to the second
  fundamental representation of $U_q(C_2^{(1)})$. In this case, we
  have $^L\ghat = D_3^{(2)}, \wh{\g'} = A_3^{(1)}, \wh{({}^L\g)'} =
  D_3^{(1)} \simeq A_3^{(1)}$ (the last isomorphism involves
  switching the indices $1 \leftrightarrow 2 \in \{1,2,3\} = I'$).

The above $5$ specializations (listed in the same clockwise order) are,
respectively,

\begin{itemize}

\item the $q$-character of the fundamental representation $L(Y_{2,1})$ of $U_q(C_2^{(1)})$:
$$Y_{2,1} + Y_{2,q^4}^{-1} Y_{1,q}Y_{1,q^3} + Y_{1,q}Y_{1,q^5}^{-1} 
+ Y_{1,q^3}^{-1}Y_{1,q^5}^{-1}Y_{2,q^2} + Y_{2,q^6}^{-1},$$

\item the twisted $t$-character of the fundamental representation $L(Z_{2,1})$ of $U_t(D_3^{(2)})$:
$$Z_{2,1} + Z_{2,t^2}^{-1} Z_{1,-t^2} + Z_{1,-t^6}^{-1}Z_{2,-t^2} + Z_{2,-t^4}^{-1},$$

\item the $t$-character in $\mathcal{K}_t^-(C_2)$ of highest monomial $Y_{2,1}$:
$$Y_{2,1} + Y_{2,t^2}^{-1} Y_{1,t}^2 + 2 Y_{1,t}Y_{1,t^3}^{-1} 
+ Y_{1,t^3}^{-2}Y_{2,t^2} + Y_{2,t^4}^{-1},$$
also equal to the folded $t$-character of the fundamental representation $L(Y_{2,1})$ of $U_t(A_3^{(1)})$, 
\item the $t$-character in $\mathcal{K}_t^-({}^LC_2)$ of highest monomial $\overline{Y}_{2,1}$:
$$\overline{Y}_{2,1} + \overline{Y}_{2,t^2}^{-1} \overline{Y}_{1,t} + \overline{Y}_{1,t^3}^{-1} \overline{Y}_{2,t^2} + \overline{Y}_{2,t^4}^{-1},$$
which after switching the indices $1$ and $2$ is equal to the
folded $t$-character of the fundamental representation $L(Y_{1,1})$ of
$U_t(A_3^{(1)})$, defined in terms of the variables
$\overline{Y}_{i,a}$ instead of $Y_{i,a}$.

\item the Langlands dual $q$-characters of the fundamental representation $L(Z_{2,1})$ of $U_q(D_3^{(2)})$ 
$$W_{2,1} + W_{2,q^4}^{-1} W_{1,q^2} + W_{1,q^4}^{-1}W_{2,q^2} + W_{2,q^6}^{-1}.$$
\end{itemize}

\end{example}

\begin{example}\label{fdeuxex} Let us consider the interpolating $(q,t)$-character (\ref{fdeux}) corresponding to 
the representation $L(W_{1,1})$ of $U_q(C_2^{(1)})$. The $5$ specializations are, respectively,

\begin{itemize}

\item the $q$-character of the $11$-dimensional representation $L(Y_{1,q^{-1}}Y_{1,q})$ of $U_q(C_2^{(1)})$.

\item the twisted $t$-character of the fundamental representation 
$L(Z_{1,-1})$ of $U_t({}^L\widehat{\mathfrak{g}}) = U_t(D_3^{(2)})$.
$$Z_{1,-1} + Z_{1,-t^4}^{-1}Z_{2,t}Z_{2,-t} 
+ Z_{2,-t}Z_{2,t^3}^{-1} + Z_{2,-t^3}^{-1}Z_{2,t}
+ Z_{2,t^3}^{-1}Z_{2-t^3}^{-1} Z_{1,-t^4}
+ Z_{1,-t^8}^{-1}.$$

\item the $t$-character in $\mathcal{K}_t^-(C_2)$ of highest monomial $Y_{1,1}^2$:
$$(Y_{1,1} + Y_{1,t^2}^{-1}Y_{2,t} + Y_{2,t^3}^{-1}Y_{1,t^2} + Y_{1,t^4}^{-1})^2,$$
also equal to the folded $t$-character of the representation $L(Y_{1,1}Y_{3,1})$ of $U_t(A_3^{(1)})$ 

\item the $t$-character in $\mathcal{K}_t^-({}^LC_2)$ of highest monomial $\overline{Y}_{1,1}$:
$$\overline{Y}_{1,1}
+ \overline{Y}_{1,t^2}^{-1}\overline{Y}_{2,t}^2
+2\overline{Y}_{2,t}\overline{Y}_{2,t^3}^{-1}
+\overline{Y}_{2,t^3}^{-2}\overline{Y}_{1,t^2}
+\overline{Y}_{1,t^4}^{-1}$$
which after switching the indices $1$ and $2$ is equal to the folded
$t$-character of the fundamental representation $L(Y_{2,1})$ of
$U_t(A_3^{(1)})$, defined in terms of the variables $\overline{Y}_{i,a}$ instead of $Y_{i,a}$.

\item the Langlands dual $q$-characters of the fundamental representation $L(Z_{2,1})$ of $U_q(D_3^{(2)})$ 
$$W_{1,1} 
+ W_{1,q^2}^{-1}W_{2,1}W_{2,q^2}
+W_{2,q^2}W_{2,q^4}^{-1}
+W_{1,q^2}^{-1}W_{1,q^4}W_{2,q^6}^{-1}Y_{2,t}
+W_{2,q^4}^{-1}W_{2,q^6}^{-1}W_{1,q^4}
+W_{1,q^6}^{-1}.$$

\end{itemize}

\end{example}

\begin{example} Let us consider the interpolating $(q,t)$-character (\ref{ftrois}) corresponding to 
the first fundamental representation of $U_q(C_2^{(1)})$. The
specializations under $\Pi_t$, $\Pi_t'$, $\overline{\Pi}_q$ 
are zero because of the $\alpha$ factor. The specializations under
$\Pi_q$ and $\overline{\Pi}_t$ are respectively:

\begin{itemize}

\item the $q$-character of the fundamental representation $L(Y_{1,1})$ of $U_q(C_2^{(1)})$:
$$Y_{1,1} 
+ Y_{1,q^2}^{-1}Y_{2,q} + Y_{2,q^5}^{-1} Y_{1,q^4} + Y_{1,q^6}^{-1},$$

\item the $t$-character in $\mathcal{K}_t^-(C_2)$ of highest monomial $2Y_{1,1}$:
$$2( Y_{1,1} 
+ Y_{1,t^2}^{-1}Y_{2,t} + Y_{2,t^3}^{-1} Y_{1,t^2} + Y_{1,t^4}^{-1})$$
also equal to (the double of) the folded $t$-character of the representation $L(Y_{1,1})$ of $U_t(A_3^{(1)})$. 
\end{itemize}

\end{example}

\subsection{$\sigma$-fundamental interpolating
  $(q,t)$-characters}    \label{fundsec}

In this subsection we identify the representation $M(W)$ of
$U_q({}^L\ghat)$ in terms of Conjecture \ref{part3}
  for the simplest $U_q(\widehat{\g'})$-modules with
  $\sigma$-invariant highest monomials.
These monomials have the form $Y_{i,a}$ if $\sigma(i)=i$ and
$\prod_{1\leq k \leq d} Y_{\sigma^k(i),a}$ if $\sigma(i)\neq i$. For
this reason, we call these $U_q(\widehat{\g'})$-modules $\sigma$-{\em
  fundamental}. Note however that the corresponding
$U_t({}^L\wh{\mathfrak{g}})$-modules (see the second part of
Proposition \ref{fundqtr}) are in fact fundamental.

Consider a $U_q(\widehat{\g'})$-module of the form
$L(\wt{Y}_{i,a}), i \in I$ (the set of vertices of the Dynkin diagram
of $\g = (\g')^\sigma$) and with $a\in q^{\mathbb{Z}}$.

Let
$$X^{(i)}_{q,t} :=
  \ol{F}_{q,t}(W_{i,a})\in\overline{\mathcal{K}}_{q,t}(\g).$$ We
call it the $i$th $\sigma$-{\em fundamental interpolating
    $(q,t)$-character}.  The following proposition shows
  that $X^{(i)}_{q,t}$ satisfies the properties of the element $X_{q,t}$
whose existence is stated in Conjecture \ref{part3} in the case when
$W = L(\wt{Y}_{i,a})$.

\begin{prop}\label{fundqtr} The specializations of $X^{(i)}_{q,t}$ under
  $\Pi_q$, $\Pi_t$, $\overline{\Pi}_t$, $\Pi_t'$, $\overline{\Pi}_q$
  are respectively

\begin{itemize}

\item the $q$-character of the simple $U_q(\wh{\mathfrak{g}})$-module
  of highest monomial $W_{i,a}$. It is a KR module (a fundamental
  representation if $d_i = d$).

\item the $t$-character of the fundamental
  $U_t({}^L\wh{\mathfrak{g}})$-module of highest monomial
  $Z_{i,(a_{q = \epsilon})^{\rr_i^\vee}}$ (see \eqref{dvee} for the
  definition of $\rr_i^\vee$), with $a_{q = \epsilon}$ the
  specialization of $a$ at $q = \epsilon$.

\item the $t$-character of the element $F(\overline{Y}_{i,a}) =
  F(Y_{i,a})^{\rr_i^\vee}$ in $\mathcal{K}_t^-(\mathfrak{g})$
   (defined in terms of the variables $Y_{j,b}$) of
  highest monomial $\overline{Y}_{i,a}$. It is also the folded
  $t$-character of the simple $U_t(\wh{\mathfrak{g}'})$-module of
  highest monomial $\wt{Y}_{i,a}$
	(this is a tensor product of the
	fundamental representations).
	
	\item the $t$-character of the element $F(\overline{Y}_{i,a})$
          in $\mathcal{K}_t^-({}^L\mathfrak{g})$ (which is
              defined in terms of the variables $\overline{Y}_{j,b}$
              instead of $Y_{j,b}$) of highest monomial
            $\overline{Y}_{i,a}$. It is also, after the appropriate
            permutation of indices, equal to the folded $t$-character
            of the fundamental
            $U_t\left(\wh{\left({}^L\mathfrak{g}\right)'}\right)$-module
            of highest monomial $Y_{i,a}$.

\item the Langlands dual $q$-character of the fundamental
  $U_q({}^L\wh{\mathfrak{g}})$-module of highest monomial
  $Z_{i,a^{\rr_i^\vee}}$.
\end{itemize}
\end{prop}

\begin{proof} By construction, following the algorithm in
  \cite[Section 4.2.4]{FH1},
the interpolating $(q,t)$-character $F_{q,t}(W_{i,a})$  is the sum of
$W_{i,a}$ plus other monomials of the form
$$W_{i,a}\wt{A}_{i,aq^{d}t}^{-1}\wt{A}^{-1}_{j_1,b_1}\cdots \wt{A}^{-1}_{j_N,b_N}$$ 
where $j_k\in I$, $b_j\in aq^{\mathbb{Z}}t^{\mathbb{Z}}$ and the $\wt{A}_{j,b}$ are given by formula (\ref{wta}). 
Though the argument in \cite{FH1} concerns $F_{q,t}(W_{i,a})$, the
algorithm constructing 
$X_{q,t}^{(i)} = \ol{F}_{q,t}(W_{i,a})$ in $\overline{{\mc Y}}_{q,t}$ is the
same, hence we obtain that $X_{q,t}^{(i)}$ satisfies the same property. 
This implies that each of the $5$ specializations
considered in the statement of the proposition has a unique dominant monomial; namely, the corresponding specialization of the highest monomial 
$W_{i,a}$. Then, as explained in \cite{FH1}, the first two specializations are respectively the $q$-character of 
a KR module of $U_q(\wh{\mathfrak{g}})$ and the $t$-character of a fundamental module of $U_t({}^L\wh{\mathfrak{g}})$, which are 
known to have a unique dominant monomial by \cite{hcr} and \cite{H}. The third specialization is an element of 
$\mathcal{K}_t^-(\mathfrak{g})$ with a unique dominant monomial $\overline{Y}_{i,a}$, hence it is equal to $F(\overline{Y}_{i,a})$.
Moreover, the simple $U_t(\wh{\mathfrak{g}'})$-module of highest monomial $\wt{Y}_{i,a}$ is a simple tensor product of fundamental 
representations, and, as a consequence of \cite{fm}, its (folded) $t$-character belongs to 
$\overline{Y}_{i,a} + \overline{Y}_{i,a}A_{i,at}^{-1}\mathbb{Z}[A_{j,b}^{-1}]_{j\in I, b\in at^{\mathbb{Z}}}$, with the $A_{j,b}$ as
in formula (\ref{AY2}). Hence it has a unique 
dominant monomials and is equal to $F(\overline{Y}_{i,a})$.
We use an analogous argument for the fourth specialization. For the last specialization, this is the Langlands dual $q$-character which is 
precisely defined as the specialization of the interpolating $(q,t)$-character $F_{q,t}(W_{i,a})$.
\end{proof}

The following conjecture is true in all examples known to us.

\begin{conj}    \label{positivity}
$X^{(i)}_{q,t}$ can be written in such a way that all coefficients of
  its monomials are positive.
\end{conj}

\begin{cor}    \label{finiteness}
  If Conjecture \ref{positivity} holds, then $X^{(i)}_{q,t}$ is a
  polynomial and therefore every element of
  $\overline{\mathcal{K}}_{q,t}(\g)$ is a polynomial (a finite linear
  combination of monomials).
\end{cor}

\begin{proof}
Positivity of coefficients of $X^{(i)}_{q,t}$ implies
that if $X^{(i)}_{q,t}$ were an infinite combination of monomials,
then so would be $\overline{\Pi}_t(X^{(i)}_{q,t})$, which is not the
case. Since the elements $X^{(i)}_{q,t}$ generate
$\overline{\mathcal{K}}_{q,t}(\g)$, it follows that every element of
$\overline{\mathcal{K}}_{q,t}(\g)$ is finite as well.
\end{proof}

\begin{rem} The statement analogous to Conjecture \ref{positivity} for
  the elements $F_{q,t}(W_{i,a})$ of the original ring
  ${\mathcal{K}}_{q,t}(\g)$ is also a conjecture.  At present, we are
  not aware of a uniform proof of this statement or of Conjecture
  \ref{positivity}. We expect, however, that it is possible to write
  an explicit positive and finite expression for these elements for
  Lie algebras $\g$ of classical types, and to check the statement
  with the help of a computer for $\g$ of exceptional types. There is
  a similar question for the corresponding elements of the deformed
  $\mathcal{W}$-algebra. We hope to discuss this in another paper.
\end{rem}

Assuming Conjecture \ref{positivity}, we obtain a proof of the second
part of Conjecture \ref{main conj bis},(ii) (modulo Conjecture
\ref{main conj bis},(i)) and Conjecture \ref{part3} for
$W=L(\wt{Y}_{i,a})$.

\begin{thm} Let $W = L(\wt{Y}_{i,a})$, $i\in I$,
  $a\in\mathbb{C}^\times$.  Suppose that Conjecture \ref{positivity}
  holds. Then there exists a subspace
  $\overline{W}\subset \wt{W}$ isomorphic, as a vector space graded by
  ${}^L\g$-weights, to a $U_q({}^L\widehat{\g})$-module $M(W)$,
   which is the simple $U_t({}^L\wh{\mathfrak{g}})$-module
    of highest monomial $Z_{i,(a_{q = \epsilon})^{\rr_i^\vee}}$. Moreover, it satisfies
    Conjecture \ref{part3} for the interpolating $(q,t)$-character
    $X^{(i)}_{q,t}$.
\end{thm}

\begin{proof}
We have seen in Proposition \ref{fundqtr} that
$\overline{\Pi}_t(X^{(i)}_{q,t})$ is the folded $t$-character of $W$ and
that $\Pi_t(X^{(i)}_{q,t}) = \chi_t(M(W))$ where $M(W)$ is the fundamental
$U_t({}^L\wh{\mathfrak{g}})$-module of highest monomial $Z_{i,(a_{q =
    \epsilon})^{\rr_i^\vee}}$.  The positivity of $X^{(i)}_{q,t}$ implies
that the multiplicities of weights in $\Pi_t(X^{(i)}_{q,t})$ (for which
$\alpha = 0$) are lower than in $\overline{\Pi}_t(X^{(i)}_{q,t})$ (for which
$\alpha = d$). Hence the result.
\end{proof}

\section{Examples}    \label{ex}

In this section we present five explicit examples confirming
Conjectures \ref{main conj bis} and \ref{part3}.

\subsection{First example: the fundamental
  representation $L(Y_{2,1})$ of $U_q(A_3^{(1)})$}

Consider the Lie algebra $\g'$ of type $A_3$ with the automorphism
$\sigma$ exchanging the nodes $1$ and $3$ of its Dynkin diagram. We
have the Lie algebra $\g = C_2$ with $d_1 = 1$, $d_2 = 2$ and its
Langlands dual is ${}^L\g = B_2$. We also have
$\left({}^L\mathfrak{g}\right)' = A_3$, but with the nodes $1$ and $3$
exchanged in comparison to the original $\g'$, and $^L\widehat{\g} =
D_3^{(2)}$.

Then the algebra $U_q(\wh{\g'}) = U_q(A_3^{(1)})$ acts on its
$6$-dimensional fundamental representation $W = L(Y_{2,1})$ whose
$q$-character is
$$Y_{2,1} + Y_{2,q^2}^{-1}Y_{1,q}Y_{3,q} + Y_{1,q^3}^{-1}Y_{3,q} +
Y_{3,q^3}^{-1}Y_{1,q} + Y_{2,q^2}Y_{1,q^3}^{-1}Y_{3,q^3}^{-1} +
Y_{2,q^4}^{-1}.$$

We have the $Q$-operators $Q_i^\pm(z,u) = t_{R_i^\pm}(z,u)$ associated
to the prefundamental representations $R_i^+(z), i=1,2,3$, of
$U_q(\wh{\g'})$. The roots of the corresponding Baxter polynomials
satisfy the BAE (\ref{bae gen}) of type $A_3^{(1)}$. Now we consider
the subspace $W(u)$ of $W$ where the actions of
$Q_1^{\pm,\text{ss}}(z,u)$ and $Q_3^{\pm,\text{ss}}(z,u)$ coincide. We
also have the $4$-dimensional subspace $\wh{W}\subset W$ which is the
direct sum of the $\sigma$-invariant weight subspaces $W$ and the
subspace $\wt{W}$ which is the direct sum of the $\ell$-weight
subspaces corresponding to the $\sigma$-invariant $\ell$-weights (or
monomials in the $q$-character).

In the present case all $\sigma$-invariant weights have multiplicity
$1$. Therefore, by Remark \ref{remtriv},(2), for all
values of $u$ we have
$$W(u) = \wt{W} = \wh{W} = W^\sigma(u).$$
Thus, part (i) of Conjecture \ref{main conj bis} is verified in this case.

This implies that only $\sigma$-invariant solutions of the BAE
\eqref{bae gen} appear in this case. Identifying the eigenvalues of
the operators $Q_1^{\pm,\text{ss}}(z) = Q_3^{\pm,\text{ss}}(z)$, we obtain the folded BAE (\ref{bae
  gen1}) corresponding to the Lie algebra $\g = C_2$.

The $q$-character of the subspace $\wt{W}$ is given by the formula
$$Y_{2,1} + Y_{2,q^2}^{-1}(Y_{1,q} Y_{3,q}) +
  Y_{2,q^2}(Y_{1,q^3}Y_{3,q^3})^{-1} + Y_{2,q^4}^{-1}.$$ Setting
  $\wt{Y}_{1,q} = Y_{1,q}Y_{3,q}$, one gets
\begin{equation}    \label{6to4}
Y_{2,1} + Y_{2,q^2}^{-1} \wt{Y}_{1,q} + Y_{2,q^2}\wt{Y}_{1,q^3}^{-1} +
Y_{2,q^4}^{-1}.
\end{equation}
The corresponding character
$$y_2 + y_2^{-1}y_1 + y_1^{-1}y_2 + y_2^{-1}$$
is equal to the character of a
  fundamental representation of $U_t({}^L\widehat{\mathfrak{g}}) =
  U_t(D_3^{(2)})$. Thus, part (ii) of Conjecture \ref{main conj bis}
  is verified.

Finally, we discuss Conjecture \ref{part3} in the
present case.  We have the interpolating
$(q,t)$-character (\ref{fun}) of the second fundamental representation of
$U_q(C_2^{(1)})$ studied in Example \ref{funex}. We verify
Conjecture \ref{part3}:

The specialization of the interpolating
$(q,t)$-character (\ref{fun}) under $\Pi_t$ is equal to
the $t$-character of the fundamental representation $M(W)$ of
$U_t({}^L\widehat{\g}) = U_t(D_3^{(2)})$.

Its specialization under $\overline{\Pi}_t$ is the folded
$t$-character of the $U_t(\widehat{\g'}) = U_t(A_3^{(1)})$-module $W$.

Its specialization under $\Pi_t'$ is (after switching the indices $1$
and $2$) the folded $t$-character of $W'$, the fundamental
representation $L(Y_{1,1})$ of $U_t(A_3^{(1)})$, defined in terms of
the variables $\overline{Y}_{i,a}$ instead of $Y_{i,a}$.

Thus, we find that both Conjectures \ref{main conj bis} and
\ref{part3} hold in this case. In other words, we obtain an example of
the folded quantum integrable system associated to $\g=C_2$ whose
spectra correspond to the solutions of the folded BAE equation
(\ref{bae gen1}) associated to $\g=C_2$, with the space of states
being isomorphic to a representation of $U_q({}^L\widehat{\g} =
D_3^{(2)})$.

\subsection{Second example: representation $L(Y_{1,1}Y_{2n-1,1})$ of
  $U_q(A_{2n-1}^{(1)})$}    \label{ex2}

We work with the  Lie algebras $\g' = A_{2n-1}$, $\g = C_n$, ${}^L\g =
B_n$, $\left({}^L\mathfrak{g}\right)' = D_{n+1}$ and $^L\widehat{\g} =
D_{n+1}^{(2)}$ (we assume that $n\geq 2$).

Consider the representation $W=L(Y_{1,1}Y_{2n-1,1})$ of
$U_q(A_{2n-1}^{(1)})$. It is $4n^2$-dimensional and is isomorphic to the tensor
product of the fundamental representations $L(Y_{1,1})\otimes
L(Y_{2n-1,1})$. Its $q$-character has $2n$ $\sigma$-invariant monomials,
including one with multiplicity $2$:
$$\chi_q(\wt{W}) = Y_{1,1}Y_{2n-1,1} +
Y_{1,q^2}^{-1}Y_{2n-1,q^2}^{-1}Y_{2,q}Y_{2n-2,q} + \cdots + Y_{n-1,q^n}^{-1}Y_{n,q^n}^{-1}Y_{n,q^{n-1}}^2$$
$$+2 Y_{n,q^{n-1}}Y_{n,q^{n+1}}^{-1} 
+ Y_{n,q^{n+1}}^{-2}Y_{n-1,q^n}Y_{n+1,q^n} + \cdots + Y_{1,q^{2n}}^{-1}Y_{2n-1,q^{2n}}^{-1}.$$

All weight subspaces of $W$ are one-dimensional, except for the $0$-weight
subspace $W_0$ of $W$ which is $2n$-dimensional,
containing as a proper subspace its intersection with $\wt{W}$, which
is the $2$-dimensional $\ell$-weight space associated to the
$\sigma$-invariant monomial $Y_{n,q^{n-1}}Y_{n,q^{n+1}}^{-1}$
(the other $n-1$ monomials of weight $0$ are not $\sigma$-invariant).
This $\ell$-subspace is therefore
  precisely the intersection $\wt{W} \cap W_0$, and
Lemma \ref{whsigma},(2) implies that $\wh\sigma$ preserves this
$\ell$-weight subspace. We are going to show that $\wh\sigma$ acts on
it as the identity.

Let $v$ be a generating vector of the weight subspace $W_{\alpha_n}$,
which is one-dimensional, and hence is an $\ell$-weight subspace. The
corresponding monomial is
$Y_{n,q^{n-1}}^2Y_{n-1,q^n}^{-1}Y_{n-1,q^n}^{-1}$. Under the action of
the $U_q(\wh{\sw}_2)$ subalgebra corresponding to the node $n$, the
vector $v$ generates a representation of dimension $4$ with a
$2$-dimensional $0$-weight space, which is spanned by the
$x_{n,m}^-.v$, $m\in\mathbb{Z}$. Since $W_{\alpha_n}$ is preserved by
$\wh\sigma$, it follows that $\wh{\sigma}(v) = \pm v$. In fact, it is
easy to see that $\wh\sigma(v) = v$ by restricting $W$ to $U_q(\g)$
and taking the limit $q \to 1$. Then $W$ decomposes into the direct
sum of the adjoint representation of $A_{2n-1}$ and the trivial
one-dimensional representation, and the weight subspace $W_{\alpha_n}$
appears in the former. It is easy to see that the operator $\wh\sigma$
acts on the adjoint representation as the automorphism induced by the
automorphism $\sigma$ of the Dynkin diagram of $A_{2n-1}$ (whose
invariant Lie subalgebra is $C_n$). From this we readily obtain that
$\wh\sigma$ acts as the identity on $W_{\alpha_n}$.

Next, we have $\sigma(x_{n,m}^-) = x_{n,m}^-$ for any
$m\in\mathbb{Z}$, so we obtain that $\wh{\sigma}(x_{n,m}^-.v) =
x_{n,m}^-.v$. This implies that the vectors in $\wt{W}\cap W_0$ are
indeed $\wh{\sigma}$-invariant. This implies that $\wt{W}^\sigma =
\wt{W}$.

Now we derive from this that for generic $u$ the intersection $W(u)
\cap W_0$ is two-dimensional.

We have the $2n$-dimensional weight subspace $W_0$ of $W$. By
analyzing the $q$-character of $W$, we have found that it decomposes
into a direct sum of $n-1$ two-dimensional subspaces, each containing
two $\ell$-weight subspaces corresponding to a pair of monomials $M_1
\neq M_2$ such that $\sigma(M_1)=M_2$ (and therefore $\wh\sigma$
interchanges them, according to Lemma \ref{whsigma},(2)), and the
two-dimensional $\ell$-weight subspace corresponding to the
$\sigma$-invariant monomial $Y_{n,q^{n-1}}Y_{n,q^{n+1}}^{-1}$ on which
$\wh\sigma$ acts as the identity, as we have shown above. This implies
that the trace of $\wh\sigma$ on $W_0$ is equal to 2.

On the other hand, $W_0$ also has a basis of joint eigenvectors of
$Q_i^{\pm,\on{ss}}(u), i \in I'$, with eigenvalues $\la(u) =
(\la^\pm_i(u))$. By applying the argument of Lemma \ref{whsigma},(2),
we find that $\wh\sigma$ maps such an eigenvector to another one with
eigenvalues $\sigma(\la(u))$. In the limit $u \to 0$ the eigenspaces
of $Q_i^{\pm,\on{ss}}(u), i \in I'$ become $\ell$-weight spaces. Hence
we find from the preceding paragraph that for generic $u$ we have at
least $2(n-1)$ eigenvalues $\la(u)$ which are not $\sigma$-invariant,
and the set of these eigenvalues breaks into $n-1$ pairs, with the
eigenvalues in each pair (and the corresponding one-dimensional
eigenspaces) exchanged by $\sigma$. If the remaining 2 eigenvalues
were not $\sigma$-invariant, then the corresponding 2 one-dimensional
eigenspaces would have to be exchanged by $\sigma$. But then the trace
of $\wh\sigma$ on $W_0$ would be equal to $0$, which is a
contradiction. Therefore for generic $u$ there must be one
$\sigma$-invariant eigenvalue $\la(u)$ with multiplicity 2, such that
the corresponding two-dimensional eigenspace is $\wh\sigma$-invariant
(this is the necessary condition for the trace of $\wh\sigma$ on this
subspace, and hence on $W_0$, to be equal to $2$). But then this
two-dimensional subspace is contained in $W(u)$, and moreover in
$W^\sigma(u)$. This implies that $W(u) \cap W_0 = W^\sigma(u) \cap
W_0$ is two-dimensional.

Since the other weight spaces are $1$-dimensional, it
follows from Remark \ref{remtriv},(2) that for generic
$u$ we have
$$W(u) \simeq \wt{W}$$
as vector spaces graded by $^L\g$-weights, in agreement with
Conjecture \ref{main conj bis},(1). Moreover, we obtain
that $W(u) = W^\sigma(u)$.

Now, setting $\wt{Y}_{i,a} = Y_{i,a}Y_{2n-i,a}$ for $1\leq i\leq n-1$,
we obtain that
$$
\chi_q(\wt{W}) = \wt{Y}_{1,1} + \wt{Y}_{1,q^2}^{-1} \wt{Y}_{2,q} + \cdots +
\wt{Y}_{n-1,q^n}^{-1} Y_{n,q^{n-1}}^2 +  2 Y_{n,q^{n-1}}Y_{n,q^{n+1}}^{-1} 
+ \wt{Y}_{n-1,q^n}Y_{n,q^{n+1}}^{-2} + \cdots +
\wt{Y}_{1,q^{2n}}^{-1}.
$$
The corresponding character is
$$
y_1 + y_1^{-1}y_2 + \cdots + y_{n-1}^{-1}y_n^2  + 2 +
y_n^{-2}y_{n-1} + \cdots + y_1^{-1},
$$
which coincides with the character of the $(2n+2)$-dimensional
fundamental representation $M(W)$ of $U_t({}^L\widehat{\mathfrak{g}})
= U_t(D_{n+1}^{(2)})$. Moreover, as we will see below, the
specialization $\Pi_t$ of corresponding interpolating
$(q,t)$-character gives rise to the $t$-character of this
representation of $U_t(D_{n+1}^{(2)})$.

The above character can also be interpreted as the character of the
direct sum of the $(2n+1)$-dimensional fundamental representation and
the trivial representation of $U_t(\widehat{{}^L\mathfrak{g}}) =
U_t(B_{n}^{(1)})$. But we cannot obtain its $t$-character from the
interpolating $(q,t)$-character below because this $t$-character
contains the monomial $1$ (compare with the discussion in Remark
\ref{remtriv},(3)).

The interpolating $(q,t)$-character of the simple representation
$L(Y_{1,q^{-1}}Y_{1,q})$ of $U_q(C_n^{(1)})$ is given by formula
(\ref{fdeux}) in Example \ref{fdeuxex} for $n = 2$. For general $n$,
the statements about its specializations follow from Proposition
\ref{fundqtr}:

Its specialization under $\overline{\Pi}_t$ it is equal to
$^{\on{f}}\chi_t(W)$.

Its specialization under $\Pi_t$ is the
$t$-character of the fundamental representation $M(W)$ of
$U_t({}^L\widehat{\mathfrak{g}}) = U_t(D_{n+1}^{(2)})$.

Its specialization under $\Pi_t'$ is (after switching the indices $i$
and $n + 1 -i$) is the folded $t$-character of $W'$, the fundamental
representation $L(Y_{n,1})$ of $U_t(A_{2n - 1}^{(1)})$, defined in terms of
the variables $\overline{Y}_{i,a}$ instead of $Y_{i,a}$.

All of this is in agreement with Conjectures \ref{main conj
  bis} and \ref{part3}.

\subsection{Third example: the fundamental representation
  $L(Y_{3,1})$ of $U_q(A_5^{(1)})$}    \label{ex3}

Consider the Lie algebra $\g'$ of type $A_5$ with $\sigma$ being the
unique automorphism of the Dynkin diagram of order 2. Then $\g = C_3$
with $d_1 = d_2 = 1$, $d_3 = 2$ and its Langlands dual Lie algebra is
${}^L\g = B_3$. We also have $\left({}^L\mathfrak{g}\right)' = D_4$
and $^L\widehat{\g} = D_5^{(2)}$.

The fundamental representation
  $L(Y_{3,1})$ of $U_q(\wh{\g'}) = U_q(A_5^{(1)})$ is 20-dimensional.
It has $8$ invariant monomials:
$$\chi_q(\wt{W}) = 
Y_{3,1}
+ Y_{3,q^2}^{-1} Y_{4,q}Y_{2,q}
+ Y_{1,q^2}Y_{5,q^2}Y_{2,q^3}^{-1}Y_{4,q^3}^{-1} Y_{3,q^2}
+ Y_{1,q^2}Y_{5,q^2}Y_{3,q^4}^{-1}
+ Y_{3,q^2}Y_{1,q^4}^{-1}Y_{5,q^4}^{-1}$$
$$+ Y_{1,q^4}^{-1}Y_{5,q^4}^{-1}Y_{2,q^3}Y_{4,q^3}Y_{3,q^4}^{-1}
+ Y_{3,q^4} Y_{2,q^5}^{-1}Y_{4,q^5}^{-1} 
+ Y_{3,q^6}^{-1}.$$
In this case all ordinary weights have multiplicity $1$. Therefore
Remark \ref{remtriv},(2) implies that for any $u$ we have
$$W(u) = \wt{W} = \wh{W} = W^\sigma(u).$$
Setting $\wt{Y}_{1,a} = Y_{1,a}Y_{5,a}$ and $\wt{Y}_{2,a} =
Y_{2,a}Y_{4,a}$, we obtain that $\chi_q(\wt{W})$ equals
$$Y_{3,1}
+ Y_{3,q^2}^{-1} \wt{Y}_{2,q}
+ \wt{Y}_{1,q^2}\wt{Y}_{2,q^3}^{-1} Y_{3,q^2}
+ \wt{Y}_{1,q^2}Y_{3,q^4}^{-1}
+ Y_{3,q^2}\wt{Y}_{1,q^4}^{-1}
+ \wt{Y}_{1,q^4}^{-1}\wt{Y}_{2,q^3}Y_{3,q^4}^{-1}
+ Y_{3,q^4} \wt{Y}_{2,q^5}^{-1} 
+ Y_{3,q^6}^{-1}.$$
The corresponding character is
$$y_3 + y_2y_3^{-1} + y_3 y_2^{-1} y_1 + y_3^{-1} y_1 + y_3 y_1^{-1} + y_3^{-1} y_2 y_1^{-1} + y_2^{-1} y_3 + y_3^{-1},$$
which is the character of the third fundamental representation $M(W)$
of $U_t({}^L\widehat{\mathfrak{g}}) = U_t(D_5^{(2)})$.

Actually, it is also the character of the fundamental representation
$L(Y_{3,1})$ of $U_t(\wh{{}^L\g}) = U_t(B_3^{(1)})$. But we cannot
obtain the $t$-character of this representation from the interpolating
$(q,t)$-character, whereas we can do it for the third fundamental
representation of $U_t(D_5^{(2)})$ (see the discussion in Remark
\ref{remtriv},(3)).

In fact, the interpolating $(q,t)$-character of the simple
representation $L(Y_{3,1})$ of $U_q(\wh{\g}) = U_q(C_3^{(1)})$
was computed in \cite[Section 4.5]{FH1} (note however that there was a
typo there for the monomials $Y_{2,q^5t^3}$ and $Y_{2,q^7t^5}$, which
we correct here):
$$
Y_{3,1} 
+ Y_{3,q^4t^2}^{-1}Y_{2,qt}Y_{2,q^3t}
+ \alpha Y_{1,q^4t^2}Y_{2,qt}Y_{2,q^5t^3}^{-1}
+ Y_{3,q^2t^2}Y_{2,q^3t^3}^{-1}Y_{2,q^5t^3}^{-1}Y_{1,q^4t^2}Y_{1,q^2t^2}
+ \alpha Y_{1,q^6t^4}^{-1}Y_{2,qt}$$
$$+ Y_{3,q^6t^4}^{-1}Y_{1,q^4t^2}Y_{1,q^2t^2}
+ \alpha Y_{3,q^2t^2}Y_{2,q^3t^3}^{-1}Y_{1,q^6t^4}^{-1}Y_{1,q^2t^2}
+ \alpha Y_{3,q^6t^4}^{-1}Y_{1,q^6t^4}^{-1}Y_{1,q^2t^2}Y_{2,q^5t^3}
+ Y_{3,q^2t^2}Y_{1,q^6t^4}^{-1}Y_{1,q^4t^4}^{-1}$$
$$+ \alpha Y_{1,q^2t^2}Y_{2,q^7t^5}^{-1}
+ Y_{3,q^6t^4}^{-1}Y_{2,q^3t^3}Y_{2,q^5t^3}Y_{1,q^6t^4}^{-1}Y_{1,q^4t^4}^{-1}
+ \alpha Y_{2,q^3t^3}Y_{2,q^7t^5}^{-1}Y_{1,q^4t^4}^{-1}
+ Y_{2,q^5t^5}^{-1}Y_{2,q^7t^5}^{-1}Y_{3,q^4t^4}
+ Y_{3,q^8t^6}^{-1}.$$
There are $14$-monomials, $8$ with multiplicity $1$ and $6$ with multiplicity $\alpha$.

Its specialization under $\overline{\Pi}_t$, has $20$ terms, and it is
equal to $\chi_t(W)$ after identification of the variables $Y_{1,a}
\sim Y_{5,a}$, $Y_{2,a} \sim Y_{4,a}$.

Its specialization under $\Pi_t$, has $8$ terms:
$$
Z_{3,1} 
+ Z_{3,t^2}^{-1}Z_{2,-t^2}
+ Z_{3,-t^2}Z_{2,-t^6}^{-1}Z_{1,t^4}
+ Z_{3,-t^4}^{-1}Z_{1,t^4}$$
$$+ Z_{3,-t^2}Z_{1,t^8}^{-1}
+ Z_{3,-t^4}^{-1}Z_{2,-t^6}Z_{1,t^8}^{-1}
+ Z_{2,-t^{10}}^{-1}Y_{3,t^4}
+ Z_{3,t^6}^{-1},$$
and is the $t$-character of
the fundamental representation $M(W)$ of
$U_t({}^L\widehat{\mathfrak{g}}) = U_t(D_5^{(2)})$.

Its specialization under $\Pi_t'$ is the folded $t$-character
$$
\overline{Y}_{3,1} 
+ \overline{Y}_{3,t^2}^{-1}\overline{Y}_{2,t}
+ \overline{Y}_{3,t^2}\overline{Y}_{2,t^3}^{-1}\overline{Y}_{1,t^2}
+ \overline{Y}_{3,t^4}^{-1}\overline{Y}_{1,t^2}
+ \overline{Y}_{3,t^2}\overline{Y}_{1,t^4}^{-1}
+ \overline{Y}_{3,t^4}^{-1}\overline{Y}_{2,t^3}\overline{Y}_{1,t^4}^{-1}
+ \overline{Y}_{2,t^5}^{-1}\overline{Y}_{3,t^4}
+ \overline{Y}_{3,t^6}^{-1}$$
of the fundamental representation $L(Y_{3,1})$ of $U_t\left(\wh{\left({}^L\mathfrak{g}\right)'}\right) = U_t(D_4^{(1)})$, defined in terms of the variables $\overline{Y}_{i,a}$ instead of $Y_{i,a}$.

All of this is in agreement
with Conjectures \ref{main conj bis} and \ref{part3}.

\subsection{Fourth example: the trivalent fundamental
  representation $L(Y_{1,1})$ of $U_q(D_4^{(1)})$}    \label{ex4}

Now we consider the Lie algebra $\g'$ of type $D_4$ with an
automorphism $\sigma$ of order $3$ and denote by $i = 1$ the trivalent
node. Then $\g = G_2$ with $d_1 = 3$, $d_2 = 1$ and its Langlands dual
Lie algebra is ${}^L\g = G_2$ with $d_1^\vee = 1$ and $d_2^\vee = 3$.
We also have $\left({}^L\mathfrak{g}\right)' = D_4$ (but with the
trivalent node now being $i = 2$) and $^L\widehat{\g} = D_4^{(3)}$.

Let us consider the example of the fundamental representation $W =
L(Y_{1,1})$ of $U_q(\wh{\g'}) = U_q(D_4^{(1)})$. Its
$q$-character is computed in \cite[Example 5.3.2]{N0}. It has $27$
monomials, all of multiplicity $1$ except one
  ($Y_{1,q^2}Y_{1,q^4}^{-1}$) of multiplicity $2$.
It has $8$ invariant monomials for an automorphism $\sigma$ of
order $3$, including one of multiplicity $2$:
$$\chi_q(\wt{W}) = 
Y_{1,1}
+ Y_{1,q^2}^{-1}Y_{2,q}Y_{3,q}Y_{4,q}
+ Y_{1,q^2}^2 Y_{2,q^3}^{-1}Y_{3,q^3}^{-1}Y_{4,q^3}^{-1}
+ 2 Y_{1,q^2}Y_{1,q^4}^{-1}
+ Y_{1,q^4}^{-2} Y_{2,q^3}Y_{3,q^3}Y_{4,q^3}$$
$$+ Y_{1,q^4}Y_{2,q^5}^{-1}Y_{3,q^5}^{-1}Y_{4,q^5}^{-1}
+ Y_{1,q^6}^{-1}.$$
All weight spaces of $W$ are one-dimensional, except for the
$0$-weight space $W_0$ of $W$ which is $5$-dimensional,
containing as a proper subspace its intersection with $\wt{W}$, which
is the $2$-dimensional $\ell$-weight space associated to
$Y_{1,q^2}Y_{1,q^4}^{-1}$ (the other $3$ monomials of weight $0$ are
$Y_{2,q}Y_{2,q^5}^{-1}$, $Y_{3,q}Y_{3,q^5}^{-1}$,
$Y_{4,q}Y_{4,q^5}^{-1}$, which are not $\sigma$-invariant).

Let $v$ be a generating vector of the weight space $W_{\alpha_1}$,
which is one-dimensional and whose corresponding monomial is
$Y_{1,q^2}^2Y_{2,q}^{-1}Y_{3,q}^{-1}Y_{4,q}^{-1}$. Under the action of
the $U_q(\wh{\sw}_2)$-subalgebra corresponding to the node $1$, the
vector $v$ generates a representation of dimension $4$ with a
$2$-dimensional $0$-weight space. We obtain as in Section \ref{ex2}
above that  for generic $u$ we have $\on{dim} W(u)\cap
  W_0 = \on{dim} \wt{W}\cap W_0 = 2$.

Since the other weight spaces are $1$-dimensional, it follows from
Remark \ref{remtriv},(2) that for generic $u$ we have
$$W(u) \simeq \wt{W}$$
as vector spaces graded by $^L\g$-weights. Moreover, it follows that
$W^\sigma(u) = W(u)$ and $\wt{W}^\sigma = \wt{W}$.

Setting $\wt{Y}_{2,a} = Y_{2,a}Y_{3,a}Y_{4,a}$, we obtain that
$$
\chi_q(\wt{W}) = Y_{1,1}
+ Y_{1,q^2}^{-1}\wt{Y}_{2,q}
+ Y_{1,q^2}^2 \wt{Y}_{2,q^3}^{-1}
+ 2 Y_{1,q^2}Y_{1,q^4}^{-1}
+ Y_{1,q^4}^{-2} \wt{Y}_{2,q^3}
+ Y_{1,q^4} \wt{Y}_{2,q^5}^{-1}
+ Y_{1,q^6}^{-1}.$$
The corresponding character
$$y_1 + y_2y_1^{-1} + y_1^2 y_2^{-1} y_1 + 2  + y_2 y_1^{-2} +
y_2^{-1} y_1 + y_1^{-1}$$
is the character of the first fundamental representation $M(W)$ of
$U_t({}^L\widehat{\mathfrak{g}}) = U_t(D_4^{(3)})$.

Note that it is also the character of the direct sum of the first
fundamental representation and the trivial one-dimensional
representation of $U_t(\wh{{}^L\g}) = U_t(G_2^{(1)})$. Therefore we
cannot obtain the $t$-character of this representation of
$U_t(G_2^{(1)})$ from the corresponding interpolating
$(q,t)$-character, whereas we can do it for the first fundamental
representation of $U_t(D_4^{(3)})$ (compare with the discussion in
Remark \ref{remtriv},(3)).

In fact, the interpolating $(q,t)$-character of the simple
representation $L(Y_{1,1})$ of $U_q(G_2^{(1)})$ was computed in
\cite[Section 5.2]{FH1}. It has $15$ monomials, $8$ with multiplicity
$1$ and $7$ with multiplicity $\alpha$.

Its specialization under $\overline{\Pi}_t$ has $29$ terms and is
equal to $^{\on{f}}\chi_t(W)$.

Its specialization under $\Pi_t$, has $8$ terms:
$$
Z_{1,1} 
+ Z_{1,t^2}^{-1}Z_{2,-t^3}
+ Z_{2,-t^9}^{-1}Z_{1,-\epsilon t^2}Z_{1,\epsilon^2t^2}
+ Z_{1,\epsilon t^2}Z_{1,\epsilon^2t^4}^{-1}
+ Z_{1,\epsilon^2t^2}Z_{1,-\epsilon t^4}^{-1}$$
$$+ Z_{1,\epsilon^2 t^4}^{-1}Z_{1,-\epsilon t^4}^{-1}Z_{-2,t^9}
+ Z_{-2,t^{15}}Z_{1,t^4}
+ Z_{1,t^6}^{-1}
$$
and is the $t$-character of the
fundamental representation $M(W)$ of $U_t({}^L\widehat{\mathfrak{g}})
= U_t(D_4^{(3)})$.

Its specialization under $\Pi_t'$ is the folded $t$-character
$$
\overline{Y}_{1,1} 
+ \overline{Y}_{1,t^2}^{-1}\overline{Y}_{2,t}
+ \overline{Y}_{2,t^3}^{-1}\overline{Y}_{1,t^2}^2
+ 2 \overline{Y}_{1,t^4}^{-1}\overline{Y}_{1,t^2}
+ \overline{Y}_{2,t^3}\overline{Y}_{1,t^4}^{-2}
+ \overline{Y}_{2,t^5}^{-1}\overline{Y}_{1,t^4}
+ \overline{Y}_{1,t^6}^{-1}$$
of the fundamental representation $L(Y_{1,1})$ of
$U_t\left(\wh{\left({}^L\mathfrak{g}\right)}'\right) =
  U_t(D_4^{(1)})$, defined in terms of the variables
    $\overline{Y}_{i,a}$ instead of $Y_{i,a}$.

All of this is in
agreement with Conjectures \ref{main conj bis} and \ref{part3}.

\subsection{Fifth example: a simple tensor product of 4 fundamental
  representations of $U_q(A_3^{(1)})$}    \label{fifth}

Here we study an example in which $\overline{W}(u)$ is not equal to the
whole space $W(u)$.

We work with the same Lie algebras $\g' = A_3$, $\g = C_2$, 
	${}^L\g = B_2$, $\left({}^L\mathfrak{g}\right)' = A_3$ and
$^L\widehat{\g} = D_3^{(2)}$ as in the first example.

Consider the following simple tensor product of 4 fundamental
representations of $U_q (\wh{\g'}) = U_q(A_3^{(1)})$,
$$W = L(Y_{1,1}^2Y_{3,1}^2) \simeq L(Y_{1,1})^{\otimes 2}\otimes
L(Y_{3,1})^{\otimes 2}.$$
Its highest monomial $Y_{1,1}^2Y_{3,1}^2$ is $\sigma$-invariant.

It is the tensor square of the $16$-dimensional representation $W_1 = L(Y_{1,1}Y_{3,1})$
studied in Section \ref{ex2} (with $n = 2$), whose invariant 
subspace $\wt{W}_1 = W_1(0)$ has dimension $6$. Its dimension is $256$ and its subspace $\wt{W} = W(0)$ contains 
a subspace of dimension $6^2 = 36$ corresponding to the square of the
$q$-character of the invariant subspace $\wt{W}_1$ of $L(Y_{1,1}Y_{3,1})$:
\begin{equation}\label{square}\wt{Y}_{1,1}^2 
+ 2 \wt{Y}_{1,1}\wt{Y}_{1,q^2}^{-1}\wt{Y}_{2,q}^2
+ 4 \wt{Y}_{1,1}\wt{Y}_{2,q}\wt{Y}_{2,q^3}^{-1}
+   \wt{Y}_{1,q^2}^{-2}\wt{Y}_{2,q}^4
+  4 \wt{Y}_{1,q^2}^{-1}\wt{Y}_{2,q}^3\wt{Y}_{2,q^3}^{-1}
+  2 \wt{Y}_{1,1}\wt{Y}_{1,q^2}\wt{Y}_{2,q^3}^{-2}
+  6 \wt{Y}_{2,q}^2\wt{Y}_{2,q^3}^{-2}\end{equation}
$$+  2 \wt{Y}_{1,1}\wt{Y}_{1,q^4}^{-1}
+  4 \wt{Y}_{2,q}\wt{Y}_{2,q^3}^{-3}\wt{Y}_{1,q^2}
+  2 \wt{Y}_{1,q^2}^{-1}\wt{Y}_{1,q^4}^{-1}\wt{Y}_{2,q}^2
+ \wt{Y}_{2,q^3}^{-4}\wt{Y}_{1,q^2}^2
+ 4 \wt{Y}_{2,q}\wt{Y}_{2,q^3}^{-1}\wt{Y}_{1,q^4}^{-1}
+ 2 \wt{Y}_{2,q^3}^{-2}\wt{Y}_{1,q^2}\wt{Y}_{1,q^4}^{-1}
+ \wt{Y}_{1,q^4}^{-2}.$$
However, the analysis of the $q$-character 
$$\chi_q(W) = (Y_{1,1} + Y_{1,q^2}^{-1}Y_{2,q} + Y_{2,q^3}^{-1}Y_{3,q^2} + Y_{3,q^4}^{-1})^2
(Y_{3,1} + Y_{3,q^2}^{-1}Y_{2,q} + Y_{2,q^3}^{-1}Y_{1,q^2} + Y_{1,q^4}^{-1})^2$$
shows that $\wt{W}$ is larger: it has dimension $54$ and $\chi_q(\wt{W})$ equals
$$
\wt{Y}_{1,1}^2 
+ 4 \wt{Y}_{1,1}\wt{Y}_{1,q^2}^{-1}\wt{Y}_{2,q}^2
+ 8 \wt{Y}_{1,1}\wt{Y}_{2,q}\wt{Y}_{2,q^3}^{-1}
+   \wt{Y}_{1,q^2}^{-2}\wt{Y}_{2,q}^4
+  4 \wt{Y}_{1,q^2}^{-1}\wt{Y}_{2,q}^3\wt{Y}_{2,q^3}^{-1}
+  4 \wt{Y}_{1,1}\wt{Y}_{1,q^2}\wt{Y}_{2,q^3}^{-2}
+   6 \wt{Y}_{2,q}^2\wt{Y}_{2,q^3}^{-2}$$
$$+ 4 \wt{Y}_{1,1}\wt{Y}_{1,q^4}^{-1}
+ 4 \wt{Y}_{2,q}\wt{Y}_{2,q^3}^{-3}\wt{Y}_{1,q^2}
+ 4 \wt{Y}_{1,q^2}^{-1}\wt{Y}_{1,q^4}^{-1}\wt{Y}_{2,q}^2
+ \wt{Y}_{2,q^3}^{-4}\wt{Y}_{1,q^2}^2
+ 8 \wt{Y}_{2,q}\wt{Y}_{2,q^3}^{-1}\wt{Y}_{1,q^4}^{-1}
+ 4 \wt{Y}_{2,q^3}^{-2}\wt{Y}_{1,q^2}\wt{Y}_{1,q^4}^{-1}
+ \wt{Y}_{1,q^4}^{-2}.$$
But the corresponding character is {\em not} equal to the character of
a representation of the twisted quantum affine algebra $U_q(D_3^{(2)})$.
Indeed, the dimensions of the simple $U_q(D_3^{(2)})$-modules whose
highest $^L\g$-weight is $y_1^2$ are $36$ (simple tensor 
  square of the first fundamental representation), $35$ ($L(Z_{1,t}Z_{1,t^4})$) or $32$ (KR modules). The dimension of the 
remaining space and the multiplicity of the highest weight $y_2^2$
within it are respectively 
$18$ and $2y_2^2$, $19$ and $2y_2^2$, $22$ and $3y_2^2$. But the dimension of the 
simple modules of highest weight $y_2^2$ are $16$ (simple tensor product of 
fundamental representation), $15$ ($L(Z_{2,t}Z_{2,t^4})$) or $10$ (KR modules).
This means that this is not the character of a representation of $U_q(D_3^{(2)})$.

Hence in this case $\overline{W}(u)$ cannot be equal to the whole
$W(u)$. Rather, it should be a subspace of the space isomorphic to the
one whose $q$-character is (\ref{square}). We expect that $M(W)$ is a
simple tensor square of the first fundamental
  representation of $U_q(D_3^{(2)})$. The computation of the
  specialization $\Pi_t$ of the corresponding interpolating
  $(q,t)$-interpolating character confirms it.

We also note that $M(W)$ is smaller than $\wt{W}^\sigma$ in this case. 
Indeed, the weight spaces associated to the
 $^L\g$-weights $2\omega_1^\vee$, $2\omega_2^\vee$ and $\omega_1^\vee$ have respective dimensions $1$, $4$ and $8$ 
in $\wt{W}$ of dimension $54$. The dimensions of 
the corresponding weight spaces in $\wt{W}^\sigma$ are $1$, $3$ and $6$ 
(indeed, denoting by $v$ a highest weight vector of $W$, the generators for $2\omega_2^\vee$ are $x_{1,0}^-x_{3,0}^-.v$, 
$x_{1,1}^-.x_{3,1}^-.v$, $(x_{1,0}^-x_{3,1}^-+x_{1,1}^-x_{3,0}^-).v$,
and for $\omega_1^\vee$ their images under the action of
$x_{2,0}^-$ and $x_{2,1}^-$). So the dimension of $\wt{W}^\sigma$
is smaller than $51$, 
and then by a symmetry argument (by considering the opposite weights),
it is smaller than $48$.
We have seen that the dimensions of the simple modules over $U_q(D_3^{(2)})$ 
of highest weight $\omega_1^\vee$ are $36$, $35$ or $32$. The dimension of the 
remaining space in $\wt{W}^\sigma$ is smaller than $12$, $13$ and $20$ respectively, and 
the multiplicity of the weights $2\omega_2^\vee$, $\omega_1^\vee$ within are respectively 
$y_2^2 + 2y_1$, $y_2^2 + 2y_1$ and $2y_2^2 + 3y_1$. But the dimensions of the 
simple modules of highest weight $2\omega_2^\vee$ are $16$, $15$  or $10$. Only the last case is possible and it corresponds 
to a KR-module which contributes to $\omega_1^\vee$ with multiplicity $1$. In all cases, the remaining space has dimension 
smaller than $3$ with a space associated to $\omega_1^\vee$ of
dimension $1$. This means that this is not the character of a representation of $U_q(D_3^{(2)})$.

Now let us interpret this in terms of the interpolating
$(q,t)$-character of the simple representation $L(Y_{1,q^{-1}}Y_{1,q})^{\otimes 2}$ of $U_q(C_2^{(1)})$, 
that is the square of the formula (\ref{fdeux}) studied in Example \ref{fdeuxex}.

Its specialization under $\overline{\Pi}_t$ is equal to
$^{\on{f}}\chi_t(W)$.

Its specialization under $\Pi_t$ is the
$t$-character of the $36$-dimensional simple representation $M(W) =
L(Z_{1,-1})^{\otimes 2}$ of $U_t({}^L\widehat{\mathfrak{g}}) =
U_t(D_3^{(2)})$.

Its specialization under $\Pi_t'$ is (after switching the indices $1$ and
$2$) the folded $t$-character of $W'$, the $36$-dimensional simple
representation $L(Y_{2,1})^{\otimes 2}$ of $U_t(A_3^{(1)})$, defined
in terms of the variables $\overline{Y}_{i,a}$ instead of $Y_{i,a}$.

\section{Connection to monomial crystals}     \label{cryssec}

In this section we formulate a conjecture linking folded
$t$-characters to Kashiwara's extension of Nakajima's monomial model
of crystals to non-simple laced Lie algebras. We have previously mentioned
it in Remark \ref{crystal1},(2).

Recall that a $\mathfrak{g}$-{\em crystal} is a set $C$ together with
an assignment to
each element $c\in C$ a $\mathfrak{g}$-weight $\text{wt}(c)$;
the crystal operators $\wt{e}_i,\wt{f}_i : C\rightarrow
C\sqcup\{0\}, i \in I$ (crystallized versions of Chevalley operators
of $U_q(\mathfrak{g})$);
and maps $\epsilon_i,\phi_i : C\rightarrow \mathbb{Z}$  
satisfying the axioms of the crystal theory (in general,
$\epsilon_i,\phi_i$ could have infinite values,
but we do not consider this possibility here). There is a
corresponding notion of morphisms and isomorphisms of
$\mathfrak{g}$-crystals as well as that
of $\mathfrak{g}$-subcrystals.
For example, each simple finite-dimensional 
representation $V(\lambda)$ of $U_q(\mathfrak{g})$ of highest weight
$\lambda$ is known to have
a crystal basis.  The set of its elements has the structure of a
$\mathfrak{g}$-crystal, which is denoted by $\mathcal{B}(\lambda)$. 
This $\mathfrak{g}$-crystal is called a simple crystal and it has
various realizations. One of them is the {\em monomial crystal} which
was introduced by
Nakajima \cite{N} and further studied by Kashiwara \cite{K}. We now
recall the definition.

Let $s\colon I\rightarrow \{0,1\}$ ($i\mapsto s_i$) be a map such that
$C_{i,j}\leq -1$ implies $s_i+s_j=1$.  Consider the set of monomials
$\mathcal{M}$ generated by the $Y_{i,q^{r}}^{\pm 1}$ such that $r \equiv
s_i \text{ mod } [2]$.

For $m = \prod_{j\in I, l\in\mathbb{Z}}
Y_{j,q^l}^{u_{j,l}}\in\mathcal{M}$, define its $\mathfrak{g}$-weight
by the formula
$$
\text{wt}(m)=\sum_{j\in I,l\in \mathbb{Z}}u_{j,l}\omega_j.$$ 
Next, set for $i\in I$,
$$\phi_i(m)=\max\{ \phi_{i,L}(m) \mid L\in\ZZ\}\text{ where }
\phi_{i,L}(m) = \sum_{l\leq L} u_{i,l}(m),$$
$$\epsilon_i(m) = \max\{ \epsilon_{i,L}(m) \mid L\in\ZZ\}\text{ where }\epsilon_{i,L}(m)=-{\sum_{l\geq L}}u_{i,l}(m).$$
Finally, define $\wt{e}_i,\wt{f}_i\colon \mathcal{M}\to
\mathcal{M}\cup\{0\}$ for $i\in I$ by the formulas
\begin{equation*}
\begin{split}
  & \wt{e}_i(m) =
  \begin{cases}
    0 & \text{if $\epsilon_i(m) = 0$},
    \\
   mA_{i,q^{p_i(m)-1}} & \text{if $\epsilon_i(m) > 0$},
  \end{cases}
\\
  & \wt{f}_i(m) =
  \begin{cases}
    0 & \text{if $\phi_i(m) = 0$},
    \\
   mA_{i,q^{q_i(m)+1}}^{-1} & \text{if $\phi_i(m) > 0$},
  \end{cases}
\end{split}
\end{equation*}
where the $A_{j,q^d}$ are given by formula (\ref{AY2}) and 
$$p_i(m)=\max\{L\in\ZZ\mid
\epsilon_{i,L}(m)=\epsilon_i(m)\}=\max\{L\in\ZZ\mid
        {\sum_{l<L}}u_{i,l}(m)=\phi_i(m)\},$$
$$q_i(m)=\text{min}\{L\in\ZZ\mid
        \phi_{i,L}(m)=\phi_i(m)\}=\text{min}\{L\in\ZZ\mid -{\sum_{l>
            L}}u_{i,l}(m)=\epsilon_i(m)\}.
        $$

\begin{thm}[\cite{N,K}]    \label{crystal}
(1) The collection $(\mathcal{M},\on{wt}, \epsilon_i, \phi_i,
\wt{e}_i,\wt{f}_i)$ is a $\mathfrak{g}$-crystal. It is called
the monomial crystal.

(2) For a dominant $m\in\mathcal{M}$, let $\mathcal{M}(m)\subset
\mathcal{M}$ be the $\mathfrak{g}$-subcrystal generated by $m$. Then
$\mathcal{M}(m)$ is isomorphic, as a $\mathfrak{g}$-crystal, to
$\mathcal{B}(\on{wt}(m))$.
\end{thm}

For example, for $\mathfrak{g} = A_2$ we have 
\begin{equation}\label{mca2}\mathcal{M}(Y_{1,1}) = \{Y_{1,1},
  Y_{2,q}Y_{1,q^2}^{-1}, Y_{2,q^3}^{-1}\}\simeq \mathcal{B}(\omega_1)
\end{equation}
and for $\mathfrak{g} = C_2$ we have
\begin{equation}\label{mcc2}\mathcal{M}(Y_{2,1}) = \{Y_{2,1},
  Y_{2,q^2}^{-1} Y_{1,q}^2,
Y_{1,q}Y_{1,q^3}^{-1},Y_{1,q^3}^{-2}Y_{2,q^2},Y_{2,q^4}^{-1}\}\simeq
\mathcal{B}(\omega_2).
\end{equation}

Observe that the first set (\ref{mca2}) is the set of monomials of the
$q$-character of a fundamental representation of $U_q(A_2^{(1)})$, but
the second set (\ref{mcc2}) in not the set of monomials of the
$q$-character of a representation of $U_q(C_2^{(1)})$.

 In fact, if $\g$ is a simply-laced Lie algebra, there is
  a precise relation between the monomial crystals and the
  $q$-characters of modules over $U_q(\ghat)$ discovered by
  Nakajima \cite{N}. Namely, it is proved in \cite[Theorem 3.3]{N}
  that for a dominant monomial $m$ in $\mathcal{M}$, the set ${\mc
    V}(m)$ of monomials (without multiplicities) of the $q$-character
  of the standard module $V(m)$ associated to $m$ (i.e.  $V(m)$ is the
  corresponding tensor product of the fundamental representations) is
  a $\mathfrak{g}$-subcrystal of the monomial crystal $\mathcal{M}$
  (in particular, its union with $\{0\}$ is stable by the crystal
  operators). Moreover, for each dominant $\g$-weight $\lambda$, there
  is a choice of a dominant monomial $m\in\mathcal{M}$ of weight
  $\text{wt}(m) = \lambda$ so that the $\g$-crystal ${\mc V}(m)$ is
  equal to $\mathcal{M}(m)$. By Theorem \ref{crystal},(2), ${\mc
    V}(m)$ is isomorphic, as a $\mathfrak{g}$-crystal, to the
  corresponding simple crystal $\mathcal{B}(\lambda)$
  \cite[Proposition 3.4]{N} (for a fundamental weight $\lambda =
  \omega_i$ we can choose $m = Y_{i,a}$).

We will now conjecture an analogous result for non-simply laced $\g$,
with the $q$-characters of representations of quantum affine algebras
in $\mathcal{K}_q^+(\mathfrak{g})$ replaced by the $t$-characters in
the {\em folded $t$-character ring} $\mathcal{K}_t^-(\mathfrak{g})$
introduced in Section \ref{Kminus}. Note that we now denote the quantum
  parameter in $\mathcal{M}$ by $t$ instead of $q$ to fit the
  notation of the previous sections.

\begin{conj}    \label{conjcrys}
(1) The set of monomials (without multiplicities) of a product of the
fundamental elements $F(Y_{i,t^r})\in\mathcal{K}_t^-(\mathfrak{g})$,
with $Y_{i,t^r}\in\mathcal{M}$, is a $\mathfrak{g}$-subcrystal of the
monomial crystal $\mathcal{M}$.

(2) For each dominant weight $\lambda$, there is a product of the
fundamental elements $F(Y_{i,a})$ with highest monomial
$m\in\mathcal{M}$ of weight $\lambda$ whose set of monomials is equal
to $\mathcal{M}(m)$ and hence is a $\mathfrak{g}$-crystal isomorphic
to $\mathcal{B}(\lambda)$.
\end{conj}

In particular, for $i\in I$, we expect that the set of monomials
(without multiplicities) occurring in
$F(Y_{i,1})\in\mathcal{K}_t^-(\mathfrak{g})$ is equal to
  $\mathcal{M}(Y_{i,1})$ and so is isomorphic, as a
  $\mathfrak{g}$-crystal, to $\mathcal{B}(\omega_i)$. We also expect
that $\wt{W} = W'$ in this case (see Conjecture \ref{part3}).

For example, we have the following element in $\mathcal{K}_t^-(C_2)$:
$$Y_{2,1} + Y_{2,t^2}^{-1} Y_{1,t}^2 + 2 Y_{1,t}Y_{1,t^3}^{-1} 
+ Y_{1,t^3}^{-2}Y_{2,t^2} + Y_{2,t^4}^{-1}.$$
The set of its monomials is the set (\ref{mcc2}), which is equal to
$\mathcal{M}(Y_{2,1})$, and so it
is isomorphic, as a $\mathfrak{g}=C_2$-crystal, to $\mathcal{B}(\omega_2)$.

One of our motivations for this Conjecture in the following. Consider
the representation $L(\wt{Y}_{i,1})$ of
$U_q(\widehat{\left({}^L\mathfrak{g}\right)'})$. It is a simple tensor
product of fundamental representations.  By Theorem \ref{crystal}, the
set of monomials of its $q$-character,
$\mathcal{M}(L(\wt{Y}_{i,1}))$, has the structure of a simple
$\left({}^L\mathfrak{g}\right)'$-crystal.

Now, the
action of $\sigma$ on $\mathcal{M}(L(\wt{Y}_{i,1}))$ is an
automorphism of $\left({}^L\mathfrak{g}\right)'$-crystal. It then
follows from the general results on
crystals in \cite{Ksy} that the subset of monomials
$(\mathcal{M}(L(\wt{Y}_{i,1})))^\sigma$ fixed by $\sigma$ has the
structure of a simple ${}^L({}^L\g) = \g$-crystal isomorphic to
$\mathcal{B}(\omega_i)$.

On the other hand, the
identification of variables $Y_{j,b}\equiv Y_{j,\sigma(b)}$ is
injective on $(\mathcal{M}(L(\wt{Y}_{i,a})))^\sigma$. Hence the set
$(\mathcal{M}(L(\wt{Y}_{i,1})))^\sigma$ can be identified with a
subset $(\mathcal{M}(L(\wt{Y}_{i,a})))^{\sigma,f}$ of the set of monomials of 
the folded $t$-character
	$$\chi_t^f(L(\wt{Y}_{i,a})) = F(\overline{Y}_{i,a})\in
\mathcal{K}_t^-({}^L\g).$$ 
Hence this subset $(\mathcal{M}(L(\wt{Y}_{i,a})))^{\sigma,f}$ inherits
the structure of a simple $\g$-crystal.

In Proposition \ref{fundqtr} we have established that
$\chi_t^L(L(\wt{Y}_{i,a}))$ is the specialization of the interpolating
$(q,t)$-character $F_{q,t}(W_{i,1})\in
\overline{\mathcal{K}}_{q,t}({}^L\g)$ under $\overline{\Pi}_t$.
But the specialization under $\Pi_t'$ of
the same interpolating $(q,t)$-character $F_{q,t}(W_{i,1})$ is
$F(\overline{Y}_{i,1})\in
\mathcal{K}_t^-(\g)$ (defined in terms of the variables
$\overline{Y}_{j,b}$ instead of $Y_{j,b}$).

We expect that the monomials of $F(\overline{Y}_{i,1})$ correspond
through this interpolation to the monomials in the subset
$(\mathcal{M}(L(\wt{Y}_{i,a})))^{\sigma,f}$, and so the set of these
monomials should inherit the structure of a $\g$-crystal. 
  This leads us to the statement of Conjecture \ref{conjcrys}. Proving
this identification requires a finer analysis. We plan to come back
to this question in another paper.

For example, consider $\mathfrak{g} = C_2$ as above. Then $\left({}^L\mathfrak{g}\right)' = A_3$ and the 
set of monomials of $L(\wt{Y}_{1,1}) \simeq L(Y_{1,0})\otimes L(Y_{3,1})$ has $15$ elements 
$$\mathcal{M}(\wt{Y}_{1,1}) = \{Y_{1,1},Y_{1,q^2}^{-1}Y_{2,q},Y_{2,q^3}^{-1}Y_{3,q^2},Y_{3,q^4}^{-1}\}
\times \{Y_{3,1},Y_{3,q^2}^{-1}Y_{2,q},Y_{2,q^3}^{-1}Y_{1,q^2},Y_{1,q^4}^{-1}\},$$
($Y_{2,1}Y_{2,3}^{-1}$ occurs with multiplicity $2$ in the $q$-character). The set of fixed monomials is 
$$\mathcal{M}(\wt{Y}_{1,1})^{\sigma} = \{\wt{Y}_{1,1},\wt{Y}_{1,q^2}^{-1}\wt{Y}_{2,q}^2,\wt{Y}_{2,q}\wt{Y}_{2,q^3}^{-1},
\wt{Y}_{2,q^3}^{-2}\wt{Y}_{1,q^2},\wt{Y}_{1,q^4}^{-1}\}.$$
It has the structure of a $C_2$-crystal isomorphic to $\mathcal{B}(\omega_2)$ by the general results mentioned above. In this 
explicit example, we can check directly that the above identification gives the $C_2$-crystal (\ref{mcc2}).
Indeed, by folding, this set is identified with
$$\mathcal{M}(\wt{Y}_{1,1})^{\sigma,f} = \{Y_{1,1}^2,Y_{1,q^2}^{-2}Y_{2,q}^2,Y_{2,q}Y_{2,q^3}^{-1}, 
Y_{2,q^3}^{-2}Y_{1,q^2}^2,Y_{1,q^4}^{-2}\},$$
which is a subset of the set of monomials of $\chi_t^f(L(\wt{Y}_{1,1}))$.
The corresponding interpolating $(q,t)$-character $\overline{F}_{q,t}(W_{1,1})$ is given by Formula (\ref{fdeux}) 
in type $B_2$ (with $1$ and $2$ exchanged), see Example \ref{fdeuxex}.
Its specialization under $\overline{\Pi}_t$ is the folded $t$-character of the representation $L(\wt{Y}_{1,1})$ 
above. Its specialization under $\Pi_t'$ is the $t$-character of an element in $\mathcal{K}_t^-(C_2)$, whose set of monomials 
is the set $\mathcal{M}(\wt{Y}_{1,1})^{\sigma,f}$ and is equal to
$$\{\overline{Y}_{1,1},\overline{Y}_{1,q^2}^{-1}\overline{Y}_{2,q}^2,\overline{Y}_{2,q}\overline{Y}_{2,q^3}^{-1}, 
\overline{Y}_{2,q^3}^{-2}\overline{Y}_{1,q^2},\overline{Y}_{1,q^4}^{-1}\}.$$
Exchanging the indices $1$ and $2$, we recover the set (\ref{mcc2}) (with the variables $\overline{Y}_{i,a}$ instead of the $Y_{i,a}$) as expected.

\section{The Gaudin model}    \label{gaudin}

In this section, we consider the Gaudin limit of the folded integrable
model discussed above. In this case, Bethe Ansatz equations simplify
and we can study more directly the links between the objects
associated to the Lie algebras $\g, {}^L\g$, and $\g'$.

\subsection{The appearance of $^L\g$}    \label{appear}

First, some general observations about finite-dimensional
representations of these three Lie algebras. Recall that $\g'$ is a
simply-laced Lie algebra with an automorphism $\sigma$ of order $2$ or
$3$, and $\g$ is its Lie subalgebra of $\sigma$-invariant
elements. Let $\h'$ be a Cartan subalgebra of $\g'$. Then its subspace
$ \h$ of $\sigma$-invariant elements is a Cartan subalgebra of $\g$.

Consider for simplicity the case when $\sigma$ has order 2. The case
of order 3 can be analyzed in a similar way.

The Cartan subalgebra $\h'$ is generated by the coroots $\chal'_i, i
\in I'$.  Here $I'$ is the set of vertices of the Dynkin diagram of
$\g'$. Recall that $\sigma$ acts on $I'$, and the set $I$ of vertices
of the Dynkin diagram of $\g$ is the quotient of $I'$ by this
action. For each $i \in I$, let $J_i$ be the preimage of $i$ in
$I'$. It either consists of one element stable by $\sigma$, or two
elements exchanged by $ \sigma$. The coroot generators $\chal_i,
i \in I$, of $\h$ are
\begin{equation}    \label{hi}
\chal_i = \sum_{j \in J_i} \chal'_j.
\end{equation}
In other words, they have the form
\begin{equation}    \label{hi1}
\chal'_i, \quad \on{if} \; \sigma(i)=i; \qquad  \chal'_i +
\chal'_{\sigma(i)}, \quad \on{if} \; \sigma(i) \neq i.
\end{equation}

Consider now the dual spaces $(\h')^*$ and $\h^*$. We have a
surjective map $ (\h')^* \twoheadrightarrow \h^*$ dual to the
inclusion $\h \hookrightarrow \h'$. We can try to embed $\h^*$ into $
(\h')^*$ so that the pairing between the image of this embedding of
$\h^*$ and $\h \subset \h'$ is the pairing we have on $\h$, but then
an interesting thing happens: if $\sigma(i)=i$, then we can take as
the image of the corresponding fundamental weight $\omega_i$ of $\g$
in $(\h')^*$ to be equal to $\omega'_i$, the $i$th fundamental weight
of $\g'$. But if $\sigma(i) \neq i$, then we have to take {\em half}
the sum of the fundamental weights:
\begin{equation}    \label{omega}
\frac{1}{2}(\omega'_i + \omega'_{\sigma(i)}).
\end{equation}
Indeed, we need to have $\langle \omega_i,\chal_j \rangle =
\delta_{ij}$, and then formula \eqref{hi1} shows that we must insert
the factor $\frac{1}{2}$. The same factor appears in the expressions
for the simple roots of $\g$.

This has an important consequence: the integral weight lattice
associated to $\g$ does {\em not} embed into the integral weight
lattice of $\g'$. Instead, we have an embedding of the integral weight
lattice associated not to $\g$ but to $^L\g$ -- the Langlands dual Lie
algebra! -- into the integral weight lattice of $\g'$.

Actually, this is clear from the fact that $^L\h = \h^*$, and so
$({}^L\h)^* = \h$, which, as we have seen above, naturally embeds into
$\h'$. We then identify the latter with $ (\h')^*$ using the unique
$W$-invariant bilinear form normalized so that the square length of
each root of $\g'$ is equal to $2$. Thus, we obtain an embedding
$({}^L\h)^* \hookrightarrow (\h')^*$. Under this embedding, a
fundamental weight of $^L\g$ is mapped to
\begin{equation}    \label{omega1}
\omega'_i + \omega'_{\sigma(i)}.
\end{equation}
Unlike formula \eqref{omega}, there is {\em no} factor $\frac{1}{2}$
now. The same holds for the simple roots of $^L\g$. Thus, we have
proved the statement of Lemma \ref{bij} that there is a natural
isomorphism between the lattice of integral $^L\g$-weights and the
lattice of $\sigma$-invariant $\g'$-weights.

Now suppose that we have a finite-dimensional irreducible
representation $W$ of $\g'$ whose highest weight $\lambda$ is
$\sigma$-invariant. By analogy with Lemma \ref{whsigma}, we then prove
that there is a unique automorphism $\wh\sigma$ of $W$ which
intertwines the representation of $\g'$ on $W$ with its $\sigma$-twist
and acts as identity on the highest weight subspace. As in Section
\ref{olW}, define the subspace $W^\sigma \subset W$ as the direct sum
of the $\wh\sigma$-invariant parts of the weight subspaces of $W$
corresponding to $\sigma$-invariant weights. According to Proposition
\ref{virt}, its character is the character of a virtual representation
of $^L\g$. In Section \ref{N=1}, we will show that in fact it is
possible (modulo Conjecture \ref{completeness}) to explicitly embed
the irreducible representation of $^L\g$ with the highest weight
corresponding to $\lambda$ into $W^\sigma$.

So, an interesting switch happens: we start with a simply-laced Lie
algebra $\g'$ with an automorphism $\sigma$ whose fixed-point Lie
subalgebra is $\g$. Of course, every irreducible representation $W$ of
$\g'$ restricts to a representation of $\g$. However, because the
integral weight lattice of $\g$ does not embed into the integral
weight lattice of $\g'$, it's not natural to describe it in terms of
weight spaces corresponding to the $\sigma$-invariant weights of
$\g'$. Instead, it turns out that we can construct inside the direct
sum of $\sigma$-invariant weight subspaces of $W$ (actually, inside
its subspace $W^\sigma$) an irreducible representation of the {\em
  Langlands dual Lie algebra} $^L\g$, which at first glance has
nothing to do with $\g'$. This can be done using the the results of
\cite{FFRyb} on the spectra of the Hamiltonians of the Gaudin
model.

\begin{rem}
In \cite{FH0} and references therein it was shown that representations
of $^L\g$ can be extracted from irreducible finite-dimensional
representations of $\g$. In contrast, here we aim to extract
representations of $^L\g$ from irreducible representations of
$\g'$. It would be interesting to see whether there is a connection
between the two approaches.\qed
\end{rem}

\subsection{Gaudin model for $\g'$}    \label{Gprime}

Here we discuss the Gaudin model associated to $\g'$, more precisely,
its modification with the twist parameter $\chi' \in \h'$. It appears
as the $q \to 1$ limit of the XXZ-type quantum integrable model
associated to $U_q(\wh{\g'})$ ($\chi'$ is the analogue of the parameter
$u \in H'$ of the XXZ-type model).

We will use the results of \cite{FFR,F:gaudin,Ryb,FFTL,FFRyb}.  Let
$\la_k, k=1,\ldots,N$, be a collection of dominant integral weights of
$\g'$. Denote by $V'_{\la_k}$ the corresponding irreducible
finite-dimensional representations of $\g'$. For a collection $z_k,
k=1,\ldots,N$, of distinct complex numbers, let $V'_{\la_k}(z_k),
k=1,\ldots,N$, be the corresponding evaluation representations of
the current algebra $\g'[z]$. Consider their tensor product
\begin{equation}    \label{qs}
\bigotimes_{k=1}^N V'_{\la_k}(z_k)
\end{equation}
and its subspace of weight
\begin{equation}    \label{ga}
\gamma = \sum_k \la_k - \sum_{j=1}^m \al'_{i_j}.
\end{equation}
These are the spaces of states of the Gaudin model associated to
$\g'$.

Next we discuss the Bethe Ansatz equations. Their solutions are in
one-to-one correspondence with the {\em Miura $\g'$-opers}
on $\pone$ with trivial monodromy representation. They can be
represented by $\g'$-valued connections on $\pone$ of the form
\begin{equation}    \label{Miura}
\pa_z + \sum_{i \in I'} f'_i - \sum_{k=1}^N \frac{\chla_k}{z-z_k} + 
\sum_{j=1}^m \frac{\chal'_{i_j}}{z-w_j} + \chi'.
\end{equation}
Here the $f'_i, i \in I'$, are generators of the lower nilpotent
subalgebra of $\g'$, and we are using the above identification between
$\h'$ and $(\h')^*$ by means of the normalized bilinear form; namely,
$\{ \chla_k \}$ denote the elements of $\h' \subset \g'$ corresponding
to $\{ \la_k \in (\h')^* \}$, and $\{ \chal'_i, i \in I' \}$ is the
set of simple coroots of $\g'$, which correspond to the simple roots
$\{ \al'_i, i \in I'\}$ under the identification $(\h')^* \simeq \h$.

We will assume that $w_j$'s are distinct complex numbers such that
$w_j \neq z_k$ for all $j$ and $k$.

\begin{prop}[\cite{FFR,F:gaudin,FFTL}]    \label{nomon}
The condition that connection \eqref{Miura} has trivial monodromy is
equivalent to the following system of Bethe Ansatz equations:
\begin{equation}    \label{BAEgaudin}
\sum_{k=1}^N \frac{\langle \al'_{i_j},\chla_k \rangle}{w_j-z_k} -
\sum_{s \neq j} \frac{\langle \al'_{i_j},\chal'_{i_s}
  \rangle}{w_j-w_s} = \langle \al'_{i_j},\chi' \rangle, \qquad
j=1,\ldots,m.
\end{equation}
\end{prop}

These equations can be obtained by taking the limit $q \to 1$ of the
Bethe Ansatz equations \eqref{bae gen} of the XXZ-type model
associated to $U_q(\wh{\g'})$. The monic polynomial with the roots
$w_j$ where $i_j=i$ (i.e. those points on $\pone \backslash
\infty$ at which the connection \eqref{Miura} has residue
$\chal'_i$),
\begin{equation}    \label{Qr}
Q_i(z) = \prod_{i_j=i} (z-w_j),
\end{equation}
is the analogue of the Baxter polynomial $Q_i(z), i \in I'$ (see
Theorem \ref{eigg}) in the Gaudin model.

Note that if we apply a permutation to the set of roots of each
polynomial $Q_i(z)$ (which we recall are assumed to be distinct) we
also obtain a solution of the system \eqref{BAEgaudin}. In what
follows, by a {\em solution} of Bethe Ansatz equations we will
understand an equivalence class of solutions under these permutations.

In \cite{FFR,F:gaudin,FFTL} a joint eigenvector of the Gaudin
Hamiltonians (Bethe vector) is constructed for each solution of the
BAE \eqref{BAEgaudin}. It is known, however, that for general
$\la_k$'s and $\ga$ these Bethe vectors do not yield a basis of the
weight subspace of \eqref{qs} of weight $\ga$ given by formula
\eqref{ga}. In fact, explicit examples have been constructed in
\cite{MV:bethe} showing that for $\chi = 0$ this is so even for
generic $z_k$'s (and fixed $\la_k$'s and $\ga$). We are not aware of
such counterexamples for generic $\chi'$, so it is possible that in
this case Bethe vectors do yield a basis for generic $z_k$'s. For
$N=1$, this is the statement of Conjecture \ref{completeness} below.

As explained in \cite[Sect. 5.5]{F:gaudin} (in the case $\chi'=0$, but
the picture is similar for all $\chi'$) that the true parameters of
the spectra of the quantum Gaudin Hamiltonians are not the Miura
$\g'$-opers of the form \eqref{Miura} but rather $\g'$-opers on
$\pone$ with trivial monodromy, regular singularities at
$z_1,\ldots,z_N$ with respective residues $\chla_1,\ldots,\chla_N$
(coweights of $\g'$ corresponding to the weight $\la_1,\ldots,\la_N$),
and irregular singularity at $\infty$ of order $2$ with $2$-residue
corresponding to $\chi'$ (see \cite{FFTL,FFRyb} for more
details). Denote the set of these $\g'$-opers by
$\on{Op}^\chi_{(\chla_i),(z_i)}(\g')$.

\begin{thm}[\cite{FFRyb}]    \label{gsp1}
For any regular $\chi' \in \h'$ and any collection $z_1\ldots,z_N$
there is a bijection between $\on{Op}^\chi_{(\chla_i),(z_i)}(\g')$ and
the spectrum of the Gaudin Hamiltonians on the space \eqref{qs}.
Moreover, the eigenspace corresponding to every $\g'$-oper in
$\on{Op}^\chi_{(\chla_i),(z_i)}(\g')$ is always one-dimensional. In
addition, for generic $z_1\ldots,z_N$ and $\chi'$ the Gaudin
Hamiltonians are diagonalizable.
\end{thm}

In other words, for any regular $\chi$ and general $z_1\ldots,z_N$
there is at most one Jordan block for each joint eigenvalue of the
Gaudin Hamiltonians (this is expressed in \cite[Corollary 5]{FFRyb}
as the statement that the commutative algebra of Gaudin Hamiltonians
has a cyclic vector in \eqref{qs}). Moreover, for generic $\chi$ and
$z_1\ldots,z_N$ all Jordan blocks have size $1$ (see \cite[Corollary
  6]{FFRyb}).

Following the argument in \cite[Sect. 5.5]{F:gaudin} in the case
$\chi'=0$, one can show that every $\g'$-oper in
$\on{Op}^\chi_{(\chla_i),(z_i)}(\g')$ corresponds to a unique Miura
$\g'$-oper of the form
\begin{equation}    \label{Miura-1}
\pa_z + \sum_{i \in I'} f'_i - \sum_{k=1}^N
\frac{g_k(\chla_k+\crho') - \crho'}{z-z_k} + 
\sum_{j=1}^m \frac{\crho' - \wt{g}_j(\crho')}{z-w_j} + \chi',
\end{equation}
where the $w_j, j=1,\ldots,m$, are distinct complex numbers such that
$w_j \neq z_k$ for all $j$ and $k$; $g_k, k=1,\ldots,N$, and
$\wt{g}_j, i=1,\ldots,m$, are elements of the Weyl group of $\g'$; and
$\crho'$ is the sum of the dominant coweights of $\g'$.

A Miura $\g'$-oper of the form \eqref{Miura-1}, and the corresponding
$\g'$-oper, are called {\em non-degenerate} (see
\cite[Sect. 5.2]{F:gaudin}) if $g_k=1$ for all $k=1,\ldots,N$, and
each $\wt{g}_j$ is a simple reflection $s_{i_j}$ from the Weyl group
of $\g'$. Since $\crho' - s_{i_j}(\crho') = \chal_{i_j}'$, a Miura
$\g'$-oper of the form \eqref{Miura-1} is non-degenerate $\g'$-oper if
and only if it has the form \eqref{Miura}. Further, according to
Proposition \ref{nomon}, the no-monodromy condition is then equivalent
to BAE \eqref{BAEgaudin}.

Theorem \ref{gsp1} implies that non-degenerate Miura $\g'$-opers of
the form \eqref{Miura} satisfying \eqref{BAEgaudin} correspond to a
subset in the spectrum of the Gaudin Hamiltonians on the weight
subspace of \eqref{qs} of weight $\ga$. But in general, this is not
the entire spectrum; there may be other joint eigenvalues of the
Gaudin Hamiltonians on this weight subspace which correspond to
degenerate Miura $\g'$-opers of the form \eqref{Miura-1} with at least
one $g_k$ not equal to the identity or at least one of the $\wt{g}_j$
not equal to a simple reflection, or both. For such Miura $\g'$-opers,
the equations expressing the no-monodromy condition are more
complicated, as are the corresponding eigenvectors (though they can
be constructed in principle by a certain algorithm). See
\cite[Sect. 5.5]{F:gaudin} for more details.

\subsection{Gaudin model for $^L\g$}

Next, consider the Gaudin model for the Lie algebra $^L\g$ with a
twist by $\chi \in ({}^L\h)^* = \h$. The space of states of the model
is then a tensor product of evaluation representations of the current
algebra $^L\g[z]$ corresponding to irreducible finite-dimensional
representations of $^L\g$:
\begin{equation}    \label{quantum space}
\bigotimes_{k=1}^N V_{\mu_k}(z_k)
\end{equation}
where $\mu_1,\ldots,\mu_N$ are dominant integral weights of $^L\g$ and
$z_1,\ldots,z_N$ are distinct complex numbers. According to the
results of \cite{FFR,F:gaudin,FFTL,FFRyb}, the spectrum of the Gaudin
Hamiltonians exhibits {\em Langlands duality}; namely, the joint
eigenvalues of the $^L\g$-Gaudin Hamiltonians are described in terms
of $\g$-opers rather than $^L\g$-opers.

More precisely, let $\cmu_k \in \h = ({}^L\h)^*$ be the integral
coweight of $\g$ corresponding to $\mu_k$. Let
$\on{Op}^\chi_{(\cmu_i),(z_i)}(\g)$ be the set of $\g$-opers on
$\pone$ with trivial monodromy, regular singularities at
$z_1,\ldots,z_N$ with respective residues $\cmu_1,\ldots,\cmu_N$, and
irregular singularity at $\infty$ of order $2$ with $2$-residue
corresponding to $\chi$ (see \cite{FFTL,FFRyb} for details). Then we
have the following analogue of Theorem \ref{gsp1}.

\begin{thm}[\cite{FFRyb}]    \label{gsp2}
For any regular $\chi \in \h$ and any collection $z_1\ldots,z_N$ there
is a bijection between $\on{Op}^\chi_{(\cmu_i),(z_i)}(\g)$ and the
spectrum of the $^L\g$-Gaudin Hamiltonians on the space of
\eqref{quantum space}. Moreover, the eigenspace corresponding to every
$\g$-oper in $\on{Op}^\chi_{(\cmu_i),(z_i)}(\g)$ is always
one-dimensional. In addition, for generic $z_1\ldots,z_N$ and $\chi$
the Gaudin Hamiltonians are diagonalizable.
\end{thm}

\subsection{Embedding of the space of states}

This theorem has an intriguing consequence. Note that each coweight
$\cmu_k$ of $\g$ defines a coweight $\cmu'_k$ of $\g'$ while the
corresponding weight $\mu_k$ of $^L\g$ defines a weight of $\g'$,
and the element $\chi \in \h$ defines an element $\chi' \in \h'$.

Next, define the nilpotent generators $f_i, i \in I$, of $\g$ by the
formulas
$$
f_i = f'_i \quad \on{if} \quad \sigma(i)=i, \qquad f_i =
f'_i+f'_{\sigma(i)} \quad \on{if} \quad \sigma(i) \neq i.
$$
Then we obtain that the embedding $\g \hookrightarrow \g'$ maps the
principal nilpotent element of $\g$ to that of $\g'$:
\begin{equation}    \label{principal}
p_- := \sum_{i \in I} f_i \quad \mapsto \quad p'_- := \sum_{i \in I'}
f'_i.
\end{equation}
In the same way, we obtain that the element $\crho$ of $\g$ (the sum
of its dominant coweights) maps to the corresponding element $\crho'$
of $\g'$. Recall that there is a unique nilpotent element $p_+ \in \g$
which is a linear combination of the generators $e_i, i \in I$, such
that $\{ p_-,2\crho,p_+ \}$ is an $\sw_2$ triple in $\g$. Likewise, we
have an $\sw_2$ triple $\{ p'_-,2\crho',p'_+ \}$ in $\g'$. Uniqueness
implies that $p'_+$ is the image of $p_+$ in $\g'$.

Using these $\sw_2$ triples, we construct the Kostant slices $K(\g)$
and $K(\g')$ of regular conjugacy classes in $\g$ and $\g'$,
respectively. Namely,
$$
K(\g) = p_- + \on{span}\{ p_j \}_{j=1,\ldots,\ell},
$$
where $\{ p_j \}_{j=1,\ldots,\ell}$ is a basis of the subspace of
$\on{ad} p_+$-invariants in the upper nilpotent subalgebra $\n_+$ of
$\g$, such that $[\crho,p_j] = d_j p_j$, with $\{ d_1,\ldots,d_\ell \}$
being the set of exponents of $\g$. The Kostant slice $K(\g')$ in
$\g'$ is defined similarly. The fact that the embedding $\g
\hookrightarrow \g'$ maps the above $\sw_2$ triples to each other then
implies that we have a natural embedding
\begin{equation}    \label{Kostant}
K(\g) \hookrightarrow K(\g').
\end{equation}

Beilinson and Drinfeld have constructed \cite{BD:opers} canonical
representatives of $\g$-opers in terms of $K(\g)$ (see
\cite[Sect. 4.2]{FFTL}). Using the embedding \eqref{Kostant}, we
obtain the following.

\begin{lem}    \label{embopers}
There is a natural embedding
\begin{equation}    \label{eopers}
  \on{Op}^\chi_{(\cmu_i),(z_i)}(\g) \hookrightarrow
  \on{Op}^{\chi'}_{(\cmu'_i),(z_i)}(\g').
\end{equation}
\end{lem}

Now we derive from Theorem \ref{gsp2} and Lemma \ref{embopers} the
following surprising result.

\begin{thm}    \label{embqs}
For generic regular $\chi \in \h$, such that $\chi' \in \h'$ is
regular, and generic $z_1,\ldots,z_N$ there exists an embedding of the
tensor product \eqref{quantum space} of simple $^L\g$-modules with
highest weights $\mu_k$ into the tensor product of the simple
$\g'$-modules with the corresponding highest weights $\mu'_k$:
\begin{equation}    \label{emb quantum space}
\bigotimes_{k=1}^N V_{\mu_k}(z_k) \hookrightarrow \bigotimes_{k=1}^N
V'_{\mu'_k}(z_k).
\end{equation}
\end{thm}

\begin{proof}
By Theorem \ref{gsp2}, for generic $\chi$ and $z_1,\ldots,z_N$ there
exists an eigenbasis $\{ e_m \}$ of the Gaudin Hamiltonians in
\eqref{quantum space} labeled by $\g$-opers $m \in
\on{Op}^\chi_{(\cmu_i),(z_i)}(\g)$ with each $\g$-oper corresponding
to exactly one basis vector. Under the embedding \eqref{eopers}, each
of these $\g$-opers gives a $\g'$-oper $m' \in
\on{Op}^{\chi'}_{(\cmu'_i),(z_i)}(\g')$. Applying Theorem \ref{gsp2} in
the case of the Lie algebra $\g'$ and the tensor product of simple
$\g$-modules with the corresponding highest weights $\mu'_k,
k=1,\ldots,N$, we find that each $m'$ obtained this way corresponds to
the one-dimensional eigenspace $E_{m'}$ of the Hamiltonians of the
$\g'$-Gaudin model. Mapping each basis vector $e_m$ in
\eqref{quantum space} to a non-zero vector in $E_{m'}$, we obtain
the desired embedding.
\end{proof}

The embedding \eqref{emb quantum space} is not unique since we can
rescale the images of the basis elements $e_m$ by arbitrary non-zero
numbers. Another issue is that the embedding \eqref{eopers} is defined
in a rather abstract way (using Kostant slices). It turns out that we
can obtain a more concrete realization of \eqref{eopers} if all
$\g$-opers in $\on{Op}^\chi_{(\cmu_i),(z_i)}(\g)$ are non-degenerate,
and hence correspond to solutions of the Bethe Ansatz equations. In
this case we can also construct the corresponding eigenvectors (Bethe
vectors) explicitly, so we obtain a more concrete realization of the
embedding \eqref{emb quantum space} as well. This may be viewed as an
analogue of the statement of our main Conjecture \ref{main conj bis}
in the case of simple Lie algebras. We explain it in the next
subsection.

\begin{rem}
Alternatively, under the conditions of Theorem \ref{embqs} we can
construct a surjective map
\begin{equation}    \label{quot quantum space}
\bigotimes_{k=1}^N V'_{\mu'_k}(z_k) \twoheadrightarrow \bigotimes_{k=1}^N
V_{\mu_k}(z_k)
\end{equation}
instead of an embedding \eqref{emb quantum space}. Indeed, the
inclusion \eqref{eopers} gives rise to a surjective homomorphism of
the corresponding algebras of functions
\begin{equation}    \label{funopers}
\on{Fun} \on{Op}^{\chi'}_{(\cmu'_i),(z_i)}(\g') \twoheadrightarrow
\on{Fun} \on{Op}^\chi_{(\cmu_i),(z_i)}(\g).
\end{equation}
The corresponding algebras of $\g$- and $\g'$-Gaudin Hamiltonians were
denoted in \cite{FFRyb} by ${\mc A}_\chi(z_1,\ldots,z_N)$ and ${\mc
  A}_{\chi'}(z_1,\ldots,z_N)$. According to Theorem \ref{gsp2}, for
generic regular $\chi \in \h$ and generic $z_1,\ldots,z_N$, the algebra
$\on{Fun} \on{Op}^\chi_{(\cmu_i),(z_i)}(\g)$ is equal to the image of
${\mc A}_\chi(z_1,\ldots,z_N)$ in the algebra of endomorphisms of
$\bigotimes_{k=1}^N V_{\mu_k}(z_k)$. Moreover, $\on{Fun}
\on{Op}^\chi_{(\cmu_i),(z_i)}(\g)$ then has a cyclic vector in
$\bigotimes_{k=1}^N V_{\mu_k}(z_k)$. If we choose such a cyclic
vector, we obtain an isomorphism of vector spaces
\begin{equation}    \label{funopers1}
\bigotimes_{k=1}^N V_{\mu_k}(z_k) \simeq
\on{Fun} \on{Op}^\chi_{(\cmu_i),(z_i)}(\g).
\end{equation}

On the other hand, according to Theorem \ref{gsp1}, $\on{Fun}
\on{Op}^{\chi'}_{(\cmu'_i),(z_i)}(\g')$ is the quotient of the image
of ${\mc A}_{\chi'}(z_1,\ldots,z_N)$ in the algebra of endomorphisms
of $\bigotimes_{k=1}^N V_{\mu_k}(z_k)$ by its radical. Choosing a
cyclic vector of ${\mc A}_{\chi'}(z_1,\ldots,z_N)$ in
$\bigotimes_{k=1}^N V'_{\mu'_k}(z_k)$ and using the homomorphism
\eqref{funopers}, we obtain a surjective map \eqref{quot quantum
  space}.

This construction can be generalized to the case when $\chi$ is a
regular nilpotent element of $\g$ and $\chi$ is the corresponding
element of $\g'$. Then we can choose as the cyclic vectors the highest
weight vectors in the tensor products \eqref{quot quantum space}, so
that the resulting map becomes canonical up to a scalar.

Constructing surjective maps \eqref{quot quantum space} may seem more
appealing than constructing the embeddings \eqref{emb quantum
  space}. However, in the case when all $\g$-opers in
$\on{Op}^\chi_{(\cmu_i),(z_i)}(\g)$ are non-degenerate, the embedding
\eqref{emb quantum space} can be linked to the inclusion of the sets
of solutions of the corresponding Bethe Ansatz equations, as we
explain in the next subsection. Therefore we can make contact to the
folded integrable models. (Moreover, the embedding \eqref{emb quantum
  space} can then be constructed rather explicitly using Bethe
vectors, see Section \ref{N=1}.) It is for this reason that we focus
on the embeddings \eqref{emb quantum space} rather than the
surjections \eqref{quot quantum space}, even though the latter also
deserve to be studied.\qed
\end{rem}

\subsection{Embedding in the case of non-degenerate $\g$-opers}

As in the case of $\g'$-opers discussed in Section \ref{Gprime}, each
$\g$-oper in $\on{Op}^\chi_{(\cmu_i),(z_i)}(\g)$ can be represented
in a unique way by a Miura $\g$-oper given by a formula similar to
\eqref{Miura-1}. Among them are the non-degenerate Miura $\g$-opers of
the form
\begin{equation}    \label{Miura2}
\pa_z + \sum_{i \in I} f_i - \sum_{k=1}^N \frac{\cmu_k}{z-z_k} + 
\sum_{j=1}^m \frac{\chal_{i_j}}{z-w_j} + \chi,
\end{equation}
where $\chal_i, i \in I$, are the simple coroots of $\g$. The
no-monodromy condition on \eqref{Miura2} is equivalent to the Bethe
Ansatz equations similar to \eqref{BAEgaudin} (see equation (6.11) in
\cite{FFTL}):
\begin{equation}    \label{BAEgaudin1}
\sum_{k=1}^N \frac{\langle \al_{i_j},\cmu_k \rangle}{w_j-z_k} -
\sum_{s \neq j} \frac{\langle \al_{i_j},\chal_{i_s}
  \rangle}{w_j-w_s} = \langle \al_{i_j},\chi \rangle, \qquad
j=1,\ldots,m.
\end{equation}

The image of a non-degenerate Miura $\g$-oper \eqref{Miura2} under the
embedding \eqref{eopers} is particularly easy to describe. Namely, we
interpret formula \eqref{Miura2} as a Miura $\g'$-oper of the form
\eqref{Miura} as follows: According to formula \eqref{principal}, the
first summation in \eqref{Miura2} is equal to the first summation in
\eqref{Miura}. Next, the second summation in \eqref{Miura2} is equal
to the second summation in \eqref{Miura} with $\chla_k=\cmu'_k$.

Finally, every term in the third summation in \eqref{Miura2} can be
written as the sum of terms appearing in the third summation of
\eqref{Miura}. Namely, for $i \in I'$ such that $\sigma(i) = i$, we
map
\begin{equation}    \label{alinv}
\frac{\chal_i}{z-w_j} \mapsto \frac{\chal'_i}{z-w_j},
\end{equation}
and for $i \in I$ such that if $\sigma(i) \neq i$, we map
\begin{equation}    \label{alnoninv}
\frac{\chal_i}{z-w_j} \mapsto \frac{\chal'_i}{z-w_j} +
\frac{\chal'_{\sigma(i)}}{z-w_j}.
\end{equation}
In other words, each term corresponding to $i \in I$ such that
$\sigma(i) = i$ in \eqref{Miura2} gives us the corresponding term in
\eqref{Miura}, whereas in the case $\sigma(i) \neq i$
it gives us the sum of two terms in \eqref{Miura}.

Observe that if the terms \eqref{alnoninv} are present in
\eqref{Miura2}, then the corresponding $\g'$-Miura oper \eqref{Miura}
is {\em degenerate} because its residue at the corresponding point
$w_j$ is equal to
\begin{equation}    \label{res}
\chal'_{i} + \chal'_{\sigma(i)} = \crho' - s_i s_{\sigma(i)}(\crho')
\end{equation}
(the simple reflections $s_i$ and $s_{\sigma(i)}$ commute with each
other, so it doesn't matter in which order we take their
product).

Thus, the Miura $\g'$-opers we obtain this way have the form
\eqref{Miura-1}, with $g_k=1$ for all $k=1,\ldots,N$, and $\wt{g}_j$
being equal to the simple reflection $s_i$ of the Weyl group of $\g'$
if the residue of this Miura $\g'$-oper at $w_j$ is equal to a
$\sigma$-invariant simple coroot $\chal'_{i}$ of $\g'$, but $\wt{g}_j$
is equal to the product of two simple reflections, $s_i
s_{\sigma(i)}$, if residue at $w_j$ is equal to \eqref{res}. We have
obtained the following result.

\begin{prop}    \label{embopers1}
Under the embedding \eqref{eopers}, every non-degenerate $\g$-oper in
$\on{Op}^\chi_{(\cmu_i),(z_i)}(\g)$ which is represented by a Miura
$\g$-oper of the form \eqref{Miura2} maps to the $\g'$-oper in
$\on{Op}^{\chi'}_{(\cmu'_i),(z_i)}(\g')$ represented by the Miura
$\g'$-oper
\begin{equation}    \label{Miura-2}
\pa_z + \sum_{i \in I'} f_i - \sum_{k=1}^N \frac{\cmu'_k}{z-z_k} + 
\sum_{\sigma(i_j) = i_j} \frac{\chal'_{i_j}}{z-w_j} + \sum_{\sigma(i_j)
  \neq i_j} \frac{\chal'_{i_j}+\chal'_{\sigma(i_j)}}{z-w_j} + \chi'.
\end{equation}
\end{prop}

According to Theorem \ref{gsp1}, the Miura $\g'$-oper \eqref{Miura-2}
defines a point in the spectrum of the $\g'$-Gaudin model.

\subsection{Folded Bethe Ansatz equations}

From the point of view of the Bethe Ansatz equations, we interpret
this as follows. In the system of BAE \eqref{BAEgaudin}, it makes
sense to impose the condition
\begin{equation}    \label{condj}
  \{ w_j \, | \, i_j = i \} = \{ w_j \, | \, i_j = \sigma(i) \}
\end{equation}
for all $i \in I'$. That's because for $i \neq \sigma(i)$ we have
$\langle \chal'_i,\chal'_{\sigma(i)} \rangle = 0$, so no singularities
occur in equations \eqref{BAEgaudin}. Equivalently, this condition may
be expressed as $Q_i(z) = Q_{\sigma(i)}(z)$ for all $i \in I$, so this
is the Gaudin model analogue of the condition we used to define the
folding of the Bethe Ansatz equations in the XXZ-type model associated
to $U_q(\wh{\g'})$ (see Proposition \ref{foldedBAE}).

\begin{lem}    \label{equivbethe}
  The system \eqref{BAEgaudin} of Bethe Ansatz equations of the
  $\g'$-Gaudin model together with the condition $Q_i(z) =
  Q_{\sigma(i)}(z)$ for all $i \in I'$ is equivalent to the system
  \eqref{BAEgaudin1} of Bethe Ansatz equations of the $^L\g$-Gaudin
  model.
\end{lem}

Thus, the folding of the BAE of the $\g'$-Gaudin model gives the BAE
of the $^L\g$-Gaudin model. In particular, we obtain that the $q \to
1$ limit of the system \eqref{bae gen1} of folded Bethe Ansatz
equations for $U_q(\wh{\g'})$ coincides with the Bethe Ansatz
equations of the $^L\g$-Gaudin model.

Now we define an analogue of the space $W(u)$ from Section
\ref{invsub}, with the role of the twist parameter $u$ played by
$\chi$ (more precisely, this is an analogue of the subspace ${\mb
  W}(u) \subset W(u)$ introduced in Remark \ref{bfW}). Recall that we
have identified the lattice of integral $^L\g$-weights (equivalently,
integral $\g$-coweights) with the lattice of $\sigma$-invariant
integral $\g'$-weights (equivalently, integral $\g'$-coweights). We
have a collection $\cmu_k, k=1,\ldots,N$, of $\g$-coweights. Let
$\mu_k, k=1,\ldots,N$, be the corresponding $^L\g$-weights and
$\mu'_k, k=1,\ldots,N$, the corresponding $\g'$-weights. Consider the
following representation of $\g'$:
\begin{equation}    \label{qs1}
W = \bigotimes_{k=1}^N V'_{\mu'_k}(z_k).
\end{equation}

\begin{defn}    \label{Wchi}
The subspace $W(\chi)$ of $W$ is the span of eigenvectors of the
$\g'$-Gaudin Hamiltonians with the twist parameter $\chi$ whose
eigenvalues correspond to solutions of the BAE \eqref{BAEgaudin}
satisfying the condition $Q_i(z) = Q_{\sigma(i)}(z), i \in I'$.
\end{defn}

Theorems \ref{gsp1} and \ref{gsp2} and Lemma \ref{equivbethe} imply
the following.

\begin{thm}
Suppose that the $^L\g$-Gaudin Hamiltonians are diagonalizable on the
tensor product \eqref{quantum space} and all $\g$-opers in
$\on{Op}^\chi_{(\cmu_i),(z_i)}(\g)$ are non-degenerate, so they
correspond to non-degenerate Miura $\g$-opers of the form
\eqref{Miura2} or equivalently, solutions of the corresponding Bethe
Ansatz equations. Then there is a natural embedding of the
corresponding set of solutions of the BAE of this $^L\g$-Gaudin model
into the spectrum of the $\g'$-Gaudin model with the spaces of states
\eqref{qs1}. Moreover, we obtain an embedding
\begin{equation}    \label{emb1}
\bigotimes_{k=1}^N V_{\mu_k}(z_k) \overset{\sim}\longrightarrow
W(\chi) \subset W
\end{equation}
where $W$ is the tensor product \eqref{qs1} of representations of
$\g'$ and $W(\chi)$ is given in Definition \ref{Wchi}.
\end{thm}

We can construct the embedding \eqref{emb1} explicitly using
the formulas given in \cite{FFTL} for the eigenvectors (Bethe vectors)
of the $^L\g$-Gaudin model associated to solutions of the
corresponding BAE. In the next subsection we will explain this in the
case $N=1$.

\subsection{The case $N=1$}    \label{N=1}

If $N=1$, the space \eqref{quantum space} is an irreducible
representation $V_\mu$ of $^L\g$ with highest weight $\mu \in
({}^L\h)^*$ and \eqref{qs1} is the irreducible representation
$V'_{\mu'}$ of $\g'$ with the $\sigma$-invariant $\g'$-highest weight
$\mu'$ corresponding to $\mu$. As before, we denote by $\cmu$ and
$\cmu'$ the corresponding coweights of $\g$ and $\g'$. Let $\chi$ be a
regular element of the Cartan subalgebra of $\g$. As before, we will
denote by $\chi'$ its image in $\h'$, and we will assume that it is a
regular element of $\h'$ as well.

The Gaudin model for general $N$ is invariant under simultaneous
shifts of the spectral parameters $z_i$. Using this symmetry, we will
set the only parameter $z_1$ in the case $N=1$ equal to $0$. The
algebra of Gaudin Hamiltonians in this case is known as the {\em shift
  of argument subalgebra} ${\mc A}_\chi$ of the universal enveloping
algebra $U(\g)$ (see \cite{Ryb,FFTL,FFRyb}).

The corresponding non-degenerate Miura $\g$-opers \eqref{Miura2} have
the form
\begin{equation}    \label{Miura3}
\pa_z + \sum_{i \in I} f_i - \frac{\cmu}{z} + 
\sum_{j=1}^m \frac{\chal_{i_j}}{z-w_j} + \chi.
\end{equation}
This Miura $\g$-oper has trivial monodromy if and only if the
following BAE are satisfied (see formula (6.16) of \cite{FFTL}):
\begin{equation}    \label{one mod}
\frac{\langle \al_{i_j},\cmu \rangle}{w_j} -
\sum_{s\neq j} \frac{\langle \al_{i_j},\chal_{i_s} \rangle}{w_j-w_s} =
\langle \al_{i_j},\chi \rangle, \quad j=1,\ldots,m.
\end{equation}
Suppose that these equations are satisfied. Then, as shown in
\cite{FFTL}, the following Bethe vector is an eigenvector of the
Gaudin Hamiltonians (provided that it is non-zero) of weight $\ga =
\mu - \sum_{j=1}^m \al_{i_j}$:
\begin{multline}    \label{bethe for one mod}
\phi(w_1^{i_1},\ldots,w_m^{i_m}) = \\ \sum_{\tau \in S_m}
\frac{f_{i_{\tau(1)}} f_{i_{\tau(2)}} \ldots
f_{i_{\tau(m)}}}{(w_{{\tau(1)}} -
w_{{\tau(2)}})(w_{{\tau(2)}} - w_{{\tau(3)}}) \ldots
(w_{{\tau(m-1)}}-w_{{\tau(m)}}) w_{{\tau(m)}}} v_{\mu},
\end{multline}
where the sum is over all permutations $\tau$ on $m$ letters (see
formula (6.15) of \cite{FFTL}).

According to Theorem \ref{gsp2}, for generic $\chi$ the spectrum of
the algebra ${\mc A}_\chi$ of Gaudin Hamiltonians on $V_\mu$ is simple
and in bijection with the set $\on{Op}^\chi_{\cmu,0}(\g)$ of
$\g$-opers on $\pone$ with trivial monodromy, regular singularity at
$0$ with the residue $\cmu$, and irregular singularity of order $2$ at
$\infty$ with the $2$-residue $\chi$.

\begin{conj}    \label{completeness}
For generic regular $\chi \in \h$, all $\g$-opers appearing in the
spectrum of this Gaudin model are non-degenerate, and so the spectrum
is parametrized by Miura $\g$-opers \eqref{Miura3} with trivial
monodromy, or equivalently, by collections $\{ w_1,\ldots,w_m \}$ and
$\{ \chal_{i_1},\ldots,\chal_{i_m} \}$ solving the system of BAE
\eqref{one mod}. Moreover, the corresponding Bethe vectors
\eqref{bethe for one mod} form an eigenbasis of ${\mc A}_\chi$ in
the irreducible representation $V_\mu$ of $\g$.
\end{conj}

Assuming this conjecture, we can construct explicitly an embedding of
$V_\mu$ into the irreducible representation $V'_{\mu'}$ of
$\g'$. Namely, suppose that we have collections $\{ w_1,\ldots,w_m \}$
and $\{ \chal_{i_1},\ldots,\chal_{i_m} \}$ satisfying the BAE
\eqref{one mod} of the $^L\g$-Gaudin model. Then we have the
corresponding non-degenerate Miura $\g$-oper \eqref{Miura3} with
trivial monodromy, to which we associate a Miura $\g'$-oper with
trivial monodromy via Proposition \ref{embopers1}. It has the form
\eqref{Miura-2}:
\begin{equation}    \label{Miura-3}
\pa_z + \sum_{i \in I'} f_i - \frac{\cmu'}{z} + 
\sum_{\sigma(i_j) = i_j} \frac{\chal'_{i_j}}{z-w_j} + \sum_{\sigma(i_j)
  \neq i_j} \frac{\chal'_{i_j}+\chal'_{\sigma(i_j)}}{z-w_j} + \chi'.
\end{equation}

We wish to associate to this Miura $\g'$-oper an eigenvector of the
$\g'$-Gaudin Hamiltonians in $V'_{\mu'}$. There are two ways to do it,
which give the same result. The first is to apply the construction of
\cite[Sect. 6.3]{FFTL} but insert at every point $w_j$ such that
$\sigma(i_j) \neq i_j$ the vector $e_{i_j,-1}^R
e_{\sigma(i_j),-1}^R|0\rangle$ in the corresponding Wakimoto module,
rather than $e_{i_j,-1}^R|0\rangle$ (see formula (6.4) in
\cite{FFTL}). The second way (which is more direct) is to apply
formula \eqref{bethe for one mod} directly to the $\g'$-oper
\eqref{Miura-3}. The problem is that if $i_j$ is such that
$\sigma(i_j) \neq i_j$, then the corresponding $w_j$ appears twice
giving rise to seeming singularities in \eqref{bethe for one
  mod}. However, since for $\sigma(i) \neq i$ we always have
$[f_i,f_{\sigma(i)}] = 0$, it is easy to see that these singularities
cancel out, so we do obtain a well-defined vector
$\phi'(w_1^{i_1},\ldots,w_m^{i_m})$ in $V'_{\mu'}$. It then follows
that it is an eigenvector of the $\g'$-Gaudin Hamiltonians with the
eigenvalue corresponding to the $\g'$-oper in
$\on{Op}^{\chi'}_{\cmu',0}(\g')$ represented by
\eqref{Miura-3}. Moreover, it is a weight vector of the same weight as
$\phi(w_1^{i_1},\ldots,w_m^{i_m})$; namely, $\ga = \mu - \sum_{j=1}^m
\al_{i_j}$ (here we identify $^L\g$-weights with $\sigma$-invariant
$\g'$-weights).

Now recall that in Section \ref{olW} we defined the invariant subspace
$(V'_{\mu'})^\sigma$ of $V'_{\mu'}$ as the direct sum of the subspaces
of $\wh\sigma$-invariant vectors in all weights subspaces of
$V'_{\mu'}$ corresponding to $\sigma$-invariant weights. Using the
explicit formula for $\phi'(w_1^{i_1},\ldots,w_m^{i_m})$, we obtain
that it belongs to $(V'_{\mu'})^\sigma$. We summarize this as follows.

\begin{lem}
  The vector $\phi'(w_1^{i_1},\ldots,w_m^{i_m}) \in V'_{\mu'}$ is a
  well-defined eigenvector of the $\g'$-Gaudin Hamiltonians
  corresponding to the $\g'$-oper in $\on{Op}^{\chi'}_{\cmu',0}(\g')$
  which is the image of the $\g$-oper in $\on{Op}^\chi_{\cmu,0}(\g)$
  represented by \eqref{Miura3} under the embedding
  \eqref{eopers}. Moreover, it is $\sigma$-invariant, so it belongs to
  $(V'_{\mu'})^\sigma$.
\end{lem}

This implies the following statement.

\begin{thm}     \label{embbethe}
  Suppose that $\chi$ is generic, so that Conjecture
  \ref{completeness} holds. Assume also that the vectors
  $\phi'(w_1^{i_1},\ldots,w_m^{i_m})$ are non-zero in
  $V'_{\mu'}$. Then the map $V_\mu \to V'_{\mu'}$ sending
  $\phi(w_1^{i_1},\ldots,w_m^{i_m}) \in V_\mu$ to
  $\phi'(w_1^{i_1},\ldots,w_m^{i_m}) \in V'_{\mu'}$ defines an
  embedding of $V_\mu$ into $V'_{\mu'}$ whose image is the subspace
  $V'_{\mu'}(\chi)$ (see Definition \ref{Wchi}). Moreover,
  $V'_{\mu'}(\chi)$ is contained in $(V'_{\mu'})^\sigma$.
\end{thm}

We view this result as an analogue of our main Conjecture \ref{main
  conj bis}.

\subsection{Example}

Let $V'_{\al'_{\on{max}}}$ be the adjoint representation of $\g'$. Its
highest weight is the maximal root $\al'_{\on{max}}$ of $\g'$, which
is $\sigma$-invariant. It is easy to describe the automorphism
$\wh\sigma: \g' \to \g'$ from Lemma \ref{whsigma} (adapted to simple
Lie algebras) in this case.

\begin{lem}
The automorphism $\wh\sigma: V'_{\al'_{\on{max}}} \to
V'_{\al'_{\on{max}}}$ is equal to $\sigma^{-1}$.
\end{lem}

\begin{proof}
For $g \in \g'$, we have $\rho(g) \cdot x = [g,x], \forall x \in
W=\g'$. Hence $\rho_\sigma(g) \cdot x = [\sigma(g),x]$. The operator
$\wh\sigma$ is uniquely defined by the equation
\begin{equation}    \label{rhosigma1}
\wh\sigma \rho(g) \wh\sigma^{-1} = \rho(\sigma(g)), \qquad \forall
g\in \g'
\end{equation}
and the condition that the restriction of $\wh\sigma$ to
the highest weight subspace of $W$ is the identity. We claim that
$\wh\sigma = \sigma^{-1}$ satisfies these conditions. Indeed, it acts
as the identity on the highest weight subspace of $W=\g'$, and the
equation \eqref{rhosigma1} is equivalent to
$$
\wh\sigma \cdot [\sigma(g),\wh\sigma^{-1}(x)] = [g,x], \forall x \in
\g'.
$$
Setting $\wh\sigma = \sigma^{-1}$, we obtain
$$
[\sigma(g),\sigma(x)] = \sigma([g,x]),
$$
which follows from the fact that $\sigma$ is an automorphism of
$\g'$.
\end{proof}

Now let $\g' = A_{2n-1} = \sw_{2n}$ with its order two automorphism
$\sigma$. Then $\g = C_n = sp_{2n}$ and $^L\g = B_n = so_{2n+1}$.  The
highest weight $\al'_{\on{max}}$ of the adjoint representation of
$\g'$ is equal to $\omega'_1+\omega'_{2n-1}$ in this case. By the
above lemma, $\wh\sigma = \sigma^{-1} = \sigma$. Therefore, we have
following description of the invariant subspace
$(V'_{\omega'_1+\omega'_{2n-1}})^\sigma$ (the direct sum of the
subspaces of $\wh\sigma$-invariant vectors in all weights subspaces of
$V'_{\omega'_1+\omega'_{2n-1}}$ corresponding to $\sigma$-invariant
weights).

\begin{lem}
  $(V'_{\omega'_1+\omega'_{2n-1}})^\sigma$ is the direct sum of the root
  subspaces of $V'_{\omega'_1+\omega'_{2n-1}}$ corresponding to the
  roots
  $$
  \pm \al_n \qquad \on{and} \qquad \pm (\al_i+\al_{i+1}+ \ldots +
  \al_{2n-i-1}+\al_{2n-i}), \qquad i=1,\ldots,n-1,
  $$
as well as the $n$-dimensional invariant subspace of the zero weight
subspace of $W$, which is the Cartan subalgebra of $\g=C_n$ inside the
Cartan subalgebra of $\g'$.
\end{lem}

Expressing $\sigma$-invariant $\g'$-weights as $^L\g$-weights, we
obtain:

\begin{lem}
  $(V'_{\omega'_1+\omega'_{2n-1}})^\sigma$ is isomorphic, as a
  vector space graded by $^L\g$-weights, to the direct sum of the
  first fundamental representation $V_{\omega_1} = \C^{2n+1}$ of
  $^L\g=so_{2n+1}$ and $(n-1)$ copies of the trivial representation of
  $^L\g$.
\end{lem}

Assuming Conjecture \ref{completeness}, we obtain that we can embed
$V_{\omega_1}$ into $V'_{\omega'_1+\omega'_{2n-1}}$ in such a way that
the image is the subspace $V'_{\omega'_1+\omega'_{2n-1}}(\chi)$ (the
span of eigenvectors of the $\g'$-Gaudin Hamiltonians for which
$Q_i(z) = Q_{\sigma(i)}(z), i \in I'$). The latter is the analogue of
the subspace $W(u)$ from Section \ref{invsub} which we used to define
the folded integrable model. Moreover, using Proposition
\ref{embbethe} we can construct this embedding explicitly, using the
Bethe vectors \eqref{bethe for one mod}, and this shows that
$V'_{\omega'_1+\omega'_{2n-1}}(\chi)$ is contained in
$(V'_{\omega'_1+\omega'_{2n-1}})^\sigma$. The weight subspaces of
these two vector spaces coincide for all non-zero weights, but their
weight $0$ subspaces are different: the former is $1$-dimensional and
the latter is $n$-dimensional.

\begin{rem}
  At first glance, it may appear that this result is in contradiction
  with the results of Section \ref{ex2}, where we considered a similar
  example in the case of the quantum affine algebras associated to
  $\g'=A_{2n-1}$ and $\g=C_n$. Namely, we showed that if we take as
  $W$ the irreducible representation $L(Y_{1,1}Y_{2n-1,1})$ of
  $U_q(A_{2n-1}^{(1)})$, then its subspace $W(u)$ defined in Section
  \ref{olW} is $(2n+2)$-dimensional for generic $u$. This means that
  the corresponding subspace in the limit $q \to 1$ should have
  dimension at least $2n+2$. And yet, we have found that in the case
  of the adjoint representation of $A_{2n-1}$ the analogous subspace
  $V'_{\omega'_1+\omega'_{2n-1}}(\chi)$ (which is the image of
  $V_{\omega_1}$ in $V'_{\omega'_1+\omega'_{2n-1}}$) has dimension
  $2n+1$ for generic $\chi$.

The explanation is that the evaluation representation of
$U_q(A_{2n-1}^{(1)})$ corresponding to the adjoint representation of
$A_{n-1}$ (of dimension $4n^2-1$) is not $L(Y_{1,1}Y_{2n-1,1})$ but
$L(Y_{1,1}Y_{2n-1,q^{2n}})$. The highest monomial of
$L(Y_{1,1}Y_{2n-1,q^{2n}})$ is not $\sigma$-invariant and therefore we
cannot define the operator $\wh\sigma$ and related structures on
$L(Y_{1,1}Y_{2n-1,q^{2n}})$. On the other hand, the highest monomial
of $L(Y_{1,1}Y_{2n-1,1})$ is $\sigma$-invariant, but the dimension of
this representation is $4n^2$, i.e. it is greater than the dimension
of the adjoint representation $V'_{\omega'_1+\omega'_{2n-1}}$ of
$A_{n-1}$ by $1$. Hence it is not surprising that the subspace $W(u)$
of $W=L(Y_{1,1}Y_{2n-1,1})$ has dimension greater by $1$ than the
dimension of $V'_{\omega'_1+\omega'_{2n-1}}(\chi)$. In fact, as shown
in Section \ref{ex2}, the weight 0 subspace of $W(u)$ is
two-dimensional and can be identified with the weight 0 subspace of a
$4$-dimensional irreducible representation of the $U_q(\wh{\sw}_2)$
subalgebra of $U_q(A_{2n-1}^{(1)})$ corresponding to node $n$ of the
Dynkin diagram of $A_{2n-1}$. But the corresponding irreducible
representation of the $\sw_2$ Lie subalgebra corresponding to node $n$
is $3$-dimensional and so its weight $0$ subspace is
$1$-dimensional. This is the weight $0$ subspace of
$V'_{\omega'_1+\omega'_{2n-1}}(\chi)$.

As the result, $W(u)$ is actually isomorphic, as a vector space graded
by weights of $^L\g = B_n$, to an irreducible representation of
$U_q({}^L\ghat)$, where $^L\ghat = D^{(2)}_{n+1}$, whereas
$V'_{\omega'_1+\omega'_{2n-1}}(\chi)$ is isomorphic to an irreducible
representation of $^L\g = B_n$.\qed
\end{rem}

\begin{rem}    \label{cyclic}
The above example gives us a nice illustration of how the {\em affine}
Langlands duality of the folded integrable model becomes in the limit
$q \to 1$ the finite-dimensional Langlands duality of the Gaudin
model. The point is that when we take the $q \to 1$ limit, we restrict
finite-dimensional representations of $U_q(\ghat)$ to $U_q(\g)$ and
then take the limit $q \to 1$ (the latter is an equivalence of
categories if $q$ is generic). The smallest finite-dimensional
representation $W$ of $U_q(\ghat)$ with $\sigma$-invariant highest
monomial that has $\mu$ as the highest $^L\g$-weight is usually
isomorphic, as a $U_q(\g)$-module, to the direct sum of $V_\mu$ and
many other irreducible representations.\footnote{In the case of
$\g=\sw_n$, every finite-dimensional irreducible representation of
$U_q(\sw_n)$ can be lifted to $U_q(\wh\sw_n)$, but the highest
monomial of the resulting representation of $U_q(\wh\sw_n)$ is not
$\sigma$-invariant, see the above example. The representations with
$\sigma$-invariant monomials are generally much larger.} So it is not
surprising that the intersection of an irreducible representation
$M(W)$ of $U_q({}^L\ghat)$, which is embedded into $W$ according to
Conjecture \ref{main conj bis},(2), with $V_\mu$ would be an
irreducible representation of $U_q({}^L\g)$.

This is closely related to another seeming paradox. According to a
result of \cite{FFRyb} that we used in this section, every eigenspace
in the space of states of the Gaudin model is one-dimensional (if the
Gaudin Hamiltonians are not diagonalizable, this means that there can
be at most one Jordan block for every joint eigenvalue). In
particular, for every irreducible representation $V_\mu$ of $^L\g$,
the joint eigenspace of the algebra ${\mc A}_\chi$ is one-dimensional
for all joint eigenvalues. On the other hand, we expect that the
analogous spaces in the folded integrable model for generic $q$ can
have dimensions greater than one. The explanation is that components
other than $V_\mu$ in the decomposition of $W$ as a representation of
$U_q(\g)$ may well contain (in the limit $q \to 1$) additional
eigenvectors of ${\mc A}_\chi$.\qed
\end{rem}

\subsection{Conclusion}

It follows that, somewhat surprisingly, we can realize the
$^L\g$-Gaudin model inside the $\g'$-Gaudin model by embedding the
space of states \eqref{quantum space} of the former into the space of
states \eqref{qs1} of the latter in such a way that on the dual side
this corresponds to the embedding of the Miura $\g$-opers
\eqref{Miura2} into the Miura $\g'$-opers \eqref{Miura}. This is an
intriguing consequence of the Langlands duality in the Gaudin models
discovered in \cite{FFR} and the results of \cite{FFRyb}, which
deserves further study.

The embedding of the Miura $\g$-opers \eqref{Miura2} into the Miura
$\g'$-opers \eqref{Miura} means that the Baxter polynomials $Q_i(z)$
and $Q_{\sigma(i)}(z)$ (see formula \eqref{Qr}) associated to a
$\g'$-oper obtained this way will be equal. Thus, the folding of the
$\g'$-Gaudin model is equivalent to the $^L\g$-Gaudin model. This
statement may be viewed as a $q \to 1$ limit of Conjecture \ref{main
  conj bis}, so it provides support for this conjecture.

\section{Appendix: Toward constructing a folded version of the qKZ equations
  for non-simply laced Lie algebras} \label{qKZ sect}

The qKZ equations \cite{IFR} for $U_q(\ghat)$ (with the twist parameter
$u$, which is an element of the Cartan subgroup $H$ as above) can be
written in the form
\begin{multline}    \label{qKZ}
  \Psi(z_1,\ldots,pz_j,\ldots,z_N) =
  R^{V_jV_{j-1}}_{j,j-1}(pz_j/z_{j-1}) \ldots
  R^{V_jV_1}_{j,1}(pz_j/z_1) \cdot u_j \cdot \\
  R^{V_jV_N}_{j,N}(z_j/z_N) \ldots
  R^{V_jV_{j+1}}_{j,j+1}(z_j/z_{j+1})
  \Psi(z_1,\ldots,z_j,\ldots,z_N).
\end{multline}
Here each $V_i$ denotes an irreducible finite-dimensional
representation of $U_q(\ghat)$, $V_i(z_i)$ is its shift by the
spectral parameter $z_i$, and $R^{V_jV_i}_{j,i}(z_j/z_i)$ is the
$R$-matrix acting on $V_j(z_j) \otimes V_i(z_i)$ normalized so that it
acts as the identity on $v_j \otimes v_i$, where $v_j$ and $v_i$
are highest weight vectors in $V_j(z_j)$ and $V_i(z_i)$,
respectively. Also, $u_j$ denotes $u|_{V_j}$.

Let us denote the operator on the right hand side of the $j$th
equation by $K_j(p)$.

The critical level limit corresponds to $p \to 1$. In this limit we
have
\begin{equation}    \label{Kj}
  K_j(1) = R^{V_jV_{j-1}}_{j,j-1}(z_j/z_{j-1}) \ldots
  R^{V_jV_1}_{j,1}(z_j/z_1) \cdot u_j \cdot
  R^{V_jV_N}_{j,N}(z_j/z_N) \ldots
  R^{V_jV_{j+1}}_{j,j+1}(z_j/z_{j+1}).
\end{equation}

Recall that for an auxiliary representation $V$ of $U_q(\ghat)$ we
have the transfer matrix $t_V(z,u)$. These transfer-matrices commute
with each other for a fixed $u$ and different $V$ and $z$. In what
follows, we will use the same notation $t_V(z,u)$ for the action of
the transfer-matrix $t_V(z,u)$ on the tensor product $V_N(z_N) \otimes
\ldots \otimes V_1(z_1)$. Thus,
\begin{equation}    \label{tvzu}
  t_V(z,u) = \on{Tr}_a(u_a R^{VV_N}_{a,N}(z/z_N) \ldots
  R^{VV_1}_{a,1}(z/z_1)),
\end{equation}
where the subscript $a$ indicates the auxiliary representation $V(z)$.

\begin{prop}    \label{qKZcrit}
  Suppose that $V_j$ is such that up to a scalar, $R^{V_jV_j}(1) = P$,
  the permutation operator on $V_j(z) \otimes V_j(z)$ sending $x
  \otimes y$ to $y \otimes x$. Then $K_j(1) = t_{V_j}(z_j,u)$.
\end{prop}

\begin{proof}
  If $R^{V_jV_j}(1) = P$, then formula \eqref{tvzu} with $V(z) =
  V_j(z_j)$ becomes
\begin{multline}    \label{tvzu1}
  t_V(z,u) = \\ \on{Tr}_a(u_a R^{V_jV_N}_{a,N}(z_j/z_N) \ldots
  R^{V_jV_{j+1}}_{a,j+1}(z_j/z_{j+1}) P_{aj}
    R^{V_jV_{j-1}}_{a,j-1}(z_j/z_{j-1}) \ldots
  R^{V_jV_1}_{a,1}(z_j/z_1)).
\end{multline}
Using the identify $A_a P_{aj} = P_{aj} A_j$, we rewrite the RHS of
\eqref{tvzu1} as
\begin{equation}    \label{tvzu2}
  \on{Tr}_a(P_{aj} u_j R^{V_jV_N}_{j,N}(z_j/z_N) \ldots
  R^{V_jV_{j+1}}_{j,j+1}(z_j/z_{j+1})
    R^{V_jV_{j-1}}_{a,j-1}(z_j/z_{j-1}) \ldots
  R^{V_jV_1}_{a,1}(z_j/z_1)).
\end{equation}
Next, using the cyclic property of the trace, we rewrite
\eqref{tvzu2} as
\begin{equation}    \label{tvzu3}
  \on{Tr}_a(R^{V_jV_{j-1}}_{a,j-1}(z_j/z_{j-1}) \ldots
  R^{V_jV_1}_{a,1}(z_j/z_1) P_{aj} u_j R^{V_jV_N}_{j,N}(z_j/z_N) \ldots
  R^{V_jV_{j+1}}_{a=j,j+1}(z_j/z_{j+1})).
\end{equation}
Using formula $A_a P_{aj} = P_{aj} A_j$ again, we rewrite
\eqref{tvzu3} as
\begin{equation}    \label{tvzu4}
  \on{Tr}_a(P_{aj} R^{V_jV_{j-1}}_{j,j-1}(z_j/z_{j-1}) \ldots
  R^{V_jV_1}_{j,1}(z_j/z_1) u_j R^{V_jV_N}_{j,N}(z_j/z_N) \ldots
  R^{V_jV_{j+1}}_{a=j,j+1}(z_j/z_{j+1})).
\end{equation}
In the last formula, the only operator depending on the auxiliary
space is $P_{aj}$ and its trace over the auxiliary space is the
identity operator on $V_j(z_j)$. Hence \eqref{tvzu4} is equal to
$K_j(1)$.
\end{proof}

Next, we discuss under what conditions $R^{VV}(1) = P$.

\begin{lem}    \label{real}
  Let $V$ be an irreducible representation of $U_q(\ghat)$ such that $V
  \otimes V$ is also irreducible. Then $R^{VV}(1) = P$.
\end{lem}

\begin{proof}
  It follows from the definition that $P \circ R^{VV}(1)$ is an
  intertwining operator $V \otimes V \to V \otimes V$. If $V \otimes
  V$ is irreducible, then by Schur's lemma it is a scalar
  operator. Under our normalization, it then has to be the identity
  operator, and hence $R^{VV}(1) = P$.
\end{proof}

Note that a representation $V$ satisfying the condition of Lemma
\ref{real} is called {\em real} in \cite{HL}. Not all irreducible
representations of $U_q(\ghat)$ are real, as shown in \cite{L}, see
\cite[Sect. 13.6]{HL}. However, there is a large class of real
representations: Kirillov--Reshetikhin modules.

\begin{prop}
  Let $V$ be any Kirillov--Reshetikhin module over $U_q(\ghat)$. Then
  $V \otimes V$ is irreducible.
\end{prop}

\begin{proof} There are several possible arguments. 
It follows from \cite{C} that $V\otimes V$ is cyclic and cocyclic, and
hence is irreducible.  The statement also follows from the fact that
the square of the $q$-character of a Kirillov--Reshetikhin module has
a unique dominant monomial; namely, its highest monomial. The latter
follows from the results of \cite[Sect. 3.2.2]{HL1}. In addition, the
statement has been established by a different method in \cite{KKOP}.
\end{proof}

Note however that even if $V \otimes V$ is reducible, we may still
have $R^{VV}(1) = P$. It would be interesting to describe all
irreducible representations $V$ of $U_q(\ghat)$ satisfying the
condition $R^{VV}(1) = P$.

In any case, the above results readily imply that in the case when all
$V_j$'s are Kirillov--Reshetikhin modules, the operators $K_j(1),
j=1,\ldots,N$, are commuting Hamiltonians of the XXZ-type integrable
model associated to $U_q(\ghat)$. It is in this sense that we say that
one recovers this integrable model in the critical level limit of the
qKZ equations.

Similarly, under the above condition on $V_j$, the operator
$K_j(p)$ is the transfer-matrix $t_{V_j}(z_j,u)$ acting on
$$
V_N(z_N) \otimes \ldots \otimes V_j(z_j) \otimes V_{j-1}(z_{j-1}p^{-1})
\otimes V_1(z_1p^{-1}).
$$
However, because of the shifts by $p^{-1}$ these operators do not
commute with each other if $p \neq 1$.

Now suppose that $\g$ is a non-simply laced simple Lie algebra. We
would like to construct a ``folded qKZ system'' such that in the
critical level limit the operators on the right hand side become the
Hamiltonians of the {\em folded quantum integrable model} described in
this paper. This means, in particular, that they should correspond to
transfer-matrices of $U_q(\wh{\g'})$ rather than $U_q(\ghat)$. Thus,
each $V_i(z_i)$ should be a representation of $U_q(\wh{\g'})$.

Unfortunately, naive attempts to construct this folded qKZ system
appear to fail:

(1) For each $V_i(z_i)$, we have its subspace $(V_i(z_i))(u)$ defined
as above and we can take the tensor product of these subspaces,
$\otimes_{i=1}^N (V_i(z_i))(u)$. However, it is not clear why this
subspace would be preserved by the operators $K_j$ or their
$p$-deformed versions $K_j(p)$ appearing on the RHS of \eqref{qKZ}.

(2) We take the subspace $V(u)$ of the entire tensor product
$V=\otimes_{i=1}^N V_i(z_i)$. According to Conjecture \ref{main conj
  bis},(ii), it contains a subspace isomorphic to a representation
$M(V)$ of $U_q(^L\wh{\g})$. Moreover, since the algebra of
transfer-matrices of $U_q(\wh{\g'})$ commutes with the Baxter
operators $Q_j(z), j \in I'$, it follows that all transfer-matrices of
$U_q(\wh{\g'})$ preserve this subspace $V(u)$.  In particular, the
operators $K_j$ given by formula \eqref{Kj}, being the
transfer-matrices of $U_q(\wh{\g'})$, preserve $V(u)$. But the problem
is that on the right hand side of these qKZ equations we have the
operators $K_j(p)$ with $p\neq 1$, which are the transfer matrices
acting on the tensor product of representations in some of which
(namely, the ones with $i=1,\ldots,j-1$) there is a multiplicative
shift in the spectral parameter by $p^{-1}$. It is not clear why these
operators should preserve $V(u)$ (where $V$ is the tensor product of
the representations $V_i(z_i)$ without any shift by $p^{-1}$).

Hence at the moment it is unclear to us how to fold the qKZ equations
for $U_q(\wh{\g'})$ in such a way that in the critical level limit we
recover the commuting Hamiltonians of the folded quantum integrable
model associated to $U_q(\ghat)$.


\begin{thebibliography}{KKKO}

\bi[AFO]{AFO} M. Aganagic, E. Frenkel, and A. Okounkov, {\em Quantum
  q-Langlands Correspondence}, Trans. Moscow Math. Soc. {\bf 79}
(2018) 1--83.

\bibitem[BD2]{BD:opers} A.~Beilinson and V.~Drinfeld, {\em Opers},
  arXiv:math/0501398.

\bi[B]{Bethe} H. Bethe, {\em Zur Theorie der Metalle I. Eigenwerte und
  Eigenfunktionen der linearen Atomkette}, Z. Phys. {\bf 71} (1931)
205--226.

\bibitem[BLZ]{BLZ3} V.V. Bazhanov, S.L. Lukyanov and A.B. Zamolodchikov, 
{\em Integrable structure of conformal field
theory. III. The Yang-Baxter Relation}, Comm. Math. Phys. {\bf 200}
(1999), 297--324.

\bibitem[C]{C} V. Chari, {\em Braid group actions and tensor
  products}, IMRN (2002) 357--382.

\bibitem[CP]{cp} V. Chari and A. Pressley, {\it A Guide to Quantum
  Groups}, {Cambridge University Press, Cambridge, 1994}.

\bi[CK]{CK} H.-Y. Chen and T. Kimura, {\em Quantum integrability from
  non-simply laced quiver gauge theory}, JHEP {\bf 2018} (2018) 165.

\bi[DHKM]{DHKM} A. Dey, A. Hanany, P. Koroteev, and N. Mekareeya, {\em
  On Three-Dimensional Quiver Gauge Theories of Type B}, JHEP {\bf 09}
(2017) 067.

\bi[D]{Dr} V.G. Drinfeld, {\em Quantum groups}, in Proceedings of the
International Congress of Mathematicians, pp. 798–820, AMS,
1987.

\bi[EP]{EP} C. Elliott and V. Pestun, {\em Multiplicative Hitchin
  Systems and Supersymmetric Gauge Theory}, Selecta Math. {\bf 25}
  (2018) 64.

\bi[FF]{FF:kdv} B. Feigin and E. Frenkel, {\em Quantization of soliton
  systems and Langlands duality}, in {\em Exploration of New
  Structures and Natural Constructions in Mathematical Physics},
pp. 185--274, Adv. Stud. Pure Math. 61, Math. Soc. Japan, Tokyo, 2011
(arXiv:0705.2486).
  
\bi[FFR]{FFR} B. Feigin, E. Frenkel, and N. Reshetikhin, {\em Gaudin model,
  Bethe Ansatz and critical level}, Comm. Math. Phys. {\bf 166} (1994)
27--62.

\bi[FFRy]{FFRyb} B. Feigin, E. Frenkel, and L. Rybnikov, {\em Opers
  with irregular singularity and spectra of the shift of argument
  subalgebra}, Duke Math. Journal {\bf 155} (2010) 337--363.

\bi[FFT]{FFTL} B. Feigin, E. Frenkel, and V. Toledano Laredo, {\em
  Gaudin models with irregular singularities}, Advances in Math. {\bf
  223} (2010) 873--948.

\bibitem[FJMM]{FJMM} B. Feigin, M. Jimbo, T. Miwa and E. Mukhin,
  Finite type modules and Bethe Ansatz equations, Ann. Henri Poincar\'e
  {\bf 18} (2017) 2543--2579.

\bi[F]{F:gaudin} E. Frenkel, {\em Gaudin model and opers}, in Infinite
Dimensional Algebras and Quantum Integrable Systems, eds. P. Kulish,
e.a., Progress in Math. {\bf 237}, pp. 1--60, Birkh\"auser, 2005.

\bibitem[FH1]{FH0} E. Frenkel and D. Hernandez, 
{\em Langlands duality for representations of quantum groups}, 
Math. Ann. {\bf 349} (2011), no. 3, 705--746.

\bibitem[FH2]{FH1} E. Frenkel and D. Hernandez, {\em Langlands duality
  for finite-dimensional representations of quantum affine algebras},
  Lett. Math. Phys. {\bf 96} (2011), no. 1-3,
  217--261 (arXiv:0902.0447).

\bibitem[FH3]{FH} E. Frenkel and D. Hernandez, {\em Baxter's Relations
    and Spectra of Quantum Integrable Models}, Duke Math. J. {\bf 164}
  (2015) 2407--2460 (arXiv:1308.3444).
	
\bibitem[FH4]{FH4} E. Frenkel and D. Hernandez, {\em Spectra of
  quantum KdV Hamiltonians, Langlands duality, and affine opers},
  Comm. Math. Phys. {\bf 362} (2018), no. 2, 361--414.

\bibitem[FKSZ]{FKSZ} E. Frenkel, P. Koroteev, D.S. Sage, and
  A.M. Zeitlin, {\em $q$-Opers, $QQ$-Systems, and Bethe Ansatz},
arXiv:2002.07344, to appear in the Journal of European Math. Soc.

\bibitem[FM]{fm} {E. Frenkel and E. Mukhin}, {\em Combinatorics of
$q$-characters of finite-dimensional representations of quantum affine
algebras}, {Comm. Math. Phys. {\bf 216} (2001), no. 1, 23--57}.

\bi[FR1]{FR:center} E. Frenkel and N. Reshetikhin, {\em Quantum affine
  algebras and deformations of the Virasoro and ${\mathcal
    W}$--algebras}, Comm. Math. Phys. {\bf 178} (1996) 237--264.

\bibitem[FR2]{FR:w} E. Frenkel and N. Reshetikhin, {\em Deformations of
    ${\mathcal W}$--algebras associated to simple Lie algebras},
  Comm. Math. Phys. {\bf 197} (1998) 1--32 (arXiv:q-alg/9708006).

\bibitem[FR3]{FR:q} E. Frenkel and N. Reshetikhin, {\em The
    $q$--characters of representations of quantum affine algebras and
    deformations of} {\em ${\mathcal W}$--algebras}, in Recent
  Developments in Quantum Affine Algebras and Related Topics, N. Jing
  and K. Misra (eds.), Contemporary Mathematics {\bf 248},
  pp. 163--205, AMS, 1999.

\bi[FRS]{FRS} E. Frenkel, N. Reshetikhin and M. Semenov-Tian-Shansky,
   {\em Drinfeld--Sokolov reduction for difference operators and
     deformations of ${\mathcal W}$--algebras I},
   Comm. Math. Phys. {\bf 192}, 605--629 (1998).

\bi[FrR]{IFR} I. Frenkel and N. Reshetikhin, {\em Quantum affine
  algebras and holonomic difference equations}, Comm. Math. Phys. {\bf
  146} (1992) 1--60.

\bibitem[H1]{H2} D. Hernandez, {\em Algebraic approach to q,t-characters}, 
Adv. Math. {\bf 187} (2004), no. 1, 1--52.

\bibitem[H2]{hcr} D. Hernandez, {\it The Kirillov--Reshetikhin
conjecture and solutions of $T$-systems}, {J. Reine Angew. Math. {\bf
596} (2006), 63--87}.

\bibitem[H3]{H} D. Hernandez, {\em Kirillov-Reshetikhin conjecture:
  the general case}, IMRN (2010) 149--193 (extended version at
  arXiv:0704.2838).

\bibitem[H4]{hinv} D. Hernandez, {\em Simple tensor products},
  Invent. Math. {\bf  181} (2010), 649--675.

\bibitem[H5]{hlast} D. Hernandez, {\em Representations of shifted
  quantum affine algebras}, Preprint arXiv:2010.06996.

\bibitem[HJ]{HJ} D. Hernandez and M. Jimbo, {\em Asymptotic
  representations and Drinfeld rational fractions},
  Compos. Math. {\bf 148} (2012), 1593--1623.

\bi[HL1]{HL} D. Hernandez and B. Leclerc, {\em Cluster algebras and
  quantum affine algebras}, Duke Math. J. {\bf 154} (2010) 265--341.

\bi[HL2]{HL1} D. Hernandez and B. Leclerc, {\em A cluster algebra
  approach to $q$-characters of Kirillov--Reshetikhin modules},
J. Eur. Math. Soc. {\bf 18} (2016) 1113-1159.

\bibitem[K1]{K8} M. Kashiwara, {\em The crystal base and Littelmann’s
  refined Demazure character formula},
Duke Math. J. {\bf 71} (1993) 839--858.

\bibitem[K2]{Ksy} M. Kashiwara, {\em Similarity of crystal bases}, 
Contemp. Math. {\bf 194} (1996), 177--186. 

\bibitem[K3]{K} M. Kashiwara, {\em Realizations of crystals}, 
in Combinatorial and geometric representation theory (Seoul, 2001),
133--139, Contemp. Math., {\bf 325} (2003).

\bi[KKOP]{KKOP} M. Kashiwara, M. Kim, S. Oh, and E. Park, {\em
  Monoidal categorification and quantum affine algebras}, Compositio
Math. {\bf 156} (2020) 1039--1077.

\bibitem[KP1]{KP1} T. Kimura and V. Pestun, {\em Quiver W-algebras},
  Lett. Math. Phys. {\bf 108} (2018) 1351--1381.

\bibitem[KP2]{KP2} T. Kimura and V. Pestun, {\em Fractional quiver
  W-algebras}, Lett. Math. Phys. {\bf 108} (2018) 2425--2451.

\bibitem[KSZ]{KSZ} P. Koroteev, D. Sage, and A. Zeitlin, {\em
  $(SL(N),q)$-opers, the $q$-Langlands correspondence, and
  quantum/classical duality}, Comm. Math. Phys. {\bf 381} (2021),
  no. 2, 641-672.

\bibitem[KZ]{KZ} P. Koroteev and A. Zeitlin, {\em $q$-Opers,
  $QQ$-systems, and Bethe Ansatz II: Generalized Minors}, Preprint
  arXiv:2108.04184.

\bibitem[L]{L} B. Leclerc, 
{\em Imaginary vectors in the dual canonical basis of $U_q(n)$}, 
Transform. Groups {\bf 8} (2003) 95--104.

\bibitem[MRV1]{MRV} D. Masoero, A. Raimondo and D. Valeri, {\em Bethe
  Ansatz and the Spectral Theory of affine Lie algebra-valued
  connections. The simply-laced case}, Commun. Math. Phys. {\bf 344}
  (2016), 719--750.

\bibitem[MRV2]{MRV2} D. Masoero, A. Raimondo and D. Valeri, 
{\em Bethe Ansatz and the Spectral Theory of affine Lie algebra-valued connections. The non simply-laced case}, 
Commun. Math. Phys. {\bf 349} (2017), 1063--1105.

\bibitem[MV1]{MV1} E. Mukhin and A. Varchenko, {\em Populations of
  solutions of the XXX Bethe equations associated to Kac--Moody
  algebras}, in Infinite-dimensional Aspects of Representation Theory
  and Applications, Contemp. Math. {\bf 392}, pp. 95--102, AMS, 2005.

\bibitem[MV2]{MV2} E. Mukhin and A. Varchenko, {\em Discrete
  Miura Opers and Solutions of the Bethe Ansatz Equations},
  Commun. Math. Phys. {\bf 256} (2005), 565--588.

\bi[MV3]{MV:bethe} E. Mukhin and A. Varchenko, {\em Multiple
  orthogonal polynomials and a counterexample to the Gaudin Bethe
  Ansatz Conjecture}, Trans. Amer. Math. Soc. {\bf 359} (2007)
5383--5418.

\bibitem[N1]{N0} H. Nakajima, {\em $t$-analogue of the $q$-characters of
  finite dimensional representations of quantum affine algebras},
  Physics and Combinatorics (2000) 196--219.

\bibitem[N2]{N} H. Nakajima, {\em $t$-analogs of $q$-characters of quantum
  affine algebras of type $A_n$, $D_n$}, in Combinatorial and geometric
  representation theory (Seoul, 2001), Contemp. Math. {\bf
    325}, pp. 141--160, AMS, 2003.

\bibitem[N3]{Nad} { H. Nakajima}, {\em $t$-analogs of $q$-characters of
Kirillov--Reshetikhin modules of quantum affine algebras},
  {Represent. Theory {\bf 7} (2003), 259--274}.

\bibitem[Ne]{Ne} N. Nekrasov, {\em BPS/CFT correspondence:
  non-perturbative Dyson–Schwinger equations and qq-characters}, JHEP
  {\bf 1603} (2016) 181.

\bibitem[OW]{OW} E. Ogievetsky and P. Wiegmann, {\em Factorized
  S-matrix and the Bethe ansatz for Simple Lie Groups}, Physics
  Letters {\bf B 168} (1986) 360--366.

\bibitem[R]{R:1987} N. Reshetikhin, {\em The spectrum of the transfer
  matrices connected with Kac--Moody algebras}, Lett. Math. Phys. {\bf
  14} (1987) 235–246.

\bi[RS]{RST} N.Yu. Reshetikhin and M.A. Semenov-Tian-Shansky, {\em
  Central extensions of quantum current groups},
Lett. Math. Phys. {\bf 19} (1990) 133--142.
  
\bibitem[RW]{RW} N. Reshetikhin and P. Weigmann, {\em Towards the
  classification of completely integrable quantum field theories (the
  Bethe-Ansatz associated with Dynkin diagrams and their
  automorphisms)}, Physics Letters {\bf B 189} (1987) 125–131.

\bi[Ry]{Ryb} L. Rybnikov, {\em Argument shift method and Gaudin
  model}, Func. Anal. Appl. {\bf 40} (2006) 188--199.

\bi[SS]{SS} M.A. Semenov-Tian-Shansky and A.V. Sevostyanov, {\em
  Drinfeld--Sokolov reduction for difference operators and
  deformations of W–algebras II. General semisimple case},
Comm. Math. Phys. {\bf 192} (1998) 631--647.

\end{thebibliography}
\end{document}